\documentclass[11pt]{article}
\usepackage{mathrsfs}
\usepackage{hyperref}
\usepackage{graphicx}
\usepackage{subfig}
\usepackage{epstopdf}
\usepackage{amsmath}
\usepackage{amsfonts}
\usepackage{amssymb, amsmath, cite}
\usepackage{setspace}
\usepackage{color}

\newtheorem{theorem}{Theorem}
\newtheorem{lemma}{Lemma}
\newtheorem{proposition}{Proposition}
\newtheorem{remark}{Remark}

\numberwithin{equation}{section}\allowdisplaybreaks
\newbox\qedbox
\newenvironment{proof}{\smallskip\noindent{\bf Proof.}\hskip \labelsep}

\newcommand{\bfR}{{\Bbb R}}
\newcommand{\bfC}{{\Bbb C}}

\newcommand{\ii}{\text{i}}
\newcommand{\e}{\text{e}}
\newcommand{\dd}{\text{d}}

\newcommand{\nn}{\nonumber}
\newcommand\be{\begin{equation}}
\newcommand\ee{\end{equation}}
\newcommand{\bea}{\begin{eqnarray}}
\newcommand{\eea}{\end{eqnarray}}
\newcommand\berr{\begin{eqnarray*}}
\newcommand\eerr{\end{eqnarray*}}
{\setlength\arraycolsep{2pt}

\begin{document}

\title{Long-time Asymptotic Behavior of the Fifth-order Modified KdV Equation in Low Regularity Spaces }
\author{{\normalsize  Nan Liu$^\dag$,\quad  Mingjuan Chen$^{\dag,}$\footnote{Corresponding author.},\quad Boling Guo$^\dag$}\\
           {\footnotesize\it $^\dag$ Institute of Applied Physics and Computational Mathematics, Beijing 100088, China}\setcounter{footnote}{-1}\footnote{E-mail addresses: ln10475@163.com (N. Liu), mjchenhappy@pku.edu.cn (M. Chen), gbl@iapcm.ac.cn (B. Guo).} \\
}
     \date{}
\maketitle

\begin{quote}
{{{\bfseries Abstract.} Based on the nonlinear steepest descent method of Deift and Zhou for oscillatory Riemann--Hilbert problems and the Dbar approach, the long-time asymptotic behavior of solutions to the fifth-order modified Korteweg-de Vries equation on the line is studied in the case of initial conditions that belong to some weighted Sobolev spaces. Using techniques in Fourier analysis and the idea of $I$-method, we give its global well-posedness in lower regularity Sobolev spaces, and then obtain the asymptotic behavior in these spaces with weights.
}

 }
\end{quote}

\section{Introduction}
\setcounter{equation}{0}
The modified KdV (Korteweg-de Vries) equation is a fundamental completely integrable model in solitary waves theory, and usually is written as
\be\label{1.1}
\partial_tu+6u^2\partial_xu+\partial_{x}^3u=0,
\ee
where $u=u(x,t)$ is a real function with evolution variable $t$ and transverse variable $x$. This equation is well-known as a canonical model for the description of nonlinear long waves in many branches of physics when there is polarity symmetry. For instance, applications in the context of electrodynamics are described in \cite{PFE}, in the context of stratified fluids in reference \cite{GPT}. Recently, a fifth-order modified KdV taking the form
\begin{equation}\label{5mKdV}
\partial_tu +\partial_x^5 u+30u^4\partial_xu-10u^2\partial_x^3u-10(\partial_xu)^3-40u\partial_xu\partial_x^2u=0,
\end{equation}
was reported in \cite{WZW,W}, which also describes nonlinear wave propagation in physical systems with polarity symmetry. Meanwhile, Equation \eqref{5mKdV} also has certain application for the description of nonlinear internal waves in a fluid stratified by both density and current \cite{GPP,PPL}.

Equation \eqref{5mKdV} is endowed with two important features. One is the complete integrability: in the sense that there are Lax pair formulations, so \eqref{5mKdV} can be solved by means of the inverse scattering transformation formalism both in the case of vanishing as well as nonvanishing boundary conditions. Being integrable, Equation \eqref{5mKdV} admits an infinite number of conservation laws and its solution exists globally in time for any Schwartz initial data. The other one is the scaling symmetry: for any solution $u(x,t)$ of \eqref{5mKdV} with initial data $u_0(x)$, the scaling functions $u_{\lambda}(x,t):=\lambda u(\lambda x, \lambda^5 t)$ also solve \eqref{5mKdV} with initial data $u_{0,\lambda}:=\lambda u_0(\lambda x)$. The Sobolev space $\dot{H}^{s_c}$ is called the critical space if it satisfies $$\|u_{0,\lambda}\|_{\dot{H}^{s_c}}\sim\|u_0\|_{\dot{H}^{s_c}}.$$
By simple calculation, one can know that $s_c=-1/2$. This scale invariance suggests ill-posedness for $H^s$, $s<-1/2$.

Using the Fourier analysis techniques and the theory of dispersive equations, one can study the low regularity theory for the fifth-order modified KdV equation. It seems natural that the regularity index of the well-posedness theory can be decreased to $-1/2$. However, it is extremely difficult to achieve this goal by the known methods in Sobolev spaces. The best local well-posedness theory so far was given by Kwon \cite{SK08}. He expressed that Equation \eqref{5mKdV} is locally well-posed in Sobolev space $H^s(\mathbb{R})$ for $s\geq 3/4$ via the $X^{s,b}$ method which was posed by Bourgain, and showed that it is ill-posed when $s<3/4$ in the sense that the data-to-solution map fails to be uniformly continuous.

In addition, the solutions to \eqref{5mKdV} enjoy mass conservation and energy conservation:
\begin{align}\label{conserv}
M(u)=\int_\mathbb{R} u^2\dd x;\quad\quad E(u)=\int_\mathbb{R} (\partial_x u)^2+u^4\dd x.
\end{align}
Energy conservation and the local well-posedness immediately yield the global well-posedness of \eqref{5mKdV} from initial data in $H^1(\mathbb{R})$. The biggest obstacle in getting global solutions in $H^s$ with $0<s<1$ is the lack of any conservation law. In this paper, by utilizing the idea of $I$-method, we modify the $H^1$ energy to obtain an ``almost conservation law".  Therefore, we can get the global well-posedness in $H^s$, $s>19/22$.

There are many researchers studying the local and global well-posdeness theory of fifth-order KdV-type equations. For the fifth-order KdV equation, whose nonlinearities are quadratic, Chen and Guo \cite{CG} gave the local and global well-posedness in $H^s$ with $s\geq -7/4$ for the equation with the first-order derivative nonlinearities. The fifth-order KdV equation with the third-order derivative nonlinearities was considered by Kwon \cite{Kwon}, and he showed that it is locally well-posedness in $H^s$, $s>5/2$. In \cite{GKK}, the authors prove a priori bound of solutions to this equation in $H^s$, $s\geq 5/4$, and showed that the equation is globally well-posed in $H^2$ due to the second energy conservation law. For the fifth-order modified KdV equation, whose nonlinearities are cubic, Linares \cite{Linares} gave the local well-posedness in $H^2$ by using the dispersive smoothing effect, and then the global well-posedness is immediately obtained via the conservation law. Kwon \cite{SK08} improved the local well-posedness theory to the Sobolev spaces $H^s$ with $s\geq 3/4$. In addition, by using $I$-method, Gao \cite{Gao18} prove the global well-posedness in $H^s$, $s>-3/22$ for the fifth-order modified KdV equation with a first-order derivative nonlinearity. Via the energy method and the short-time Fourier restriction norm method, Kwak \cite{Kwak} considered the Equation \eqref{5mKdV} under the periodic boundary condition and proved the local well-posedness in $H^s(\mathbb{T})$, $s>2$.  For more results, we lead the readers to \cite{JiaHuo,Killip,KPV,Ponce} and the references therein.

From the integrable point of view, in the context of inverse scattering, the first work to provide explicit formulas depending only on initial conditions for long-time asymptotics of solutions is due to Zakharov and Manakov \cite{ZM}, where the model considered was the nonlinear Schr\"odinger (NLS) equation. In this setting, the inverse scattering map and the reconstruction of the potential is formulated through an oscillatory Riemann--Hilbert (RH) problem. Then the now well-known nonlinear steepest descent method for oscillatory RH problems introduced by Deift and Zhou in \cite{DZ1993} provides a rigorous and transparent proof to analyze the long-time asymptotics of initial value problems for a large range of nonlinear integrable evolution equations with initial conditions lying in Schwartz space $\mathscr{S}(\bfR)$. Numerous new significant results for the initial value or initial-boundary value problems of a lot of integrable systems were obtained under the
assumptions that the initial or initial boundary data belong to the Schwartz space \cite{AL,DZ1993,DZ1995,CVZ,XF,LGWW}.

To consider the asymptotic behavior of the solution in lower regularity space, we have to mention another meaningful result developed by Zhou in \cite{ZX1998}, that is, the $L^2$-Sobolev space bijectivity of the direct and inverse scattering of the $2\times2$ AKNS system for the initial data $u_0(x)$ belonging to the weighted Solobev space $H^{i,j}(\bfR)$, where
$$H^{i,j}(\bfR)=\left\{f(x)\in L^2(\bfR): x^jf(x), f^{(i)}(x)\in L^2(\bfR)\right\},$$
which is endowed with the following norm
\berr
\|f(x)\|_{H^{i,j}(\bfR)}=\left(\|f(x)\|^2_{L^{2}(\bfR)}+\|x^jf(x)\|^2_{L^{2}(\bfR)}+\|f^{(i)}(x)\|^2_{L^{2}(\bfR)}\right)^{\frac{1}{2}}.
\eerr
And then Deift and Zhou in \cite{DZ2003} obtained the long-time asymptotics for solutions of the defocusing NLS equation with initial data in a weighted Sobolev space $H^{1,1}$ based on this work. Moreover, inspired by the work \cite{MM2008} and combined Zhou's result \cite{ZX1998}, the Dbar generalization of the nonlinear steepest descent method proposed by Dieng and McLaughlin \cite{DM} was implemented in the analysis of the long-time asymptotic behavior of the solutions, where they derived the asymptotics for the defocusing NLS equation with the initial conditions $u_0(x)$ belonging to the weighted Sobolev space $H^{1,1}(\bfR)$. This approach replaces the rational approximation of the reflection coefficient in Deift--Zhou method \cite{DZ1993} by some non-analytic extension of the jump matrices off the real axis, leading to a $\bar{\partial}$ problem to be solved in some sectors of the complex plane. The new $\bar{\partial}$ problem can be recast into an integral equation and solved by Neumann series, which contributed to the error estimate. In the context of NLS equatin with soliton solutions, they were successfully applied to prove asymptotic stability of $N$-soliton solutions to defocusing NLS equation \cite{CJ} and address the soliton resolution problem for focusing NLS equation in \cite{BJM}. Whereafter, the long-time asymptotic behavior of derivative NLS equation for generic initial data in $H^{2,2}(\bfR)$ that do not support solitons and can support solitons (but exclude spectral singularities) were analyzed in \cite{LPS} and \cite{JLPS}, respectively. Recently, long-time behavior of the defocusing modified KdV equation and the soliton resolution of the focusing modified KdV equation in weighted Sobolev spaces were reported in \cite{CL,CL1}.

It is worth noting that some researchers also study the long-time asymptotic behavior of the dispersive equations by PDE techniques, which do not need the complete integrability. For instance, one can see the references \cite{Hayashi}, \cite{GPR} and \cite{HG}  about the methods of factorization of operators, the space-time resonance, and wave packets analysis, respectively. However, they all need some smallness assumptions on the initial data. Specially, employing the method of testing by wave packets, Okamoto \cite{O} has proved the small-data global existence and derived the asymptotic behavior of solutions in both the self-similar and oscillatory regions to a fifth-order modified KdV type equation in spaces $H^{2,1}(\mathbb{R})$, whereas the complex-valued function $W$ arising in the leading-order term can not be written exactly in oscillatory region.

In this paper, relying on complete integrability, we drop the smallness conditions and consider the long-time asymptotic behavior of the fifth-order modified KdV equation \eqref{5mKdV} by using the $\bar{\partial}$ generalization of the nonlinear steepest descent method with the initial conditions $u_0(x)$ belonging to the weighted Sobolev space $H^{4,1}(\bfR).$

More precisely, in the physically interesting Region I: $k_0\leq M,\ \tau\rightarrow\infty$, where $M$ is a positive constant, the phase function $\Phi(k)$ defined in \eqref{2.29} has two different real critical points located at $\pm k_0$, where
$
k_0=\sqrt[4]{x/(80t)}.
$
To implement the $\bar{\partial}$ approach, our first step is to introduce a scale function $\delta(k)$ conjugated to the solution $M(x,t;k)$ of the original RH problem 1.1. This operation is aimed to factorize the jump matrix \eqref{1.20} into a product of a upper/lower triangular and lower/upper triangular matrix, which is necessary in the following contour deformation. The second step is to deform the contour from $\bfR$ to a new contour depicted in Fig. \ref{fig3} so that the jump matrix involves the exponential factor $\e^{-t\Phi}$ on the parts of the contour where Re$\Phi$ is positive and the factor $\e^{t\Phi}$ on the parts where Re$\Phi$ is negative. To achieve this goal, the idea is to construct $\bar{\partial}$ extensions $R_j(k)$ in $\Omega_j$ to have the prescribed boundary values and $\bar{\partial}R_j(k)$ small in the sector (see Lemma \ref{lem2.1}). This will allow us to reformulate RH problem 2.1 as a mixed $\bar{\partial}$-RH problem 2.2 for a new matrix-valued function $M^{(2)}(x,t;k)$. The next step is to extract from $M^{(2)}$ a contribution that is a pure RH problem. More exactly, we factorize $M^{(2)}(k)=M^{(3)}(k)M^{mod}(k)$, where $M^{mod}(k)$ is a solution of the RH problem 2.3 below with the jump matrix $J^{mod}=J^{(2)}$, and $M^{(3)}(k)$ has no jump across $\Gamma$ which will prove to be satisfied a pure $\bar{\partial}$ problem 2.6. Since the reflection coefficient $r(k)$ in the jump matrix $J^{mod}$ is only fixed at $\pm k_0$ along the deformed contours, we then can solve this RH problem in terms of parabolic cylinder functions and the large-$k$ expansion can be exactly written with the estimate of decay rate as $\tau\rightarrow\infty$ (see Subsection \ref{subsec2.3}). The remaining $\bar{\partial}$ problem may be written as an integral equation (refer to equation \eqref{2.71}) whose integral operator has small norm at large times (see equation \eqref{2.74}) and the large-time contribution to the asymptotics of $u(x,t)$ is negligible. The final step is to regroup all the transformations to find the behavior of the solution of fifth-order modified KdV equation as $\tau\rightarrow\infty$ using the large-$k$ behavior of the RH problem solutions.

However, our main RH problem has only jump across the real axis $\bfR$, the main contributions to the asymptotic formula of the solution $u(x,t)$ come from the local RH problems near the two real critical points $\pm k_0$ even though there are two real and two pure imaginary critical points of the phase function $\Phi(k)$ located at the points $\pm k_0$ and $\pm\ii k_0$. Due to symmetries of reflection coefficient $r(k)$, the leading-order aysmptotics of $u(x,t)$ exhibits decaying, of order $O(t^{-1/2})$, modulated oscillations in Region I.

Another interesting region is the Region III: $\tau\leq M'$. Especially, as $t\rightarrow\infty$, the critical points $\pm k_0$ approach 0. The above steps are much easier to operate in this case. We can show that the asymptotics of the solution $u(x,t)$ in this region is expressed in terms of the solution of a fourth-order Painlev\'e II equation (see \cite{KS})
\be\label{4Painleve}
u_p^{''''}(y)-40u_p^2(y)u_p^{''}(y)-40u_p(y)u_p'^2(y)+96u_p^5(y)-4yu_p(y)=0.
\ee
It is a beautiful example that the asymptotic behavior of the solution of an integrable equation is expressed by the solution of the high order Painlev\'e II equation in the asymptotic analysis. Nevertheless, the asymptotic behaviors of solutions to the standard defocusing mKdV in sector $0<x<Mt^{\frac{1}{3}}$ correspond to the solution of the standard Painlev\'e II equation
$u_p''(y)-2u_p^3(y)-yu_p(y)=0,$ both on the line \cite{DZ1993} and half-line problem \cite{JLSiam}. We also note that for the Camassa--Holm and Sasa--Satsuma equation, there also exists Painlev\'e II-type asymptotics in some certain regions \cite{AAD,HLJL}. Interestingly, the Painlev\'e III hierarchy  has appeared in recent study of the fundamental rogue wave solutions of the focusing nonlinear Schr\"odinger equation in the limit of large order \cite{BLP}. However, for the fifth-order modified KdV equation \eqref{5mKdV}, we find another interesting new asymptotic result which related to the solution of a fourth-order Painlev\'e II equation \eqref{4Painleve}. The study of asymptotic behaviors of the solution $u(x,t)$ in remaining regions is derived in Section \ref{sec4}, which follows the same strategy.

All in all, our asymptotic results can be divided into three categories: in Region I (which is a oscillatory region), the leading-order asymptotics of $u(x,t)$ is described by the form of cosine function; in the self-similar Regions II-IV, the leading-order asymptotic behaviors can be expressed in terms of the solution of a fourth-order Painlev\'e II equation; in the Region V, the solution $u(x,t)$ tends to 0 with fast decay.

To extend these asymptotic behavior results to lower regularity spaces, we have to obtain the global well-posedness theory in some low regularity spaces by PDE techniques first. As mentioned before, the global solution in space $H^1(\mathbb{R})$ is easily to obtain owing to the energy conservation. The biggest obstacle in getting global solutions in $H^s$ with $0<s<1$ is the lack of any conservation law. In this paper, by utilizing the idea of $I$-method, we modify the $H^1$ energy to obtain an ``almost conservation law", whose increment is very small.  Therefore, we can get the global well-posedness in low regularity Sobolev spaces and then extend the long-time asymptotic formulae to the rough data $u_0(x)\in H^{s,1}(\mathbb{R})$, $s>19/22$.

Broadly speaking, the asymptotic analysis of \eqref{5mKdV} presents following innovation points: (i) the phase function $\Phi(k)$ in jump matrix raises to the fifth power of $k$, this will lead to the computations in the scaling transform and related estimates about the $\bar{\partial}$ problem more involved; (ii) we should establish a new suitable model RH problem which arise in the study of long-time asymptotics in the Region III; (iii) because of the lack of any conservation law in $H^s$, $s<1$, we have to get an ``almost conservation law" by utilizing $I$-method. As far as the authors can see, it is the first time that using the idea of $I$-method to hand with the third-order derivative nonlinearities. Compared with the first-order derivative nonlinearities, the absence of the symmetries and the vanished property in the case of third-order derivatives makes the problem fairly tricky. Thanks to an important observation, substantial technical difficulties cased by the third-order derivative nonlinearities could be solved.

\subsection{Formulation of the Riemann--Hilbert problem}
Equation \eqref{5mKdV} is the compatibility condition for the simultaneous linear equations of a Lax pair
\begin{equation}\label{1.3}
\Psi_x=X\Psi,\quad X=\ii k\sigma_3+U,\quad\sigma_3={\left( \begin{array}{cc}
1 & 0 \\[4pt]
0 & -1\\
\end{array}
\right )},\quad
U={\left( \begin{array}{cc}
0 & u \\[4pt]
u & 0 \\
\end{array}
\right )},
\end{equation}
and
\bea\label{1.4}
\Psi_t&=&T\Psi,~\quad T=-16\ii k^5\sigma_3+V,\\
V&=&-8\ii k^3U^2\sigma_3+2\ii k(2UU_{xx}-U_x^2-3U^4)\sigma_3-16k^4U+8\ii k^3\sigma_3U_x+k^2(4U_{xx}-8U^3)\nn\\
&&+\ii k\sigma_3(12U^2U_x-2U_{xxx})-6U^5+10U^2U_{xx}+10UU_x^2-U_{xxxx}.\nn
\eea
governing a $2\times2$ matrix-valued function $\Psi(x,t;k)$ and the spectral parameter  $k\in\bfC$. For the given initial data $u_0(x)\in L^1(\bfR)$, we define the Jost eigenfunctions $\Psi_\pm(x;k)=\psi_\pm(x;k)\e^{\ii kx\sigma_3}$ by
\bea
\psi_\pm(x;k)=I+\int_{\pm\infty}^x\e^{\ii k(x-x')\sigma_3}[U(x',0)\psi_\pm(x';k)]\e^{-\ii k(x-x')\sigma_3}\dd x',\label{1.5}
\eea
which satisfy
\begin{equation}\label{1.6}
\psi_{\pm x}(x;k)-\ii k[\sigma_3,\psi_\pm(x;k)]=U(x,0)\psi_\pm(x;k),
\end{equation}
and the normalization conditions
\berr
\lim_{x\rightarrow\pm\infty}\psi_\pm(x;k)=I.
\eerr
If we denote $\psi=(\psi^1,\psi^2)$ for a $2\times2$ matrix $\psi$, it follows from \eqref{1.5} that for all $(x,t)$:

(i) $\det[\psi_\pm]=1$.

(ii) $\psi_{+}^1$ and $\psi^2_{-}$ are analytic and bounded in $\{k\in\bfC|\mbox{Im}k>0\}$, and $(\psi_{+}^1,\psi^2_{-})\rightarrow I$ as $k\rightarrow\infty.$

(iii) $\psi_-^1$ and $\psi_+^2$ are analytic and bounded in $\{k\in\bfC|\mbox{Im}k<0\}$, and $(\psi_-^1,\psi_+^2)\rightarrow I$ as $k\rightarrow\infty.$

(iv) $\psi_\pm$ are continuous up to the real axis.

(v) Symmetry:
\begin{align}
\overline{\psi_\pm(x;\bar{k})}=\psi_\pm(x;-k)=\sigma_1\psi_\pm(x;k)\sigma_1,\label{1.7}
\end{align}
where $\sigma_1$ is the first Pauli matrix. The symmetry relation \eqref{1.7} can be proved easily due to the symmetries of the matrix $X$:
$$\overline{X(x,t;\bar{k})}=X(x,t;-k)=\sigma_1X(x,t;k)\sigma_1.$$
Both $\psi_+$ and $\psi_-$ define a fundamental solution matrix for \eqref{1.6} and so there exists a continuous matrix $s(k)$ independent of $x$, satisfying
\bea\label{1.8}
\Psi_+(x;k)=\Psi_-(x;k)s(k),\quad k\in\bfR.
\eea
Evaluation at $x\rightarrow-\infty$ gives
\be\label{1.9}
\begin{aligned}
s(k)&=\lim_{x\rightarrow-\infty}\e^{-\ii kx\sigma_3}\psi_+(x;k)\e^{\ii kx\sigma_3},\\
&=I-\int_{-\infty}^{+\infty}\e^{-\ii kx\sigma_3}[U(x,0)\psi_+(x;k)]\e^{\ii kx\sigma_3}\dd x.
\end{aligned}
\ee
It follows from $\det[\psi_\pm(x;k)]=1$ and the symmetries \eqref{1.7} that the matrix-valued function $s(k)$ can be expressed as
\be\label{1.10}
s(k)=\begin{pmatrix}
a(k) ~& \overline{b(k)}\\
b(k) ~& \overline{a(k)}
\end{pmatrix},\quad \det[s(k)]=1,\quad k\in\bfR.
\ee
Moreover, by \eqref{1.9} and \eqref{1.10}, we find that $a(k)$ is analytic in $\bfC_+$, and
\be\label{1.11}
a(k)=\overline{a(-\bar{k})},\quad b(k)=\overline{b(-k)}.
\ee
However, it follows from $\det[s(k)]=1,$ i.e.,
\be
|a(k)|^2-|b(k)|^2=1
\ee
that $a(k)$ is zero-free.

At each time $t$, we define the spectral matrix $s(k,t)$ according to \eqref{1.9}, i.e.
\be\label{1.13}
s(k,t)=\lim_{x\rightarrow-\infty}\e^{-\ii kx\sigma_3}\psi_+(x,t;k)\e^{\ii kx\sigma_3}.
\ee
Then from the $t$-part Lax pair of \eqref{5mKdV}, the evolution of the spectral matrix is given by
\be\label{1.14}
s(k,t)=\e^{-16\ii k^5t\sigma_3}s(k)\e^{16\ii k^5t\sigma_3}.
\ee
More precisely, we have
\be\label{1.15}
a(k,t)=a(k),~b(k,t)=b(k)\e^{32\ii k^5t}.
\ee

Now we define the Beals--Coifman solution
\be
M(x,t;k)=\left\{
\begin{aligned}
&\bigg(\frac{\psi_+^1(x,t;k)}{a(k)},\psi_-^2(x,t;k)\bigg),\qquad k\in\bfC_+,\\
&\bigg(\psi_-^1(x,t;k),\frac{\psi_+^2(x,t;k)}{\overline{a(\bar{k})}}\bigg),\qquad k\in\bfC_-.
\end{aligned}
\right.
\ee
Then, the boundary values $M_{\pm}(x,t;k)$ of $M$ as $k$ approaches $\bfR$ from the sides $\pm\text{Im}k>0$ are related as follows:
\be\label{2.28}
M_+(x,t;k)=M_-(x,t;k)J(x,t;k),\quad k\in\bfR,
\ee
with
\be\label{2.29}
\begin{aligned}
J(x,t;k)&=\begin{pmatrix}
1-|r(k)|^2 &~ -\overline{r(k)}\e^{-t\Phi(k)}\\
r(k)\e^{t\Phi(k)} &~ 1
\end{pmatrix},\\
r(k)&=\frac{b(k)}{a(k)},\quad
\Phi(k)=32\ii k^5-2\ii\frac{x}{t}k.
\end{aligned}
\ee
It obviously follows from Zhou's results \cite{ZX1998} that the following map holds and is Lipschitz:
\bea
\mathcal{D}:H^{4,1}(\bfR)\ni\{u_0(x)\}\mapsto \{r(k)\}\in H^{1,4}(\bfR).\nn
\eea
Then, we arrive at our main RH problem which is formulated as follows.\\
\textbf{Riemann--Hilbert problem 1.1.} Given $r(k)\in H^{1,4}(\bfR)$. Find an analytic $2\times2$ matrix-valued function $M(x,t;k)$ on $\bfC\setminus\bfR$ with the following properties:

1. $M(x,t;k)$ is analytic for $k\in\bfC\setminus\bfR$ and is continuous for $k\in\bfR$.

2. The boundary values $M_\pm(x,t;k)$ satisfy the jump condition
\bea
M_+(x,t;k)=M_-(x,t;k)J(x,t;k),\label{1.19}
\eea
where
\bea
J(x,t;k)=\begin{pmatrix}
1-|r(k)|^2 &~ -\overline{r(k)}\e^{-t\Phi(k)}\\
r(k)\e^{t\Phi(k)} &~ 1
\end{pmatrix},~~r(-k)=\overline{r(k)},~~k\in\bfR.\label{1.20}
\eea

3. $M(x,t;k)$ has the asymptotics: $$M(x,t;k)=I+O\bigg(\frac{1}{k}\bigg),~k\rightarrow\infty.$$

The jump matrix admits the following factorization:
\be
J(x,t;k)=J_-^{-1}(x,t;k)J_+(x,t;k)=(I-w_-(x,t;k))^{-1}(I+w_+(x,t;k)),
\ee
where
\begin{align}
w_-&=\begin{pmatrix}
0 &~ -\overline{r(k)}\e^{-t\Phi(k)}\\
0 &~ 0
\end{pmatrix},\quad w_+=\begin{pmatrix}
0 &~ 0\\
r(k)\e^{t\Phi(k)} &~ 0
\end{pmatrix},\quad k\in\bfR.\nn
\end{align}
By standard RH theory, the existence and uniqueness of the solution to Riemann--Hilbert problem 1.1 is determined by the existence and uniqueness of the following singular integral equation
\be
\mu(x,t;k)=I+\mathcal{C}_w\mu(x,t;k)=I+\mathcal{C}_+(\mu w_-)(x,t;k)+\mathcal{C}_-(\mu w_+)(x,t;k),
\ee
where
\be
\mu(x,t;k)=M_+(x,t;k)(I+w_+(x,t;k))^{-1}=M_-(x,t;k)(I-w_-(x,t;k))^{-1},
\ee
and $\mathcal{C}_\pm$ denote the Cauchy operators:
\berr
(\mathcal{C}_\pm f)(k)=\lim_{\epsilon\rightarrow0+}\int_{\bfR}\frac{f(\zeta)}{\zeta-(k\pm\ii\epsilon)}\frac{\dd \zeta}{2\pi\ii}.
\eerr
It then follows from \cite{ZX1989} that $I-\mathcal{C}_{w}$ is invertible, as a result that RH problem 1.1 has a unique solution. Then the solution $M(x,t;k)$ is given by
\be
M(x,t;k)=I+\frac{1}{2\pi\ii}\int_\bfR\frac{[\mu(w_++w_-)](x,t;s)}{s-k}\dd s.
\ee
The solution $u(x,t)$ of \eqref{5mKdV} in terms of $M(x,t;k)$ is given by
\bea\label{1.21}
u(x,t)=-2\ii\lim_{k\rightarrow\infty}(k M(x,t;k))_{12}=\frac{1}{\pi}\left(\int_\bfR[\mu(w_++w_-)](x,t;s)\dd s\right)_{12}.
\eea
\subsection{Main results}
\begin{figure}[htbp]
  \centering
  \includegraphics[width=3.5in]{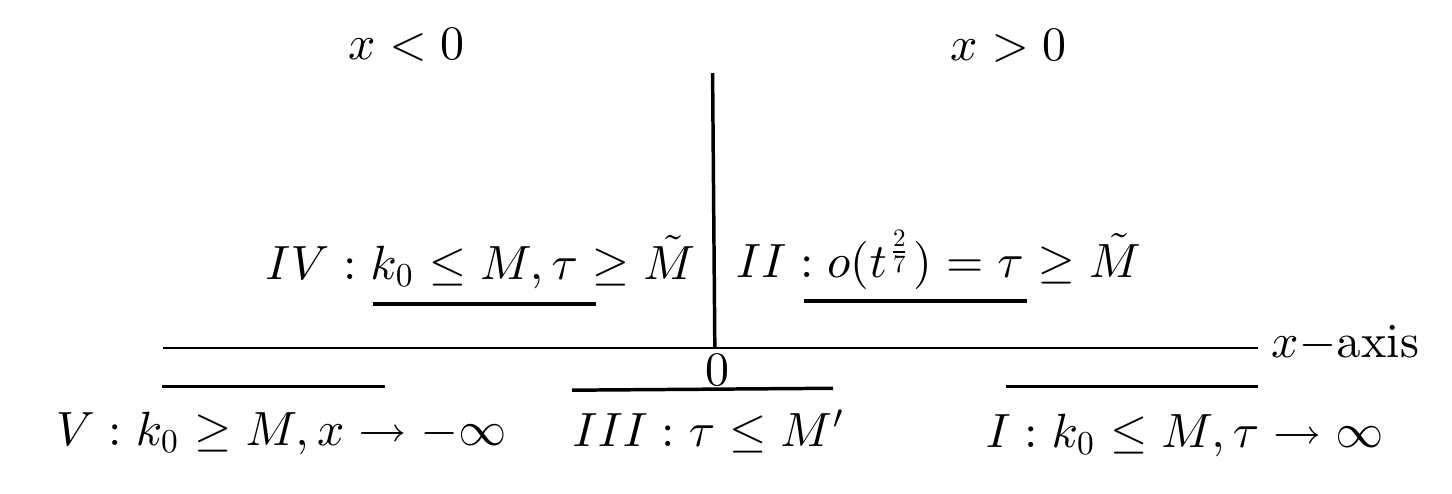}
  \caption{The five regions.}\label{fig1}
\end{figure}
Our main results in this paper are summarized by the following theorems. The first one is about the asymptotic behavior of the Beals--Coifman solution to Equation \eqref{5mKdV}.
\begin{theorem}\label{the1.1}
Let $u_0(x)\in H^{4,1}(\bfR)$, $M$, $M'$ and $\tilde{M}$ be fixed constants larger than 1. Select a suitable constant $\rho>0$. Then the solution $u(x,t)$ of fifth-order modified KdV equation \eqref{5mKdV} with the initial data $u_0(x)$ admits the following uniform asymptotics as $t\rightarrow\infty$ in the five regions depicted in Fig. \ref{fig1}. More precisely, set
\be
k_0=\sqrt[4]{\frac{x}{80t}},\quad \tau=tk_0^5,
\ee
and let $u_p(y)$ be the solution of the fourth order Painlev\'e II equation \eqref{A.5}. Then, we have \\
In Region I,
\berr
u(x,t)=u_{as}(x,t)+O\bigg(\tau^{-1}+(k_0^3t)^{-\frac{3}{4}}\bigg),
\eerr
where
\begin{align}
u_{as}(x,t)&=\frac{\sqrt{\nu}}{2k_0\sqrt{10k_0t}}\cos\bigg(128tk_0^5+\nu\ln(2560 tk_0^5)+\varphi(k_0)\bigg),\nn\\
\varphi(k_0)&=-\frac{3\pi}{4}-\arg r(k_0)+\arg\Gamma(\ii\nu)-
\frac{1}{\pi}\int_{-k_0}^{k_0}\ln\bigg(\frac{1-|r(s)|^2}{1-|r(k_0)|^2}\bigg)\frac{\dd s}{s-k_0},\nn\\
\nu&=-\frac{1}{2\pi}\ln(1-|r(k_0)|^2)>0,\nn
\end{align}
and $\Gamma(\cdot)$ denotes the standard Gamma function.\\
In Region II,
\berr
u(x,t)=\bigg(\frac{8}{5t}\bigg)^{\frac{1}{5}}u_p\bigg(\frac{x}{(20t)^{\frac{1}{5}}}\bigg)+
O\bigg(t^{-\frac{3}{10}}+\bigg(\frac{\tau}{t}\bigg)^{\frac{2}{5}}\bigg).
\eerr
In Region III,
\berr
u(x,t)=\bigg(\frac{8}{5t}\bigg)^{\frac{1}{5}}u_p\bigg(\frac{x}{(20 t)^{\frac{1}{5}}}\bigg)+O(t^{-\frac{3}{10}}).
\eerr
In Region IV,
\berr
u(x,t)=\bigg(\frac{8}{5t}\bigg)^{\frac{1}{5}}u_p\bigg(\frac{x}{(20 t)^{\frac{1}{5}}}\bigg)+O(t^{-\frac{3}{10}}\e^{-8(20\tau)^{\frac{4}{5}}\rho}).
\eerr
In Region V,
\berr
u(x,t)=O\bigg(t^{-\frac{1}{5}}\e^{-c\tau}+t^{-\frac{3}{10}}
\e^{-8(20\tau)^{\frac{4}{5}}\rho}\bigg).
\eerr
\end{theorem}
\begin{remark}
The significance of the restriction $\tau=o(t^{\frac{2}{7}})$ in Region II is to ensure that the asymptotic formulas of the solution match up in the overlap section of Region I and Region II. For more details, one can see Remark \ref{rem4.2}.
\end{remark}

We next consider Equation \eqref{5mKdV} in low regularity spaces. From the PDE point of view, one always consider the strong solution given by an integral form. For convenience, we denote
$$F(u)=30u^4\partial_xu-10u^2\partial_x^3u-10(\partial_xu)^3-40u\partial_xu\partial_x^2u,$$
then by Duhamel's formula, \eqref{5mKdV} is equivalent to the following integral equation:
\begin{align}
u(x,t)= \e^{-t\partial_x^5}{u_0} -\int^t_0 \e^{-(t-t')\partial_x^5}{F(u)}(t')\, \dd t',\quad  {\rm where} \quad \e^{-t\partial_x^5}=\mathscr{F}^{-1}\e^{-{\rm i}t\xi^5}\mathscr{F}.  \label{I5mKdV}
\end{align}

Using the idea of $I$-method \cite{Tao1,Tao2}, we obtain the global well-posedness theory.
\begin{theorem}\label{the1.2}
Let $19/22<s\leq1$, the initial value problem of fifth-order modified equation \eqref{5mKdV} is globally well-posed from initial data $u_0\in H^s(\mathbb{R})$.
\end{theorem}

In the end, cooperating the integrable structure with PDE techniques, we can give the long-time asymptotic behavior in lower regularity spaces.
\begin{theorem}\label{the1.3}
Let  $u_{0}\in H^{s,1}(\mathbb{R})$, $19/22<s\leq1$, the strong solution given by the integral form \eqref{I5mKdV} has the same asymptotic behavior as in Theorem \ref{the1.1}.
\end{theorem}

{\bf Notations.} $C$ denotes a universal positive constant which can always be different at different lines. $\mathscr{F}$ ($\mathscr{F}^{-1}$) denotes the (inverse) Fourier transform on $ \mathbb{R}$. We also denote $\widehat{\phi}$ the Fourier transform of a distribution $\phi$. We write $a\lesssim b$ if $a\leq C b$, and analogous for $a\gtrsim b$. We use the notation $a\sim b$ if $a\lesssim b \lesssim a$. We denote by $a+= a+\varepsilon$ for some $0<\varepsilon\ll 1$. We denote by $\bar{a}$ the complex conjugate of a complex number $a$. The three Pauli matrices are defined by
\berr
\sigma_1=\begin{pmatrix}
0 & 1\\
1 & 0
\end{pmatrix},\quad\sigma_2=\begin{pmatrix}
0 & -\ii\\
\ii & 0
\end{pmatrix},\quad \sigma_3=\begin{pmatrix}
1 & 0\\
0 & -1
\end{pmatrix}.
\eerr

The organization of this paper is as follows. In Section 2, we derive the long-time asymptotics of the solution for Equation \eqref{5mKdV} in Region I. Section 3 is aim to consider the asymptotics of $u(x,t)$ in the Region III. The asymptotic behavior of the solution in remaining regions is derived in Section 4. We then extend the long-time asymptotic behavior to the low regularity spaces in Section 5. A few facts related to the RH problem associated with the fourth order Painlev\'e II equation are collected in the Appendix A and B, for the detailed proof one can see \cite{LGWW}.
\section{Asymptotics in Region I}
\setcounter{equation}{0}
\setcounter{lemma}{0}
\setcounter{theorem}{0}

The jump matrix $J$ defined in \eqref{1.20} involves the exponentials $\e^{\pm t\Phi}$, therefore, the sign structure of the quantity Re$\Phi(k)$ plays an important role as $t\rightarrow\infty$. In particular, in region I, it follows that there are two different real stationary points located at the points where $\frac{\partial\Phi}{\partial k}=0$,
namely, at
\bea
\pm k_0=\pm\sqrt[4]{\frac{x}{80t}}.\label{2.1}
\eea
The signature table for Re$\Phi(k)$ is shown in Fig. \ref{fig2}.
\begin{figure}[htbp]
  \centering
  \includegraphics[width=3in]{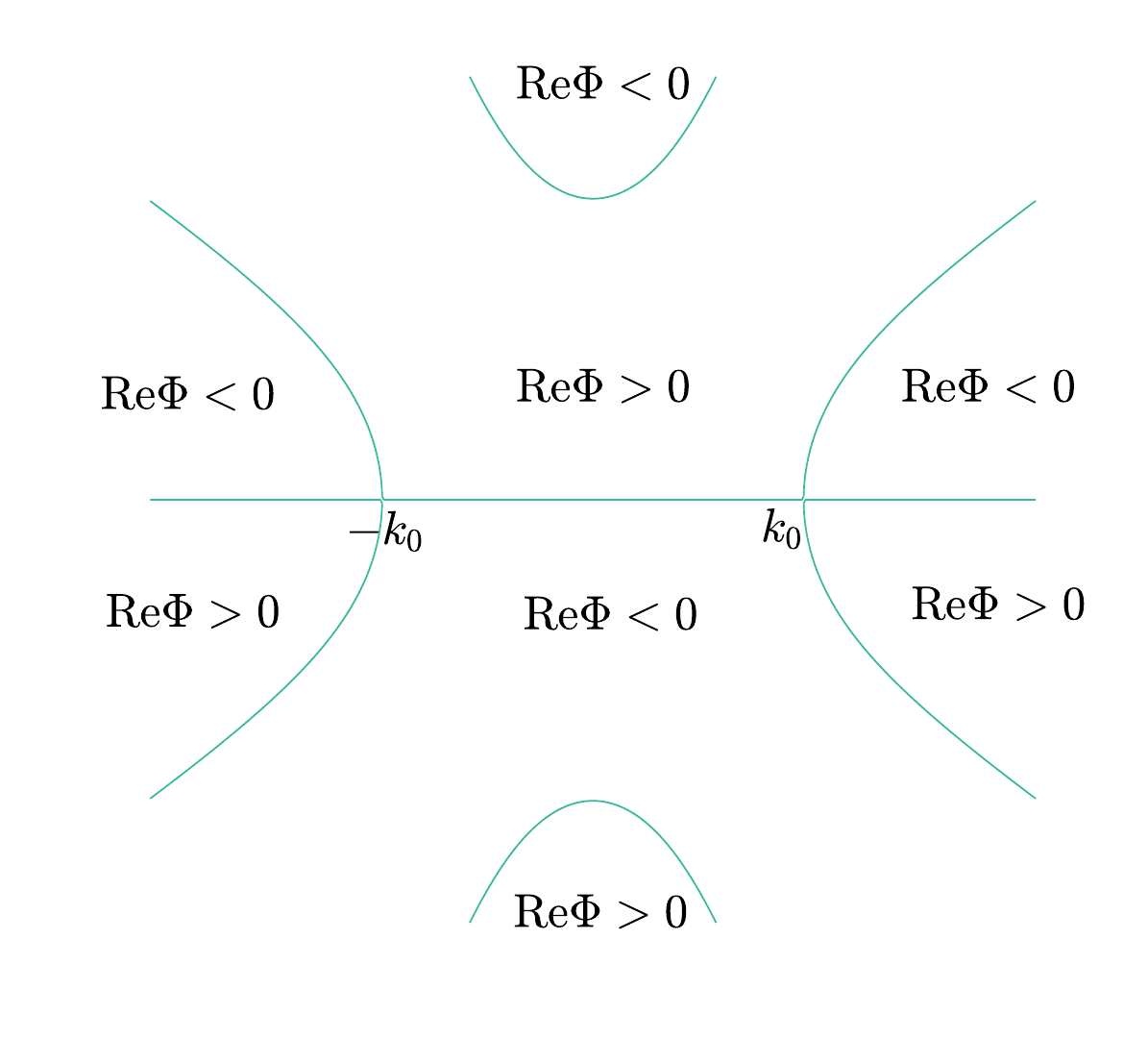}
  \caption{The signature table for Re$\Phi(k)$ in the complex $k$-plane.}\label{fig2}
\end{figure}
\subsection{Conjugation}
In order to apply the method of steepest descent, one goes from the original RH problem 1.1 for $M$ to the equivalent RH problem for the new unknown function $M^{(1)}$ defined by
\be\label{2.2}
M^{(1)}(x,t;k)=M(x,t;k)\delta^{-\sigma_3}(k),
\ee
where the complex-valued function $\delta(k)$ is given by
\be\label{2.3}
\delta(k)=\exp\bigg\{\frac{1}{2\pi\ii}\int_{-k_0}^{k_0}\frac{\ln(1-|r(s)|^2)}{s-k}\dd s\bigg\},~ k\in\bfC\setminus[-k_0,k_0].
\ee
\begin{proposition}\label{prop2.1}
The function $\delta(k)$ has the following properties:

(i) $\delta(k)$ is bounded and analytic function of $k\in\bfC\setminus[-k_0,k_0]$ with continuous boundary values on $(-k_0,k_0)$, and satisfies the jump condition across the real axis oriented from $-\infty$ to $\infty$:
\berr
\delta_+(k)=\left\{
\begin{aligned}
&\delta_-(k)(1-|r(k)|^2),~~k\in(-k_0,k_0),\nn\\
&\delta_-(k),\qquad\qquad\qquad k\in\bfR\setminus(-k_0,k_0).
\end{aligned}
\right.
\eerr

(ii) As $k\rightarrow\infty$, $\delta(k)$ satisfies the asymptotic formula
\be\label{2.4}
\delta(k)=1+O(k^{-1}),\quad k\rightarrow\infty.
\ee

(iii) $\delta(k)$ obeys the symmetry $$\delta(k)=\overline{\delta(\bar{k})}^{-1}=\overline{\delta(-\bar{k})},\quad k\in\bfC\setminus[-k_0,k_0].$$

(iv)
\be
\delta(k)=\bigg(\frac{k-k_0}{k+k_0}\bigg)^{\ii\nu}\e^{\chi(k)},
\ee
where
\bea
\nu&=&-\frac{1}{2\pi}\ln(1-|r(k_0)|^2)>0,\\
\chi(z)&=&\frac{1}{2\pi\ii}\int_{-k_0}^{k_0}\ln\bigg(\frac{1-|r(s)|^2}{1-|r(k_0)|^2}\bigg)\frac{ds}{s-k}.
\eea
Moreover, along any ray of the form $\pm k_0+\e^{\ii\phi}\bfR_+$ with $0<\phi<\pi$ or $\pi<\phi<2\pi$,
\be
\bigg|\delta(k)-\bigg(\frac{k-k_0}{k+k_0}\bigg)^{\ii\nu}\e^{\chi(\pm k_0)}\bigg|\leq C|k\mp k_0|^{\frac{1}{2}}.
\ee
\end{proposition}
\begin{proof}
See \cite{CL} and references therein. $\hfill\Box$
\end{proof}
Then $M^{(1)}(x,t;k)$ satisfies the following RH problem:\\
\textbf{Riemann--Hilbert problem 2.1.} Given $r(k)\in H^{1}(\bfR)$. Find an analytic $2\times2$ matrix-valued function $M^{(1)}(x,t;k)$ on $\bfC\setminus\bfR$ with the following properties:

1. $M^{(1)}(x,t;k)$ is analytic for $k\in\bfC\setminus\bfR$ and is continuous for $k\in\bfR$.

2. The boundary values $M_\pm^{(1)}(x,t;k)$ satisfy the jump condition
\bea
M_+^{(1)}(x,t;k)=M_-^{(1)}(x,t;k)J^{(1)}(x,t;k),\quad k\in\bfR,\label{2.9}
\eea
where
\bea\label{2.10}
J^{(1)}(x,t;k)=\left\{
\begin{aligned}
&\begin{pmatrix}
1 ~& -\overline{r(k)}\delta^{2}(k)\e^{-t\Phi(k)} \\[4pt]
0 ~& 1 \\
\end{pmatrix}\begin{pmatrix}
1 ~& 0 \\[4pt]
r(k)\delta^{-2}(k)\e^{t\Phi(k)} ~& 1 \\
\end{pmatrix},~\qquad\quad |k|>k_0,\\
&\begin{pmatrix}
1 ~& 0 \\[4pt]
\frac{r(k)}{1-|r(k)|^2}\delta_-^{-2}(k)\e^{t\Phi(k)} ~& 1 \\
\end{pmatrix}\begin{pmatrix}
1 ~& -\frac{\overline{r(k)}}{1-|r(k)|^2}\delta_+^{2}(k)\e^{-t\Phi(k)} \\[4pt]
0 ~& 1 \\
\end{pmatrix},~~ |k|<k_0.
\end{aligned}
\right.
\eea

3. $M^{(1)}(x,t;k)=I+O\big(\frac{1}{k}\big),$ as $k\rightarrow\infty.$
\subsection{Contour deformation and introducing $\bar{\partial}$ extensions}
The next step is to introduce factorizations of the jump matrix whose factors admit continuous but not necessarily analytic extensions off the real axis by exploiting the methods in \cite{JLPS,BJM,CL,LPS,DM}. More precisely, define the contours
\bea
L&:=&\bigg\{k\in\bfC|k=k_0+k_0\alpha\e^{\frac{5\pi\ii}{6}},
-\infty<\alpha\leq\frac{2\sqrt{3}}{3}\bigg\}\nn\\
&&\cup\{k\in\bfC|k=-k_0+k_0\alpha\e^{\frac{\pi\ii}{6}},
-\infty<\alpha\leq\frac{2\sqrt{3}}{3}\bigg\}.
\eea
Then, $L$ and $\bar{L}$ split the complex plane $\bfC$ into eight regions $\{\Omega_j\}_1^8$, for convenience, we also write $L\cup\bar{L}=\cup_1^8\Gamma_j$, see Fig. \ref{fig3}. Now, we define the $\bar{\partial}$ extensions in the following lemma.
\begin{figure}[htbp]
 \centering
 \includegraphics[width=3.5in]{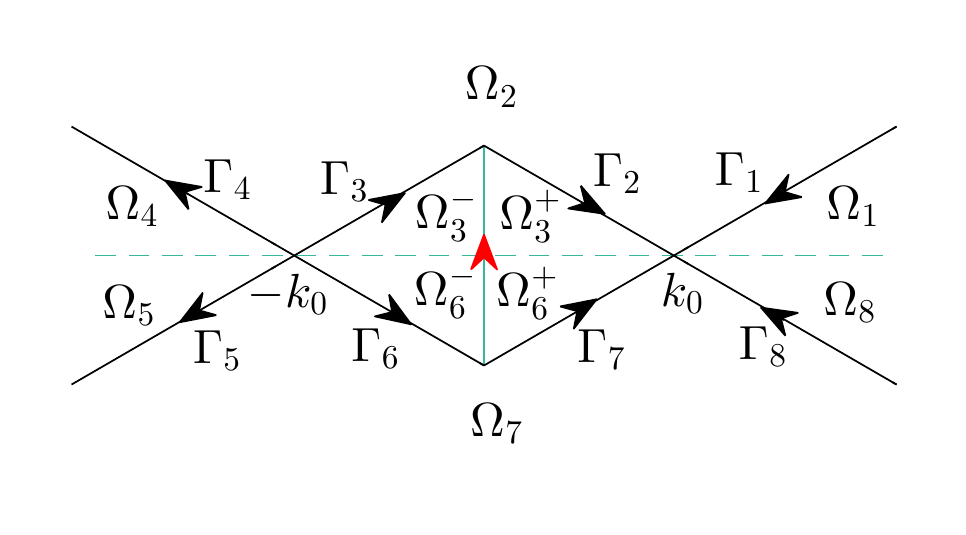}
  \caption{The open sets $\{\Omega_j\}_1^{8}$ and the contours $\{\Gamma_j\}_1^8$ in the complex $k$-plane.}\label{fig3}
\end{figure}
\begin{lemma}\label{lem2.1}
Let
\be
T(k;k_0)=\bigg(\frac{k-k_0}{k+k_0}\bigg)^{\ii\nu}.
\ee
It is possible to define functions $R_j:\bar{\Omega}_j\mapsto\bfC$, $j=1,3,4,5,6,8$ with boundary values satisfying
\begin{align}
R_1(k)&=\left\{
\begin{aligned}
& -r(k)\delta^{-2}(k),\qquad\qquad\qquad\quad~~~ k\in(k_0,\infty),\\
&-r(k_0)T^{-2}(k;k_0)\e^{-2\chi(k_0)},\qquad~~~ k\in\Gamma_1,
\end{aligned}
\right.\label{2.13}\\
R_3^+(k)&=\left\{
\begin{aligned}
&\frac{\overline{r(k)}}{1-|r(k)|^2}\delta^{2}_+(k),\qquad\qquad\qquad~~~ k\in(0,k_0),\\
&\frac{\overline{r(k_0)}}{1-|r(k_0)|^2}T^{2}(k;k_0)\e^{2\chi(k_0)},\qquad~~ k\in\Gamma_2,
\end{aligned}
\right.\label{2.14}\\
R_3^-(k)&=\left\{
\begin{aligned}
&\frac{\overline{r(k)}}{1-|r(k)|^2}\delta^{2}_+(k),~\qquad\qquad\qquad~~ k\in(-k_0,0),\\
&\frac{\overline{r(-k_0)}}{1-|r(-k_0)|^2}T^{2}(k;k_0)\e^{2\chi(-k_0)},~~~~ k\in\Gamma_3,
\end{aligned}
\right.\\
R_4(k)&=\left\{
\begin{aligned}
&-r(k)\delta^{-2}(k),\qquad\qquad\qquad\qquad~ k\in(-\infty,-k_0),\\
&-r(-k_0)T^{-2}(k;k_0)\e^{-2\chi(-k_0)},~~~\quad k\in\Gamma_4,
\end{aligned}
\right.\\
R_5(k)&=\left\{
\begin{aligned}
& -\overline{r(k)}\delta^{2}(k),\qquad\qquad\qquad~~~\qquad k\in(-\infty,-k_0),\\
&-\overline{r(-k_0)}T^{2}(k;k_0)\e^{2\chi(-k_0)},~~~~\qquad k\in\Gamma_5,
\end{aligned}
\right.\\
R_6^-(k)&=\left\{
\begin{aligned}
&\frac{r(k)}{1-|r(k)|^2}\delta^{-2}_-(k),\qquad\qquad\qquad~~~ k\in(-k_0,0),\\
&\frac{r(-k_0)}{1-|r(-k_0)|^2}T^{-2}(k;k_0)\e^{-2\chi(-k_0)},~~ k\in\Gamma_6,
\end{aligned}
\right.\\
R_6^+(k)&=\left\{
\begin{aligned}
&\frac{r(k)}{1-|r(k)|^2}\delta^{-2}_-(k),~\qquad\qquad\qquad~~ k\in(0,k_0),\\
&\frac{r(k_0)}{1-|r(k_0)|^2}T^{-2}(k;k_0)\e^{-2\chi(k_0)},~~~\quad k\in\Gamma_7,
\end{aligned}
\right.\\
R_8(k)&=\left\{
\begin{aligned}
& -\overline{r(k)}\delta^{2}(k),\qquad\qquad\qquad~~~~\qquad k\in(k_0,\infty),\\
&-\overline{r(k_0)}T^{2}(k;k_0)\e^{2\chi(k_0)},~~~\qquad\qquad k\in\Gamma_8,
\end{aligned}
\right.
\end{align}
such that
\be\label{2.21}
|\bar{\partial}R_j(k)|\leq c_1|r'(\text{Re}k)|+c_2|k\mp k_0|^{-\frac{1}{2}},
\ee
for two positive constants $c_1,c_2$ depended on $\|r\|_{H^1(\bfR)}$.
\end{lemma}
\begin{proof}
Define the functions
\begin{align}
f_1(k)&=-r(k_0)T^{-2}(k;k_0)\e^{-2\chi(k_0)}\delta^2(k),\quad\quad~~ k\in\bar{\Omega}_1,\nn\\
f_3^+(k)&=\frac{\overline{r(k_0)}}{1-|r(k_0)|^2}T^{2}(k;k_0)\e^{2\chi(k_0)}\delta^{-2}(k),\quad k\in\bar{\Omega}_3^+.\nn
\end{align}
Then, we can define the extensions for $k\in\bar{\Omega}_1,$
\bea
R_1(k)=\bigg[f_1(k)+\bigg(-r(\text{Re}k)-f_1(k)\bigg)\mathcal{X}(\phi)\bigg]\delta^{-2}(k),
\eea
where $\phi=\arg(k-k_0)$ and $\mathcal{X}(\phi)$ is a smooth cut-off function with
\be
\mathcal{X}(\phi)=\left\{
\begin{aligned}
1,\quad \phi\in\bigg[0,\frac{\pi}{12}\bigg],\\
0,\quad \phi\in\bigg[\frac{\pi}{9},\frac{\pi}{6}\bigg].
\end{aligned}
\right.
\ee
It is easy to check that $R_1(k)$ as constructed has the boundary values \eqref{2.13}. Let $k-k_0=s\e^{\ii\phi}$. It follows from $$\bar{\partial}=\frac{1}{2}\bigg(\frac{\partial}{\partial k_1}+\ii\frac{\partial}{\partial k_2}\bigg)
=\frac{1}{2}\e^{\ii\phi}\bigg(\frac{\partial}{\partial s}+\frac{\ii}{s}\frac{\partial}{\partial\phi}\bigg)$$
and Proposition \ref{prop2.1} (iv) that
\bea
|\bar{\partial}R_1(k)|&=&-\frac{1}{2}r'(\text{Re}k)\mathcal{X}(\phi)\delta^{-2}(k)
+\frac{1}{2}\ii\e^{\ii\phi}\frac{-r(\text{Re}k)-f_1(k)}{|k-k_0|}
\mathcal{X}'(\phi)\delta^{-2}(k)\nn\\
&\leq& c_1|r'(\text{Re}k)|+c_2|k-k_0|^{-\frac{1}{2}}.
\eea
For $k\in\bar{\Omega}_3^+$, let
\be
R_3^+(k)=\bigg[f_3^+(k)+\bigg(\frac{\overline{r(\text{Re}k)}}{1-|r(\text{Re}k)|^2}
-f_3^+(k)\bigg)\tilde{\mathcal{X}}(\phi)\bigg]\delta^2(k),
\ee
where $\tilde{\mathcal{X}}(\phi)$ is a smooth cut-off function with
\be
\tilde{\mathcal{X}}(\phi)=\left\{
\begin{aligned}
1,\quad \phi\in\bigg[\frac{11\pi}{12},\pi\bigg],\\
0,\quad \phi\in\bigg[\frac{3\pi}{4},\frac{5\pi}{6}\bigg].
\end{aligned}
\right.
\ee
A direct computation implies that $R_3^+(k)$ satisfies the boundary values \eqref{2.14} and the relation \eqref{2.21}. In other sectors, the extensions can find similarly.$\hfill\Box$
\end{proof}
We now use the extensions in Lemma \ref{lem2.1} to define a new unknown function $M^{(2)}$ by
\be
M^{(2)}(x,t;k)=M^{(1)}(x,t;k)\mathcal{R}^{(2)}(k),
\ee
where
\be
\mathcal{R}^{(2)}(k)=\left\{
\begin{aligned}
&\begin{pmatrix}
1 &~ 0\\
R_1(k)\e^{t\Phi(k)} &~ 1
\end{pmatrix},\quad k\in\Omega_1,\quad
\begin{pmatrix}
1 &~ R_3^\pm(k)\e^{-t\Phi(k)}\\
0 &~ 1
\end{pmatrix},\quad k\in\Omega_3^\pm,\\
&\begin{pmatrix}
1 &~ 0\\
R_4(k)\e^{t\Phi(k)} &~ 1
\end{pmatrix},\quad k\in\Omega_4,\quad
\begin{pmatrix}
1 &~ R_5(k)\e^{-t\Phi(k)}\\
0 &~ 1
\end{pmatrix},\quad k\in\Omega_5,\\
&\begin{pmatrix}
1 &~ 0\\
R_6^\pm(k)\e^{t\Phi(k)} &~ 1
\end{pmatrix},\quad k\in\Omega_6^\pm,\quad
\begin{pmatrix}
1 &~ R_8^\pm(k)\e^{-t\Phi(k)}\\
0 &~ 1
\end{pmatrix},\quad k\in\Omega_8,\\
&\begin{pmatrix}
1 &~ 0\\
0 &~ 1
\end{pmatrix},\qquad\qquad\quad~ k\in\Omega_2\cup\Omega_7.
\end{aligned}
\right.
\ee
Let $\Gamma=\{\Gamma_j\}_1^8\cup\frac{\ii k_0}{\sqrt{3}}(-1,1)$. It is an immediate consequence of Lemma \ref{2.1} and RH problem 2.1 that $M^{(2)}$ satisfies a mixed $\bar{\partial}$-Riemann--Hilbert problem.\\
\textbf{$\bar{\partial}$-Riemann--Hilbert problem 2.2.} Find a function $M^{(2)}(x,t;k)$ with the following properties:

1. $M^{(2)}(x,t;k)$ is continuous with sectionally continuous first partial derivatives in $\bfC\setminus\Gamma$.

2. Across $\Gamma$, the boundary values satisfy the jump relation
\be
M^{(2)}_+(x,t;k)=M^{(2)}_-(x,t;k)J^{(2)}(x,t;k),~~k\in\Gamma,
\ee
where the jump matrix $J^{(2)}(x,t;k)$ is given by
\bea\label{2.30}
J^{(2)}(x,t;k)&=\left\{
\begin{aligned}
&\begin{pmatrix}
1 &~ 0\\
-r(k_0)T^{-2}(k;k_0)\e^{-2\chi(k_0)}\e^{t\Phi(k)} &~ 1
\end{pmatrix},~\qquad~~ k\in\Gamma_1,\\
&\begin{pmatrix}
1 &~ -\frac{\overline{r(k_0)}}{1-|r(k_0)|^2}T^{2}(k;k_0)\e^{2\chi(k_0)}\e^{-t\Phi(k)}\\
0 &~ 1
\end{pmatrix},\quad~~ k\in\Gamma_2,\\
&\begin{pmatrix}
1 &~ -\frac{\overline{r(-k_0)}}{1-|r(-k_0)|^2}T^{2}(k;k_0)\e^{2\chi(-k_0)}\e^{-t\Phi(k)}\\
0 &~ 1
\end{pmatrix},~~ k\in\Gamma_3,\\
&\begin{pmatrix}
1 &~ 0\\
-r(-k_0)T^{-2}(k;k_0)\e^{-2\chi(-k_0)}\e^{t\Phi(k)} &~ 1
\end{pmatrix}, ~~\quad k\in\Gamma_4,\\
&\begin{pmatrix}
1 &~ \overline{r(-k_0)}T^{2}(k;k_0)\e^{2\chi(-k_0)}\e^{-t\Phi(k)}\\
0 &~ 1
\end{pmatrix},\qquad~~~ k\in\Gamma_5,\\
&\begin{pmatrix}
1 &~ 0\\
\frac{r(-k_0)}{1-|r(-k_0)|^2}T^{-2}(k;k_0)\e^{-2\chi(-k_0)}\e^{t\Phi(k)} &~ 1
\end{pmatrix},~\quad k\in\Gamma_6,\\
&\begin{pmatrix}
1 &~ 0\\
\frac{r(k_0)}{1-|r(k_0)|^2}T^{-2}(k;k_0)\e^{-2\chi(k_0)}\e^{t\Phi(k)} &~ 1
\end{pmatrix},~\qquad~ k\in\Gamma_7,\\
&\begin{pmatrix}
1 &~ \overline{r(k_0)}T^{2}(k;k_0)\e^{2\chi(k_0)}\e^{-t\Phi(k)}\\
0 &~ 1
\end{pmatrix},\qquad\qquad~~ k\in\Gamma_8,\\
&\begin{pmatrix}
1 &~ [R_3^-(k)-R_3^+(k)]\e^{-t\Phi(k)}\\
0 &~ 1
\end{pmatrix},\quad~ k\in\ii k_0\bigg(\tan\frac{\pi}{12},\frac{\sqrt{3}}{3}\bigg),\\
&\begin{pmatrix}
1 &~ 0\\
[R_6^-(k)-R_6^+(k)]\e^{t\Phi(k)} &~ 1
\end{pmatrix},\qquad k\in\ii k_0\bigg(-\frac{\sqrt{3}}{3},-\tan\frac{\pi}{12}\bigg),\\
&I,\quad\qquad\qquad\qquad\qquad\qquad\quad\qquad k\in\ii k_0\tan\frac{\pi}{12}(-1,1).
\end{aligned}
\right.
\eea

3. For $k\in\bfC\setminus\Gamma$, we have
\be\label{2.31}
\bar{\partial}M^{(2)}(k)=M^{(2)}(k)\bar{\partial}\mathcal{R}^{(2)}(k).
\ee

4. $M^{(2)}(x,t;k)$ enjoys asymptotics: $M^{(2)}(x,t;k)=I+O\big(\frac{1}{k}\big),~k\rightarrow\infty.$

In the next subsection, we aim to extract from $M^{(2)}$ a contribution that is a pure RH problem. More precisely, we write
\be\label{2.32}
M^{(2)}(k)=M^{(3)}(k)M^{mod}(k)
\ee
and we request that $M^{(3)}(k)$ has no jump across $\Gamma$. Thus we look for $M^{mod}(k)$ solution of the RH problem 2.3 below with the jump matrix $J^{mod}=J^{(2)}$.

\subsection{Analysis of the Riemann--Hilbert problem 2.3}\label{subsec2.3}

We now focus on $M^{mod}(k)$, which satisfies the following RH problem.\\
\textbf{Riemann--Hilbert problem 2.3.} Find a $2\times2$ matrix-valued function $M^{mod}(x,t;k)$, analytic on $\bfC\setminus\Gamma$ with the following properties:

1. $M^{mod}(x,t;k)$ is analytic for $k\in\bfC\setminus\Gamma$ and is continuous for $k\in\Gamma$.

2. The boundary values $M_\pm^{mod}(x,t;k)$ satisfy the jump condition
\bea
M_+^{mod}(x,t;k)=M_-^{mod}(x,t;k)J^{mod}(x,t;k),\quad k\in\Gamma,\label{2.33}
\eea
where $J^{mod}(x,t;k)=J^{(2)}(x,t;k)$.

3. $M^{mod}(x,t;k)=I+O\big(\frac{1}{k}\big),$ as $k\rightarrow\infty.$

By using the method of Beals and Coifman in \cite{BC}, we set
$$J^{mod}=(I-w_-^{mod})^{-1}(I+w_+^{mod}),\quad w^{mod}=w_+^{mod}+w_-^{mod},$$ and let
\berr
(C_\pm f)(k)=\int_{\Gamma}\frac{f(s)}{s-k_\pm}\frac{\dd s}{2\pi\ii},\quad k\in\Gamma,~f\in L^2(\Gamma),
\eerr
denote the Cauchy operator on $\Gamma$. Define the operator $C_{w^{mod}}:L^2(\Gamma)+L^\infty(\Gamma)\rightarrow L^2(\Gamma)$ by
\be\label{2.34}
C_{w^{mod}}f=C_+(fw_-^{mod})+C_-(fw_+^{mod})
\ee
for any $2\times2$ matrix-valued function $f$. Let $\mu^{mod}(x,t;k)\in L^2(\Gamma)+L^\infty(\Gamma)$ be the solution of the basic inverse equation
\be\label{2.35}
\mu^{mod}=I+C_{w^{mod}}\mu^{mod}.
\ee
Then
\be\label{2.36}
M^{mod}(x,t;k)=I+\int_\Gamma\frac{(\mu^{mod}w^{mod})(x,t;s)}{s-k}\frac{\dd s}{2\pi\ii},\quad k\in\bfC\setminus\Gamma
\ee
is the unique solution of RH problem 2.3.

For a small enough constant $\varepsilon>0$, we define
\bea
L_\varepsilon&:=&\bigg\{k\in\bfC|k=k_0+k_0\alpha\e^{\frac{5\pi\ii}{6}},
\varepsilon<\alpha\leq\frac{2\sqrt{3}}{3}\bigg\}\nn\\
&&\cup\{k\in\bfC|k=-k_0+k_0\alpha\e^{\frac{\pi\ii}{6}},
\varepsilon<\alpha\leq\frac{2\sqrt{3}}{3}\bigg\}.
\eea
Denote $\Gamma_9=\frac{\ii k_0}{\sqrt{3}}(-1,1)$. Set $\Gamma'=\Gamma\setminus(\Gamma_9\cup L_\varepsilon\cup \bar{L}_\varepsilon)$ with the orientation as in Fig. \ref{fig4}. In the following, we shall reduce the RH problem 2.3 to a RH problem on the truncated contour $\Gamma'$ with controlled error terms. Let $w^{mod}=w^e+w'$, where $w'=w^{mod}\mid_{\Gamma'}$ and $w^e=w^{mod}\mid_{(\Gamma_9\cup L_\varepsilon\cup \bar{L}_\varepsilon)}$.

\begin{figure}[htbp]
  \centering
  \includegraphics[width=4in]{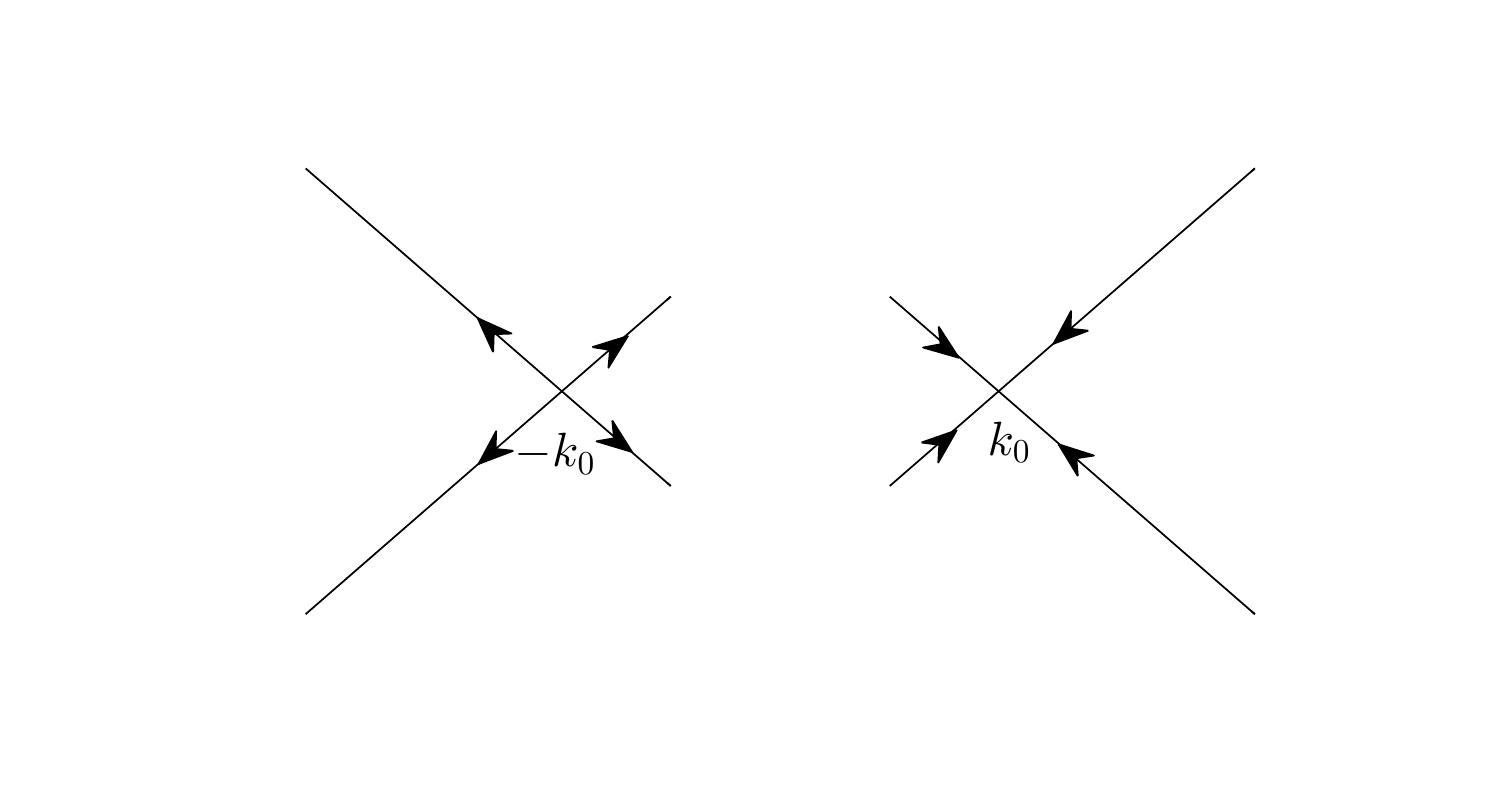}
  \caption{The oriented contour $\Gamma'$.}\label{fig4}
\end{figure}
\begin{lemma}\label{lem2.2}
For sufficiently small $\varepsilon>0$, we have
\bea
\|w^e\|_{L^n(\Gamma_9\cup L_\varepsilon\cup\bar{L}_\varepsilon)}&\leq&c\e^{-32\varepsilon^2\tau},\quad1\leq n\leq\infty,\label{2.38}\\
\|w'\|_{L^\infty(\Gamma')}&\leq&c\e^{-32\varepsilon^2\tau}.\label{2.39}
\eea
Furthermore,
\be
\|w'\|_{L^1(\Gamma')}\leq c\tau^{-\frac{1}{2}},~\|w'\|_{L^2(\Gamma')}\leq c\tau^{-\frac{1}{4}}.\label{2.40}
\ee
\end{lemma}
\begin{proof}
For $k\in L_\varepsilon$, we have
\bea
\text{Re}\Phi(k)&=&32\alpha^2k_0^5\bigg(5\sqrt{3}-10u+\frac{5\sqrt{3}}{2}u^2
-\frac{1}{2}u^3\bigg)\nn\\
&\geq&32\alpha^2k_0^5\bigg(5\sqrt{3}+(5\sqrt{3}-10)u-\frac{5\sqrt{3}}{2}
-\frac{1}{2}u^3\bigg)\geq32\alpha^2k_0^5.
\eea
Thus, we find
\be
|R_3^\pm\e^{-t\Phi}|\leq c\e^{-32\alpha^2tk_0^5}\leq c\e^{-32\varepsilon^2\tau}.
\ee
This yields \eqref{2.38}. On the other hand, a simple calculation implies
\be
\|\e^{-32\alpha^2tk_0^5}\|_{L^1(\Gamma')}\leq c\tau^{-\frac{1}{2}},\quad \|\e^{-32\alpha^2tk_0^5}\|_{L^2(\Gamma')}\leq c\tau^{-\frac{1}{4}}.
\ee
The estimate \eqref{2.40} immediately follows. $\hfill\Box$
\end{proof}
\begin{lemma}\label{lem2.3}
In the case $0<k_0\leq M$, as $\tau\rightarrow\infty$, that is, in Region I, $(1-C_{w'})^{-1}:L^2(\Gamma)\rightarrow L^2(\Gamma)$ exists and is uniformly bounded:
\berr
\|(1-C_{w'})^{-1}\|_{L^2(\Gamma)}\leq c.
\eerr
Moreover,
\berr
\|(1-C_{w^{mod}})^{-1}\|_{L^2(\Gamma)}\leq c.
\eerr
\end{lemma}
\begin{proof}
See \cite{DZ1993} and references therein.$\hfill\Box$
\end{proof}
A simple computation implies that
\bea
((1-C_{w^{mod}})^{-1}I)w^{mod}
&=&((1-C_{w'})^{-1}I)w'+w^e+((1-C_{w'})^{-1}(C_{w^e}I))w^{mod}\nn\\
&&+((1-C_{w'})^{-1}(C_{w'}I))w^e\nn\\
&&+((1-C_{w'})^{-1}C_{w^e}(1-C_{w^{mod}})^{-1})(C_{w^{mod}}I)w^{mod}.\nn
\eea
However, it follows from Lemma \ref{lem2.3} that
\bea
\|w^e\|_{L^1(\Gamma)}&\leq& c\e^{-32\varepsilon^2\tau},\nn\\
\|((1-C_{w'})^{-1}(C_{w^e}I))w^{mod}\|_{L^1(\Gamma)}&\leq&\|(1-C_{w'})^{-1}\|_{L^2(\Gamma)}
\|C_{w^e}I\|_{L^2(\Sigma)}\|w^{mod}\|_{L^2(\Gamma)}\nn\\
&\leq& c\|w^e\|_{L^2(\Sigma)}\|w^{mod}\|_{L^2(\Gamma)}\leq c\e^{-32\varepsilon^2\tau}\tau^{-\frac{1}{4}},\nn\\
\|(1-C_{w'})^{-1}(C_{w'}I))w^e\|_{L^1(\Gamma)}
&\leq&\|((1-C_{w'})^{-1}\|_{L^2(\Sigma)}
\|C_{w'}I\|_{L^2(\Gamma)}\|w^e\|_{L^2(\Gamma)}\nn\\
&\leq& c\|w'\|_{L^2(\Gamma)}\|w^e\|_{L^2(\Gamma)}\leq c\e^{-32\varepsilon^2\tau}\tau^{-\frac{1}{4}},\nn\\
\|((1-C_{w'})^{-1}C_{w^e}(1-C_{w^{mod}})^{-1})(C_{w^{mod}}I)w^{mod}\|_{L^1(\Gamma)}&\leq& c\|w^e\|_{L^\infty(\Gamma)}\|w^{mod}\|^2_{L^2(\Gamma)}\leq c\e^{-32\varepsilon^2\tau}\tau^{-\frac{1}{2}}.\nn
\eea
Thus, as $\tau\rightarrow\infty$, we have the following result:
\bea\label{2.44}
\begin{aligned}
\int_\Gamma\bigg(((1-C_{w^{mod}})^{-1}I)w^{mod}\bigg)(x,t;s)\dd s
&=\int_{\Gamma'}\bigg(((1-C_{w'})^{-1}I)w'\bigg)(x,t;s)\dd s+O(e^{-c\tau}).
\end{aligned}
\eea
On $\Gamma'$, set $\mu'=(1-C_{w'})^{-1}I$. Then it follows that
\be\label{2.45}
M'(x,t;k)=I+\int_{\Gamma'}\frac{(\mu'w')(x,t;s)}{s-k}\frac{\dd s}{2\pi\ii}
\ee
satisfies the following RH problem:\\
\textbf{Riemann--Hilbert problem 2.4.} Find a $2\times2$ matrix-valued function $M'(x,t;k)$ satisfying:

1. $M'(x,t;k)$ is analytic for $k\in\bfC\setminus\Gamma'$ and is continuous for $k\in\Gamma'$.

2. The boundary values $M_\pm'(x,t;k)$ satisfy the jump condition
\bea
M_+'(x,t;k)=M_-'(x,t;k)J'(x,t;k),\quad k\in\Gamma',\label{2.46}
\eea
where $J'(x,t;k)=J^{mod}(x,t;k)\mid_{\Gamma'}$.

3. $M'(x,t;k)=I+O\big(\frac{1}{k}\big),$ as $k\rightarrow\infty.$

Split $\Gamma'$ into a union of two disjoint crosses, $\Gamma'=\Gamma'_{-k_0}\cup\Gamma'_{k_0}$, where $\Gamma'_{\pm k_0}$ denote the cross of $\Gamma'$ centered at $\pm k_0$. Write $w'=w'_{-k_0}+w'_{k_0}$, where $w'_{-k_0}=0$ for $k\in\Gamma_{k_0}'$ and $w'_{k_0}=0$ for $k\in\Gamma_{-k_0}'$. Define the operators $C_{w'_{-k_0}}$ and $C_{w'_{k_0}}$ as in definition \eqref{2.34}. Then, analogous to Lemma 3.5 in \cite{DZ1993}, we obtain
\bea\label{2.47}
\begin{aligned}
\|C_{w'_{-k_0}}C_{w'_{k_0}}\|_{L^2(\Gamma')}=\|C_{w'_{k_0}}C_{w'_{-k_0}}\|_{L^2(\Gamma')}
&\leq c\tau^{-\frac{1}{2}},\\
\|C_{w'_{k_0}}C_{w'_{-k_0}}\|_{L^\infty(\Gamma')\rightarrow L^2(\Gamma')}&\leq c\tau^{-\frac{3}{4}},\\
\|C_{w'_{-k_0}}C_{w'_{k_0}}\|_{L^\infty(\Gamma')\rightarrow L^2(\Gamma')}&\leq c\tau^{-\frac{3}{4}}.
\end{aligned}
\eea
From the identity
\bea
&&(1-C_{w'_{-k_0}}-C_{w'_{k_0}})\bigg(1+C_{w'_{-k_0}}(1-C_{w'_{-k_0}})^{-1}
+C_{w'_{k_0}}(1-C_{w'_{k_0}})^{-1}\bigg)\nn\\
&&=1-C_{w'_{k_0}}C_{w'_{-k_0}}(1-C_{w'_{-k_0}})^{-1}-
C_{w'_{-k_0}}C_{w'_{k_0}}(1-C_{w'_{k_0}})^{-1},\nn
\eea
it follows that
\bea
(1-C_{w'})^{-1}&=&1+C_{w'_{-k_0}}(1-C_{w'_{-k_0}})^{-1}+C_{w'_{k_0}}(1-C_{w'_{k_0}})^{-1}\nn\\
&&+\bigg[1+C_{w'_{-k_0}}(1-C_{w'_{-k_0}})^{-1}+C_{w'_{k_0}}(1-C_{w'_{k_0}})^{-1}\bigg]
\bigg[1-C_{w'_{k_0}}C_{w'_{-k_0}}(1-C_{w'_{-k_0}})^{-1}\nn\\
&&-C_{w'_{-k_0}}C_{w'_{k_0}}(1-C_{w'_{k_0}})^{-1}\bigg]
^{-1}\bigg[C_{w'_{k_0}}C_{w'_{-k_0}}(1-C_{w'_{-k_0}})^{-1}+
C_{w'_{-k_0}}C_{w'_{k_0}}(1-C_{w'_{k_0}})^{-1}\bigg].\nn
\eea
According to Lemma \ref{lem2.2}, \ref{lem2.3} and estimate \eqref{2.47}, proceeding the estimates as in \cite{DZ1993}, we obtain an important result stated as follows:
\begin{lemma}\label{lem2.4}
As $\tau\rightarrow\infty$,
\bea\label{2.48}
\int_{\Gamma'}\bigg(((1-C_{w'})^{-1}I)w'\bigg)(x,t;s)\dd s&=&\int_{\Gamma'_{-k_0}}\bigg(((1-C_{w'_{-k_0}})^{-1}I)w'_{-k_0}\bigg)(x,t;s)\dd s\\
&&+\int_{\Gamma'_{k_0}}\bigg(((1-C_{w'_{k_0}})^{-1}I)w'_{k_0}\bigg)(x,t;s)\dd s+O(\tau^{-1}).\nn
\eea
\end{lemma}

In the following, we aim to perform a scaling transform and then formulate the model RH problem. Extend the contours $\Gamma'_{-k_0}$ and $\Gamma'_{k_0}$ to the contours
\bea
\hat{\Gamma}'_{-k_0}&=&\{k=-k_0+k_0\alpha\e^{\pm\frac{\ii\pi}{6}}:-\infty<\alpha<\infty\},\nn\\
\hat{\Gamma}'_{k_0}&=&\{k=k_0+k_0\alpha\e^{\pm\frac{5\ii\pi}{6}}:-\infty<\alpha<\infty\},\nn
\eea
respectively, and define $\hat{w}'_{-k_0}$, $\hat{w}'_{k_0}$ on $\hat{\Gamma}'_{-k_0}$, $\hat{\Gamma}'_{k_0}$, respectively, through
\berr
\hat{w}'_{-k_0}=\left\{
\begin{aligned}
&w'_{-k_0},\quad k\in\Gamma'_{-k_0}\subset\hat{\Gamma}'_{-k_0},\\
&0,~~\quad\quad k\in\hat{\Gamma}'_{-k_0}\setminus\Gamma'_{-k_0},
\end{aligned}
\right.\qquad
\hat{w}'_{k_0}=\left\{
\begin{aligned}
&w'_{k_0},\quad k\in\Gamma'_{k_0}\subset\hat{\Gamma}'_{k_0},\\
&0,\quad~~~ k\in\hat{\Gamma}'_{k_0}\setminus\Gamma'_{k_0}.
\end{aligned}
\right.
\eerr
Let $\Gamma_{-k_0}$ and $\Gamma_{k_0}$ denote the contours $\{k=\alpha\e^{\pm\frac{\ii\pi}{6}}:-\infty<\alpha<\infty\}$ centered at original point and oriented as shown in Fig. \ref{fig5}.
\begin{figure}[htbp]
  \centering
  \includegraphics[width=5in]{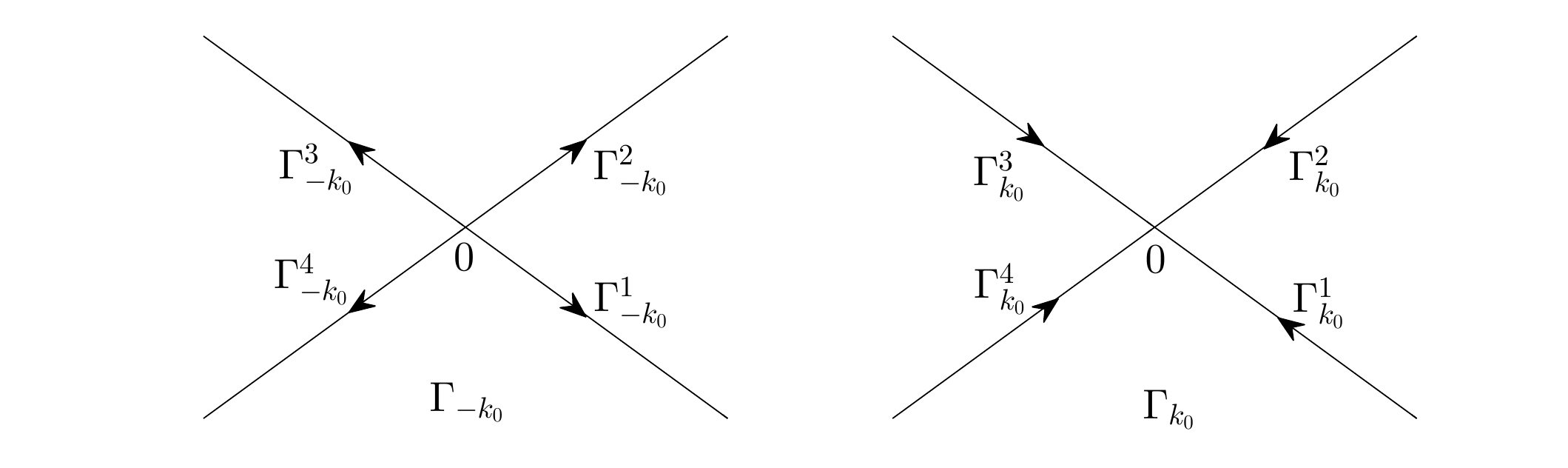}
  \caption{The contours $\Gamma_{-k_0}$ and $\Gamma_{k_0}$.}\label{fig5}
\end{figure}

We next introduce the following scaling operators
\bea
&&S_{-k_0}:~k\mapsto\frac{z}{8k_0\sqrt{10k_0 t}}-k_0,\label{2.49}\\
&&S_{k_0}:~k\mapsto\frac{z}{8k_0\sqrt{10k_0 t}}+k_0.\label{2.50}
\eea
Set $w_{-k_0}=S_{-k_0}\hat{w}'_{-k_0}$, $w_{k_0}=S_{k_0}\hat{w}'_{k_0}$. A simple change of variables argument shows that
\berr
C_{\hat{w}'_{-k_0}}=S_{-k_0}^{-1}C_{w_{-k_0}}S_{-k_0},\quad C_{\hat{w}'_{k_0}}=S_{k_0}^{-1}C_{w_{k_0}}S_{k_0},
\eerr
where $C_{w_{\pm k_0}}$ is a bounded map from $\Gamma_{\pm k_0}$ to $\Gamma_{\pm k_0}$. On the part
\berr
L_{k_0}:\{z=8k_0^2\sqrt{10k_0t}\alpha\e^{-\frac{\ii\pi}{6}}:-\varepsilon<\alpha<\infty\},
\eerr
we have
\berr
w_{k_0}=w_{k_0+}=\left\{
\begin{aligned}
&\begin{pmatrix}
0 ~& -\frac{\overline{r(k_0)}}{1-|r(k_0)|^2}(\delta_{k_0}^0)^{-1}(\delta_{k_0}^1)^{-1}\\[4pt]
0 ~& 0
\end{pmatrix},\\
&\begin{pmatrix}
0 ~& \overline{r(k_0)}(\delta_{k_0}^0)^{-1}(\delta_{k_0}^1)^{-1}\\[4pt]
0 ~& 0
\end{pmatrix},
\end{aligned}
\right.
\eerr
and on $\bar{L}_{k_0}$,
\berr
w_{k_0}=w_{k_0-}=\left\{
\begin{aligned}
&\begin{pmatrix}
0 ~& 0\\[4pt]
\frac{r(k_0)}{1-|r(k_0)|^2}\delta_{k_0}^0\delta_{k_0}^1 ~& 0
\end{pmatrix},\\
&\begin{pmatrix}
0 ~& 0\\[4pt]
-r(k_0)\delta_{k_0}^0\delta_{k_0}^1 ~& 0
\end{pmatrix},
\end{aligned}
\right.
\eerr
where
\bea
\delta_{k_0}^0&=&\e^{-2\chi(k_0)-128\ii k_0^5t}(2560k_0^5t)^{\ii\nu},\\
\delta_{k_0}^1&=&z^{-2\ii\nu}\exp\bigg\{\frac{\ii z^2}{2}\bigg(1+\frac{z}{8k_0^2\sqrt{10 k_0t}}+\frac{z^2}{1280k_0^5t}+\frac{z^3}{51200k_0^7t\sqrt{10\epsilon k_0t}}\bigg)\bigg\}\nn\\
&&\times\bigg(\frac{2k_0}{z/8k_0\sqrt{10k_0t}+2k_0}\bigg)^{-2\ii\nu}.
\eea
If we set $J^{k_0}(z;r(k_0))=(I-w_-^{k_0})^{-1}(I+w_+^{k_0})$, where
\bea\label{2.53}
w^{k_0}=w^{k_0}_+=\left\{
\begin{aligned}
&\begin{pmatrix}
0 ~& (\delta_{k_0}^0)^{-1}z^{2\ii\nu}\e^{-\frac{\ii z^2}{2}}\overline{r(k_0)}\\[4pt]
0 ~& 0
\end{pmatrix},\qquad~~~ z\in \Gamma_{k_0}^1,\\
&\begin{pmatrix}
0 ~& -(\delta_{k_0}^0)^{-1}z^{2\ii\nu}\e^{-\frac{\ii z^2}{2}}\frac{\overline{r(k_0)}}{1-|r(k_0)|^2}\\[4pt]
0 ~& 0
\end{pmatrix},\quad z\in\Gamma_{k_0}^3,
\end{aligned}
\right.
\eea
\bea\label{2.54}
w^{k_0}=w^{k_0}_-=\left\{
\begin{aligned}
&\begin{pmatrix}
0 ~& 0\\[4pt]
-\delta_{k_0}^0z^{-2\ii\nu}\e^{\frac{\ii z^2}{2}}r(k_0) ~& 1
\end{pmatrix},\qquad~ z\in\Gamma_{k_0}^2,\\
&\begin{pmatrix}
0 ~& 0\\[4pt]
\delta_{k_0}^0z^{-2\ii\nu}\e^{\frac{\ii z^2}{2}}\frac{r(k_0)}{1+|r(k_0)|^2} ~& 0
\end{pmatrix},\quad~~ z\in\Gamma_{k_0}^4.
\end{aligned}
\right.
\eea
\begin{lemma}\label{lem2.5}
Let $0<\kappa<1/2$ be a fixed small number, then we have as $\tau\rightarrow\infty$
\be\label{2.55}
\|\delta^1_{k_0}-z^{-2\ii\nu}\e^{\frac{\ii z^2}{2}}\|_{L^\infty(\bar{L}_{k_0})}\leq c|\e^{\frac{\ii\kappa z^2}{2}}|\tau^{-\frac{1}{2}}.
\ee
As a result, we have
\bea\label{2.56}
\|\delta^1_{k_0}-z^{-2\ii\nu}\e^{\frac{\ii z^2}{2}}\|_{(L^\infty\cap L^1\cap L^2)(\bar{L}_{k_0})}\leq c\tau^{-\frac{1}{2}}.
\eea
\end{lemma}
\begin{proof}
Let
\begin{align}
\triangle=\frac{z}{8k_0^2\sqrt{10k_0t}(1-2\kappa)}+\frac{z^2}{1280 k_0^5t(1-2\kappa)}+\frac{z^3}{51200k_0^7t\sqrt{10k_0t}(1-2\kappa)}.\nn
\end{align}
A direct calculation yieds
\bea
|\delta^1_{k_0}-z^{-2\ii\nu}\e^{\frac{\ii z^2}{2}}|&\leq&c|\e^{\frac{\ii\kappa z^2}{2}}|\bigg\{\bigg|\e^{\frac{\ii\kappa z^2}{2}}\bigg[\bigg(\frac{2k_0}{z/8k_0\sqrt{10k_0t}+2k_0}\bigg)^{-2\ii\nu}-1\bigg]\bigg|\nn\\
&&+\bigg|
\exp\bigg(\frac{\ii(1-2\kappa)z^2}{2}(1+\triangle)\bigg)-\e^{\frac{\ii(1-2\kappa) z^2}{2}}\bigg|\bigg\}\nn\\
&\leq&c|\e^{\frac{\ii\kappa z^2}{2}}|\tau^{-\frac{1}{2}}.\nn
\eea$\hfill\Box$
\end{proof}
Therefore, we have
\berr
\begin{aligned}
&\int_{\Gamma'_{k_0}}\bigg(((1-C_{w'_{k_0}})^{-1}I)w'_{k_0}\bigg)(s)\dd s=\int_{\hat{\Gamma}'_{k_0}}\bigg(((1-C_{\hat{w}'_{k_0}})^{-1}I)\hat{w}'_{k_0}\bigg)(s)\dd s\\
&=\int_{\hat{\Gamma}'_{k_0}}(S_{k_0}^{-1}(1-C_{w_{k_0}})^{-1}S_{k_0}I)(s)
\hat{w}'_{k_0}(s)\dd s\\
&=\int_{\hat{\Gamma}'_{k_0}}((1-C_{w_{k_0}})^{-1}I)\bigg((s-k_0)8k_0\sqrt{10k_0t}\bigg)
(S_{k_0}\hat{w}'_{k_0})\bigg((s-k_0)8k_0\sqrt{10k_0t}\bigg)\dd s\\
&=\frac{1}{8k_0\sqrt{10k_0t}}
\int_{\Gamma_{k_0}}\bigg(((1-C_{w_{k_0}})^{-1}I)w_{k_0}\bigg)(s)\dd s\\
&=\frac{1}{8k_0\sqrt{10k_0t}}\int_{\Gamma_{k_0}}\bigg(((1-C_{w^{k_0}})^{-1}I)w^{k_0}\bigg)(s)\dd s+O(\tau^{-1}).
\end{aligned}
\eerr
Together with a similar computation for $C_{w'_{-k_0}}$ and \eqref{2.44}, \eqref{2.48}, we have obtained as $\tau\rightarrow\infty$
\bea\label{2.57}
\int_\Gamma\bigg(((1-C_{w^{mod}})^{-1}I)w^{mod}\bigg)(s)\dd s&=&\frac{1}{8k_0\sqrt{10k_0t}}\bigg(\int_{\Gamma_{-k_0}}\bigg(((1-C_{w^{-k_0}})^{-1}I)
w^{-k_0}\bigg)(s)\dd s\nn\\
&&+\int_{\Gamma_{k_0}}\bigg(((1-C_{w^{k_0}})^{-1}I)
w^{k_0}\bigg)(s)\dd s\bigg)+O(\tau^{-1}).
\eea
For $z\in\bfC\setminus\Gamma_{k_0}$, let
\be
M^{k_0}(z)=I+\int_{\Gamma_{k_0}}\frac{\bigg(((1-C_{w^{k_0}})^{-1}I)w^{k_0}\bigg)(s)}{s-z}
\frac{\dd s}{2\pi\ii},
\ee
then $M^{k_0}(z)$ satisfies the following RH problem:\\
\textbf{Riemann--Hilbert problem 2.5.} Find a $2\times2$ matrix-valued function $M^{k_0}(z)$ with the following properties:

1. $M^{k_0}(z)$ is analytic for $z\in\bfC\setminus\Gamma_{k_0}$ and is continuous for $z\in\Gamma_{k_0}$.

2. The boundary values $M_\pm^{k_0}(z)$ satisfy the jump condition
\bea
M_+^{k_0}(z)&=M_-^{k_0}(z)J^{k_0}(z;r(k_0)),\quad z\in\Gamma_{k_0}.\label{2.59}
\eea

3. $M^{k_0}(z)=I+O\big(\frac{1}{z}\big),$ as $z\rightarrow\infty.$

In particular, if we assume
\be\label{2.60}
M^{k_0}(z)=I+\frac{M_1^{k_0}}{z}+O(z^{-2}),\quad z\rightarrow\infty,
\ee
then
\be\label{2.61}
M_1^{k_0}=-\int_{\Gamma_{k_0}}\bigg(((1-C_{w^{k_0}})^{-1}I)w^{k_0}\bigg)(s)\frac{\dd s}{2\pi\ii}.
\ee
There is a analogous RH problem on $\Gamma_{-k_0}$, which satisfies
\be\label{2.62}
\left\{
\begin{aligned}
M_+^{-k_0}(z)&=M_-^{-k_0}(z)J^{-k_0}(z;r(-k_0)),\quad z\in\Gamma_{-k_0},\\
M^{-k_0}(z)&\rightarrow I,\qquad\qquad~~~\qquad\qquad\qquad z\rightarrow\infty.
\end{aligned}
\right.
\ee
According \eqref{2.30}, $r(-k)=\overline{r(k)}$ and the analogously computation for $w^{-k_0}$, one can find that
\be
J^{-k_0}(z;r(-k_0))=\overline{J^{k_0}(-\bar{z};r(k_0))},
\ee
which in turn implies, by uniqueness, that
\be\label{2.64}
M^{-k_0}(z)=\overline{M^{k_0}(-\bar{z})},\quad M^{-k_0}_1=-\overline{M_1^{k_0}}.
\ee
On the other hand, it follows from \cite{DZ1993,DZ1995} that the solution $M^{k_0}(z)$ of RH problem 2.5 can be explicitly solved in terms of parabolic cylinder functions
\be\label{2.65}
M^{k_0}(z)=I+\frac{\ii}{z}
\begin{pmatrix}
0 &~ -\beta(r(k_0))(\delta^1_{k_0})^{-1}\\
\overline{\beta(r(k_0))}\delta^1_{k_0} &~ 0
\end{pmatrix}+O\bigg(\frac{1}{z^2}\bigg),
\ee
where the function $\beta(r(k_0))$ is defined by
\be
\beta(r(k_0))=\sqrt{\nu}\e^{\ii(\frac{\pi}{4}-\arg r(k_0)+\arg\Gamma(\ii\nu))},
\ee
and $\Gamma(\cdot)$ denotes the standard Gamma function. Thus, we have find that
\be\label{2.67}
M_1^{k_0}=\ii\begin{pmatrix}
0 &~ -\beta(r(k_0))(\delta^0_{k_0})^{-1}\\
\overline{\beta(r(k_0))}\delta^0_{k_0} &~ 0
\end{pmatrix}.
\ee

Taking into account that \eqref{2.36}, \eqref{2.57}, \eqref{2.60}, \eqref{2.64} and \eqref{2.65}, we get
\be\label{2.68}
M^{mod}(x,t;k)=I+\frac{M^{mod}_1}{k}+O\bigg(\frac{1}{k^2}\bigg),\quad k\rightarrow\infty,
\ee
where
\be\label{2.69}
M^{mod}_1=\frac{1}{8k_0\sqrt{10k_0t}}(2\ii\text{Im}M^{k_0}_1)+O(\tau^{-1}),~\text{as} ~\tau\rightarrow\infty.
\ee
\subsection{Analysis of the remaining $\bar{\partial}$ problem}
Recalling the transform \eqref{2.32}, it is easy to verify that $M^{(3)}(k)$ is a continuously differentiable function satisfying the following pure $\bar{\partial}$ problem.\\
\textbf{$\bar{\partial}$ problem 2.6.} Find a function $M^{(3)}(k)$ with the following properties:

1. $M^{(3)}(k)$ is continuous with sectionally continuous first partial derivatives in $\bfC\setminus\Gamma$.

2. For $k\in\bfC\setminus\Gamma$, we have
\bea
\begin{aligned}
\bar{\partial}M^{(3)}(k)&=M^{(3)}(k)W^{(3)}(k),\\
W^{(3)}(k)&=M^{mod}(k)\bar{\partial}\mathcal{R}^{(2)}(k)[M^{mod}(k)]^{-1}.
\end{aligned}
\eea

3. $M^{(3)}(k)$ admits asymptotics: $M^{(3)}(k)=I+O\big(\frac{1}{k}\big),~k\rightarrow\infty.$

It follows from \cite{AF} that the $\bar{\partial}$ problem 2.6 is equivalent to the integral equation
\be\label{2.71}
M^{(3)}(k)=I-\frac{1}{\pi}\iint\limits_{\bfC}\frac{M^{(3)}(s)W^{(3)}(s)}{s-k}\dd A(s),
\ee
where $\dd A(s)$ is Lebesgue measure on the plane. Equation \eqref{2.71} can be rewritten as the following operator form
\be
(1-K_W)[M^{(3)}(k)]=I,
\ee
where the integral operator $K_W$ is defined by
\be\label{2.73}
(K_Wf)(k)=-\frac{1}{\pi}\iint\limits_{\bfC}\frac{f(s)W^{(3)}(s)}{s-k}\dd A(s).
\ee
\begin{proposition}
There exists a constant $C>0$ such that for large $t>0$, the operator \eqref{2.73} obeys the estimate
\be\label{2.74}
\|K_W\|_{L^\infty\rightarrow L^\infty}\leq C(k_0^3t)^{-\frac{1}{4}}.
\ee
\end{proposition}
\begin{proof}
We only discuss the case in $\Omega_1$. We write $s=k_0+u+\ii v$, then $u\geq\sqrt{3}v\geq0$. Note that $$
\text{Re}\Phi=32v\bigg(10u^2v^2+20uv^2k_0+10k_0^2v^2-5u^4-
20u^3k_0-20uk_0^3-30u^2k_0^2-v^4\bigg)\leq-640k_0^3uv.$$
Set $k=\alpha+\ii\beta$. Let $f\in L^\infty(\Omega_1)$, we have
\bea\label{2.75}
|(K_Wf)(k)|\leq\|f\|_{L^\infty(\Omega_1)}
\iint\limits_{\Omega_1}\frac{|W^{(3)}(s)|}{|s-k|}\dd A(s)
\leq C\|f\|_{L^\infty(\Omega_1)}\iint\limits_{\Omega_1}
\frac{|\bar{\partial}R_1(s)|e^{-640k_0^3tuv}}{|s-k|}\dd A(s).
\eea
Thus, we find
\be
\|K_W\|_{L^\infty\rightarrow L^\infty}\leq C(I_1+I_2),
\ee
where
\bea
I_1&=&\int_0^\infty\int_{\sqrt{3}v}^\infty\frac{1}{|s-k|}
|r'(u+k_0)|e^{-640k_0^3tuv}\dd u\dd v,\nn\\
I_2&=&\int_0^\infty\int_{\sqrt{3}v}^\infty\frac{|u+\ii v|^{-\frac{1}{2}}}{|s-k|}e^{-640k_0^3tuv}\dd u\dd v.\nn
\eea
Recalling the estimates from \cite{DM}, we can similarly get
\be
|I_1|,|I_2|\leq C(k_0^3t)^{-\frac{1}{4}}.
\ee$\hfill\Box$
\end{proof}
To recover the long-time asymptotic behavior of $u(x,t)$ using \eqref{1.21}, it is necessary to determine the asymptotic behavior of the coefficient of the $k^{-1}$ term in the Laurent expansion of $M^{(3)}$ at infinity. In fact, we have
\be\label{2.78}
M^{(3)}(k)=I-\frac{1}{\pi}\iint\limits_{\bfC}\frac{M^{(3)}(s)W^{(3)}(s)}{s-k}\dd A(s)=I+\frac{M_1^{(3)}}{k}+\frac{1}{\pi}
\iint\limits_{\bfC}\frac{sM^{(3)}(s)W^{(3)}(s)}{k(s-k)}\dd A(s),
\ee
where
\be\label{2.79}
M_1^{(3)}=\frac{1}{\pi}\iint\limits_{\bfC}M^{(3)}(s)W^{(3)}(s)\dd A(s).
\ee
\begin{proposition}
For all large $t>0$, there exists a constant $C>0$ such that
\be\label{2.80}
|M_1^{(3)}|\leq C(k_0^3t)^{-\frac{3}{4}}.
\ee
\end{proposition}
\begin{proof}
We estimate the integral \eqref{2.80} as follows:
\bea
|M_1^{(3)}|&\leq& C\iint\limits_{\Omega_1}|\bar{\partial}R_1(s)|e^{-640k_0^3tuv}\dd A(s)\leq C\bigg(\int_0^\infty\int_{\sqrt{3}v}^\infty\bigg|r'(u+k_0)\bigg|e^{-640k_0^3tuv}\dd u\dd v\nn\\
&&+\int_0^\infty\int_{\sqrt{3}v}^\infty\frac{1}{|u+\ii v|^{\frac{1}{2}}}e^{-640k_0^3tuv}\dd u\dd v\bigg)\leq C(I_3+I_4).\nn
\eea
It follows from the Proposition D.2 in \cite{BJM} that $I_3$ and $I_4$ satisfy
\bea
|I_3|,|I_4|\leq C(k_0^3t)^{-\frac{3}{4}}.
\eea$\hfill\Box$
\end{proof}
\subsection{Asymptotics for $u(x,t)$}
We are now ready to find the asymptotic behavior of the solution $u(x,t)$ to fifth-order modified KdV equation \eqref{5mKdV} in Region I. Working through the transformations, we have
\be
\begin{aligned}
M(x,t;k)&=M^{(3)}(k)M^{mod}(k)[\mathcal{R}^{(2)}(k)]^{-1}\delta^{\sigma_3}(k).
\end{aligned}
\ee
Using the reconstruction formula \eqref{1.21} and \eqref{2.67}-\eqref{2.69}, \eqref{2.78} and \eqref{2.80}, we immediately find
\bea
u(x,t)=\frac{\sqrt{\nu}}{2k_0\sqrt{10k_0t}}\cos\bigg(128tk_0^5+\nu\ln(2560 tk_0^5)+\varphi(k_0)\bigg)+O\bigg(\tau^{-1}+(k_0^3t)^{-\frac{3}{4}}\bigg),
\eea
where
\begin{align}
\varphi(k_0)&=-\frac{3\pi}{4}-\arg r(k_0)+\arg\Gamma(\ii\nu)-
\frac{1}{\pi}\int_{-k_0}^{k_0}\ln\bigg(\frac{1-|r(s)|^2}{1-|r(k_0)|^2}\bigg)\frac{\dd s}{s-k_0},\nn\\
\nu&=-\frac{1}{2\pi}\ln(1-|r(k_0)|^2)>0.\nn
\end{align}

\section{Asymptotics in Region III}
\setcounter{equation}{0}
\setcounter{lemma}{0}
\setcounter{theorem}{0}

We now consider the asymptotics of $u(x,t)$ in the Region III: $\tau\leq M'$, that is, $|x|\leq80M'^{\frac{4}{5}}t^{\frac{1}{5}}$. First, for $x>0$, the function $\Phi(k)$ has two real stationary phase points
\be
\pm k_0=\pm\sqrt[4]{\frac{x}{80t}},
\ee
however, as $t\rightarrow\infty$, the critical points $\pm k_0$ approach 0 at least as fast as $t^{-\frac{1}{5}}$, i.e., $k_0\leq M'^{\frac{1}{5}}t^{-\frac{1}{5}}$. Note that the jump matrix $J(x,t;k)$ enjoys the factorization
\be\label{3.2}
J(x,t;k)=\begin{pmatrix}
1 &~ -\overline{r(k)}\e^{-t\Phi(k)}\\
0 &~ 1
\end{pmatrix}\begin{pmatrix}
1 &~ 0\\
r(k)\e^{t\Phi(k)}  &~ 1
\end{pmatrix}.
\ee
Introducing the following scaling transform:
\be\label{3.3}
k\rightarrow(20t)^{-\frac{1}{5}}z,
\ee
and
letting
\be
y=\frac{x}{(20t)^{\frac{1}{5}}}
\ee
we then have
\be\label{3.5}
J(x,t;z)=\begin{pmatrix}
1 &~ -\overline{r((20t)^{-\frac{1}{5}}z)}\e^{-2\ii(\frac{4}{5}z^5-yz)}\\
0 &~ 1
\end{pmatrix}\begin{pmatrix}
1 &~ 0\\
r((20t)^{-\frac{1}{5}}z)\e^{2\ii(\frac{4}{5}z^5-yz)}  &~ 1
\end{pmatrix}.
\ee

The first step also is to introduce the continuous but not necessarily analytic extensions. Let
\be
z_0=(20t)^{\frac{1}{5}}k_0.
 \ee
Define the new contour $\Sigma=\Sigma_1\cup\Sigma_2\cup\Sigma_3$, where the line segments
\be\label{3.6}
\begin{aligned}
\Sigma_1&=\{z_0+l\e^{\frac{\pi\ii}{6}}|l\geq0\}\cup\{-z_0+l\e^{\frac{5\pi\ii}{6}}|l\geq0\},\\
\Sigma_2&=\{z_0+l\e^{-\frac{\pi\ii}{6}}|l\geq0\}\cup\{-z_0+l\e^{-\frac{5\pi\ii}{6}}|l\geq0\},\\
\Sigma_3&=\{l|-z_0\leq l\leq z_0\}
\end{aligned}
\ee
oriented with increasing real part and denote the four open sectors $\{D_j\}_1^4$ in $\bfC$, see Fig. \ref{fig6}.
\begin{figure}[htbp]
 \centering
  \includegraphics[width=3in]{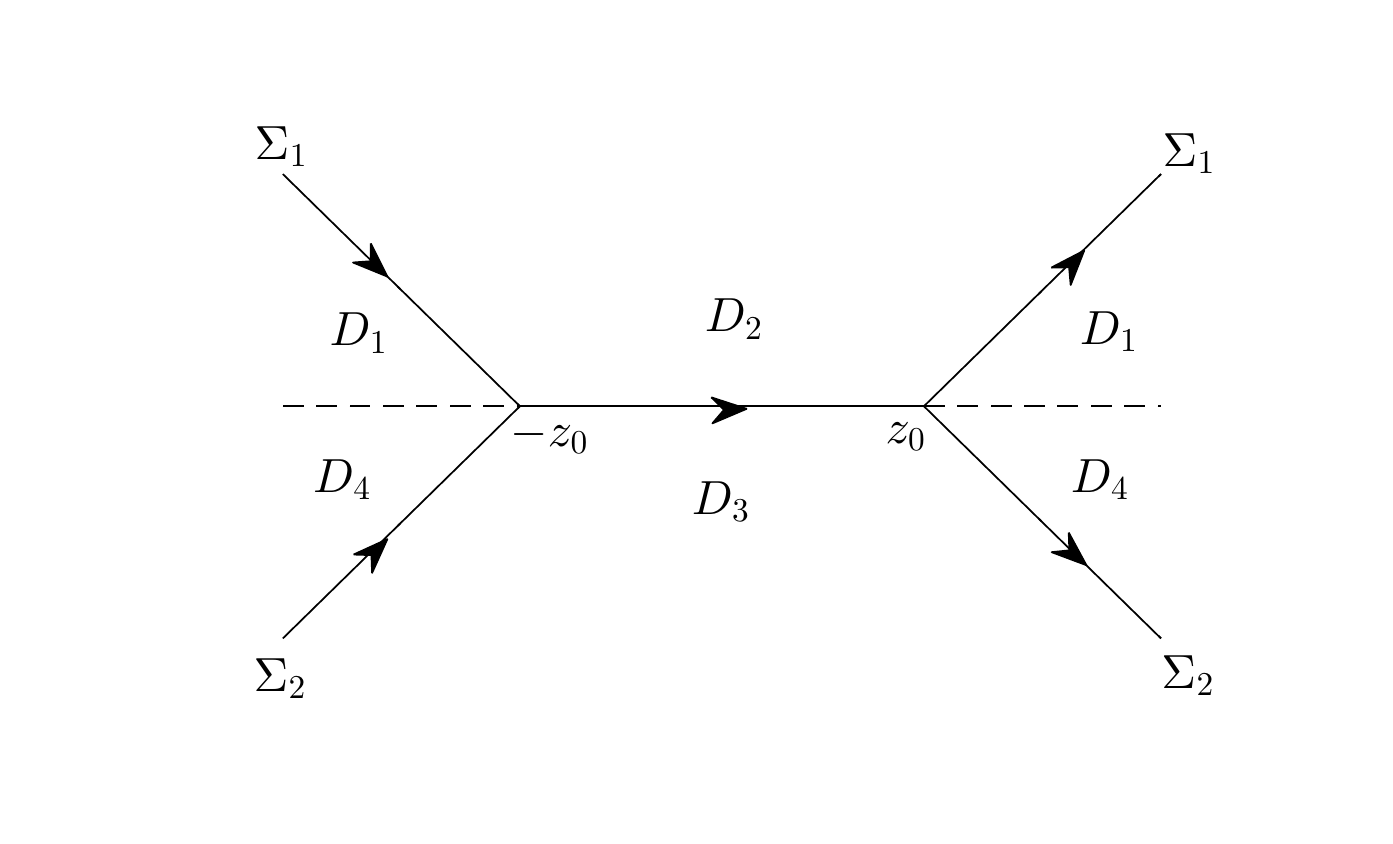}
  \caption{The contour $\Sigma$ and the open sets $\{D_j\}_1^4$ in the complex $z$-plane.}\label{fig6}
\end{figure}
\begin{lemma}\label{lem3.1}
There exist functions $R_j(z)$ on $\bar{D}_j$ for $j=1,4$ with boundary values satisfying
\bea
R_1(z)&=&\left\{
\begin{aligned}
&-r((20t)^{-\frac{1}{5}}z),\quad z\in(-\infty,-z_0)\cup(z_0,\infty),\\
&-r(k_0),\qquad\qquad z\in\Sigma_1,
\end{aligned}
\right.\label{3.8}\\
R_4(z)&=&\left\{
\begin{aligned}
&-\overline{r((20t)^{-\frac{1}{5}}z)},\quad z\in(-\infty,-z_0)\cup(z_0,\infty),\\
&-\overline{r(k_0)},\qquad\qquad z\in\Sigma_2,
\end{aligned}
\right.\label{3.9}
\eea
such that
\be\label{3.10}
|\bar{\partial}R_j(z)|\leq c_1t^{-\frac{1}{5}}|r'((20t)^{-\frac{1}{5}}\text{Re}z)|
+c_2t^{-\frac{1}{5}}|(20t)^{-\frac{1}{5}}z-k_0|^{-\frac{1}{2}}
\ee
for two positive constants $c_1,c_2$ depended on $\|r\|_{H^1(\bfR)}$.
\end{lemma}
\begin{proof}
We only consider $z\in\bar{\Omega}_1\cap\{\text{Re}z>z_0\}$. Define the extension
\bea
R_1(z)=-r(k_0)+\bigg(r(k_0)-r((20t)^{-\frac{1}{5}}\text{Re}z)\bigg)\cos(3\phi).
\eea
Let $z-z_0=s\e^{\ii\phi}$. It follows from $$\bar{\partial}=\frac{1}{2}\bigg(\frac{\partial}{\partial z_1}+\ii\frac{\partial}{\partial z_2}\bigg)
=\frac{1}{2}\e^{\ii\phi}\bigg(\frac{\partial}{\partial s}+\frac{\ii}{s}\frac{\partial}{\partial\phi}\bigg)$$
that
\bea
|\bar{\partial}R_1(z)|&=&\bigg|-\frac{1}{2}(20t)^{-\frac{1}{5}}
r'((20t)^{-\frac{1}{5}}\text{Re}z)\cos(3\phi)+
\frac{3\ii}{2}\e^{\ii\phi}\frac{r((20t)^{-\frac{1}{5}}\text{Re}z)-r(k_0)}
{|z-z_0|}\sin(3\phi)\bigg|\nn\\
&\leq&c_1t^{-\frac{1}{5}}|r'((20t)^{-\frac{1}{5}}\text{Re}z)|
+c_2t^{-\frac{1}{5}}|(20t)^{-\frac{1}{5}}z-k_0|^{-\frac{1}{2}}.
\eea$\hfill\Box$
\end{proof}

Next, we use the extensions in Lemma \ref{lem3.1} to define
\be\label{3.13}
M^{(1)}(y,t;z)=M(x,t;k)\mathcal{R}^{(1)}(z),
\ee
where
\be
\mathcal{R}^{(1)}(z)=\left\{
\begin{aligned}
&\begin{pmatrix}
1 &~ 0\\
R_1(z)\e^{2\ii(\frac{4}{5}z^5-yz)} &~ 1
\end{pmatrix},~\quad z\in D_1,\quad
\begin{pmatrix}
1 &~ 0\\
0 &~ 1
\end{pmatrix},\quad z\in D_2\cup D_3,\\
&\begin{pmatrix}
1 &~ R_4(z)\e^{-2\ii(\frac{4}{5}z^5-yz)}\\
0 &~ 1
\end{pmatrix},\quad z\in D_4.
\end{aligned}
\right.
\ee
Then it is an immediate consequence of Lemma \ref{lem3.1} and RH problem 1.1 that $M^{(1)}(y,t;z)$ satisfies the following $\bar{\partial}$-Riemann--Hilbert problem.\\
\textbf{$\bar{\partial}$-Riemann--Hilbert problem 3.1.} Find a function $M^{(1)}(y,t;z)$ with the following properties:

1. $M^{(1)}(y,t;z)$ is continuous with sectionally continuous first partial derivatives in $\bfC\setminus\Sigma$.

2. Across $\Sigma$, the boundary values satisfy the jump relation
\be
M^{(1)}_+(y,t;z)=M^{(1)}_-(y,t;z)J^{(1)}(y,t;z),~~z\in\Sigma,
\ee
where the jump matrix $J^{(1)}(y,t;z)$ is given by
\bea\label{3.16}
J^{(1)}(y,t;z)&=\left\{
\begin{aligned}
&\begin{pmatrix}
1 &~ 0\\
r(k_0)\e^{2\ii(\frac{4}{5}z^5-yz)} &~ 1
\end{pmatrix},\quad\quad~~ z\in\Sigma_1,\\
&\begin{pmatrix}
1 &~ -\overline{r(k_0)}\e^{-2\ii(\frac{4}{5}z^5-yz)}\\
0 &~ 1
\end{pmatrix},\quad z\in\Sigma_2,\\
&\begin{pmatrix}
1 &~ -\overline{r((20t)^{-\frac{1}{5}}z)}\e^{-2\ii(\frac{4}{5}z^5-yz)}\\
0 &~ 1
\end{pmatrix}\begin{pmatrix}
1 &~ 0\\
r((20t)^{-\frac{1}{5}}z)\e^{2\ii(\frac{4}{5}z^5-yz)}  &~ 1
\end{pmatrix}, \quad z\in\Sigma_3.
\end{aligned}
\right.
\eea

3. For $z\in\bfC\setminus\Sigma$, we have
\be\label{3.17}
\begin{aligned}
\bar{\partial}M^{(1)}(z)&=M^{(1)}(z)\bar{\partial}\mathcal{R}^{(1)}(z),\\
\bar{\partial}\mathcal{R}^{(1)}(z)&=\left\{
\begin{aligned}
&\begin{pmatrix}
1 &~ 0\\
\bar{\partial}R_1(z)\e^{2\ii(\frac{4}{5}z^5-yz)} &~ 1
\end{pmatrix},\quad z\in D_1,\\
&\begin{pmatrix}
1 &~ \bar{\partial}R_4(z)\e^{-2\ii(\frac{4}{5}z^5-yz)}\\
0 &~ 1
\end{pmatrix},~~ z\in D_4,\\
&0,\qquad\qquad\qquad\qquad\qquad\quad~ z\in D_2\cup D_3.
\end{aligned}
\right.
\end{aligned}
\ee

4. $M^{(1)}(y,t;z)$ enjoys asymptotics: $M^{(1)}(y,t;z)=I+O\big(\frac{1}{z}\big),~z\rightarrow\infty.$

Let $M^{\text{RHP}}(y,t;z)$ be the solution of the RH problem resulting from setting $\bar{\partial}\mathcal{R}^{(1)}\equiv0$ in $\bar{\partial}$-RH problem 3.1. Then the vanishing lemma of Zhou \cite{ZX1989} implies that $M^{\text{RHP}}$ exists and is unique. Moreover, we have the following lemma.
\begin{lemma}
As $t\rightarrow\infty$, we have
\be\label{3.18}
M^{\text{RHP}}(y,t;z)=\bigg(I+O(t^{-\frac{1}{5}})\bigg)M^Z(y;z,z_0),
\ee
where matrix-valued function $M^Z(y;z,z_0)$ is the solution of model RH problem B.1 with $s=r(0)$ and satisfies
\bea
M^Z(y;z,z_0)=I+\frac{M_1^Z(y)}{z}+O\bigg(\frac{1}{z^2}\bigg),\quad z\rightarrow\infty,\quad
(M_1^Z(y))_{12}=-\ii u_p(y),\label{3.19}
\eea
and $u_p(y)$ denotes the solution of the fourth order Painlev\'e II equation \eqref{A.5}.
\end{lemma}
\begin{proof}
For $z=z_0+l\e^{\frac{\pi\ii}{6}}$ and $z=-z_0+l\e^{\frac{5\pi\ii}{6}}$ with $l\geq0,z_0\geq0$ and $y=4z_0^4$, we find
\bea
\text{Re}\bigg(2\ii\bigg(\frac{4}{5}z^5-yz\bigg)\bigg)=
l^2\bigg(-\frac{4}{5}l^3-4\sqrt{3}z_0l^2-16z_0^2l-8\sqrt{3}z_0^3\bigg).\nn
\eea
For $z\in\Sigma_3$, $|\e^{\pm2\ii(\frac{4}{5}z^5-yz)}|=1$. Then, it follows that
\be
\|r(k_0)\e^{2\ii(\frac{4}{5}z^5-yz)}-r(0)\e^{2\ii(\frac{4}{5}z^5-yz)}\|_{L^p(\Sigma)}\leq ct^{-\frac{1}{5}},\quad p\in[1,\infty].
\ee
A simple calculation shows that the jumps $J^e$ of the quantity $e(z)=M^{\text{RHP}}(z)[M^Z(z)]^{-1},$ satisfy
\be
\|J^e-I\|_{L^p(\Sigma)}\leq ct^{-\frac{1}{5}}.
\ee
Using the theory of small-norm RH problems, one shows that $e(z)$ exists and that $e(z)=I+O(t^{-\frac{1}{5}})$.$\hfill\Box$
\end{proof}

Next we define the ratio
\be\label{3.22}
M^{(2)}(z)=M^{(1)}(z)[M^{RHP}(z)]^{-1},
\ee
then it is easy to verify that $M^{(2)}(z)$ is a continuously differentiable function satisfying the following pure $\bar{\partial}$ problem.\\
\textbf{$\bar{\partial}$ problem 3.2.} Find a function $M^{(2)}(z)$ with the following properties:

1. $M^{(2)}(z)$ is continuous with sectionally continuous first partial derivatives in $\bfC\setminus\Sigma$.

2. For $z\in\bfC\setminus\Sigma$, we have
\bea
\begin{aligned}
\bar{\partial}M^{(2)}(z)&=M^{(2)}(z)W^{(2)}(z),\\
W^{(2)}(z)&=M^{RHP}(z)\bar{\partial}\mathcal{R}^{(1)}(z)[M^{RHP}(z)]^{-1},
\end{aligned}
\eea
where $\bar{\partial}\mathcal{R}^{(1)}(z)$ is defined by \eqref{3.17}.

3. $M^{(2)}(z)$ admits asymptotics: $M^{(2)}(z)=I+O\big(\frac{1}{z}\big),~z\rightarrow\infty.$

Proceeding as in the previous section, we find that the $\bar{\partial}$ problem 3.2 is equivalent to the integral equation
\be\label{3.24}
M^{(2)}(z)=I-\frac{1}{\pi}\iint\limits_{\bfC}\frac{M^{(2)}(s)W^{(2)}(s)}{s-z}\dd A(s).
\ee
\begin{proposition}
For sufficiently large $t>0$, the operator $K_W$ defined in \eqref{2.73} satisfies
\be
\|K_W\|_{L^\infty\rightarrow L^\infty}\leq Ct^{-\frac{1}{10}}.
\ee
\end{proposition}
\begin{proof}
In $D_1^1=D_1\cap\{\text{Re}z>z_0\}$, we write $s=z_0+u+\ii v$, then $u\geq\sqrt{3}v\geq0$. We first note that
\bea
\text{Re}\bigg(2\ii\bigg(\frac{4}{5}s^5-ys\bigg)\bigg)&=&\frac{8}{5}v\bigg(
10u^2v^2+20uv^2z_0+10v^2z_0^2-5u^4-20u^3z_0-20uz_0^3-30u^2z_0^2-v^4\bigg)\nn\\
&\leq&\frac{8}{5}v\bigg(10u^2\cdot\frac{u^2}{3}-5u^4\bigg)\leq-2u^4v.\nn
\eea
Set $z=\alpha+\ii\beta$. Let $f\in L^\infty(D_1^1)$, we have
\bea\label{3.26}
|(K_Wf)(z)|\leq\|f\|_{L^\infty(D_1^1)}\iint\limits_{D_1^1}
\frac{|W^{(2)}(s)|}{|s-z|}\dd A(s)
\leq C\|f\|_{L^\infty(D_1^1)}\iint\limits_{D_1^1}
\frac{|\bar{\partial}R_1(s)|e^{-2u^4v}}{|s-z|}\dd A(s).
\eea
Thus, we find
\be
\|K_W\|_{L^\infty\rightarrow L^\infty}\leq C(I_1+I_2),
\ee
where
\bea
I_1&=&\int_0^\infty\int_{\sqrt{3}v}^\infty\frac{1}{|s-z|}
\bigg|t^{-\frac{1}{5}}r'((20t)^{-\frac{1}{5}}\text{Re}s)\bigg|e^{-2u^4v}\dd u\dd v,\nn\\
I_2&=&\int_0^\infty\int_{\sqrt{3}v}^\infty\frac{1}{|s-z|}\frac{1}{t^{\frac{1}{10}}|u+\ii v|^{\frac{1}{2}}}e^{-2u^4v}\dd u\dd v.\nn
\eea
We first note that
\berr
\bigg(\int_\bfR\bigg|t^{-\frac{1}{5}}r'((20t)^{-\frac{1}{5}}\text{Re}s)\bigg|^2\dd u\bigg)^{\frac{1}{2}}\leq Ct^{-\frac{1}{10}}.
\eerr
Recall the bound $$\bigg\|\frac{1}{|s-z|}\bigg\|_{L^2(\sqrt{3}v,\infty)}
\leq\bigg(\frac{\pi}{|v-\beta|}\bigg)^{\frac{1}{2}}.$$
Using these results and Schwarz's inequality on the $u$-integration we may bound $I_1$ by
\be
|I_1|\leq Ct^{-\frac{1}{10}}\int_0^\infty\frac{\e^{-18v^5}}{|v-\beta|^{\frac{1}{2}}}\dd v\leq Ct^{-\frac{1}{10}}.
\ee
For $p>2$, $\frac{1}{p}+\frac{1}{q}=1$, we recall the estimate from Appendix D of \cite{BJM},
\berr
\bigg\|\frac{1}{|u+\ii v|^{\frac{1}{2}}}\bigg\|_{L^p(\sqrt{3}v,\infty)}\leq Cv^{\frac{1}{p}-\frac{1}{2}},\quad \bigg\|\frac{1}{|s-z|}\bigg\|_{L^q(\sqrt{3}v,\infty)}\leq C|v-\beta|^{\frac{1}{q}-1}.
\eerr
Thus, we get
\be
|I_2|\leq Ct^{-\frac{1}{10}}\int_0^\infty v^{\frac{1}{p}-\frac{1}{2}}|v-\beta|^{\frac{1}{q}-1}\e^{-18v^5}\dd v\leq Ct^{-\frac{1}{10}}.
\ee $\hfill\Box$
\end{proof}
Finally, we consider the Laurent expansion of $M^{(2)}$ as $z\rightarrow\infty$. In fact, we have
\be\label{3.30}
M^{(2)}(z)=I+\frac{M_1^{(2)}}{z}+\frac{1}{\pi}\iint\limits_{\bfC}\frac{sM^{(2)}(s)W^{(2)}(s)}{z(s-z)}\dd A(s),
\ee
where
\be\label{3.31}
M_1^{(2)}=\frac{1}{\pi}\iint\limits_{\bfC}M^{(2)}(s)W^{(2)}(s)\dd A(s).
\ee
\begin{proposition}
For all large $t>0$, the following estimate holds:
\be\label{3.32}
|M_1^{(2)}|\leq Ct^{-\frac{1}{10}}.
\ee
\end{proposition}
\begin{proof}
We estimate the integral \eqref{3.31} as follows:
\bea
|M_1^{(2)}|&\leq& C\iint\limits_{D_1^1}|\bar{\partial}R_1(s)|e^{-2u^4v}\dd A(s)\leq C\bigg(\int_0^\infty\int_{\sqrt{3}v}^\infty\bigg|t^{-\frac{1}{5}}r'((20t)^{-\frac{1}{5}}
\text{Re}s)\bigg|e^{-2u^4v}\dd u\dd v\nn\\
&&+\int_0^\infty\int_{\sqrt{3}v}^\infty\frac{1}{t^{\frac{1}{10}}|u+\ii v|^{\frac{1}{2}}}e^{-2u^4v}\dd u\dd v\bigg)\leq C(I_3+I_4).\nn
\eea
We bound $I_3$ by applying the Cauchy--Schwarz inequality:
\bea
|I_3|&\leq& Ct^{-\frac{1}{10}}\int_0^\infty\e^{-9v^5}\bigg(\int_{\sqrt{3}v}^\infty\e^{-2u^4v}\dd u\bigg)^{\frac{1}{2}}\dd v\nn\\
&\leq& Ct^{-\frac{1}{10}}\sqrt{\Gamma(1/4)}\int_0^\infty\frac{\e^{-9v^5}}{\sqrt[8]{2v}}\dd v\leq Ct^{-\frac{1}{10}}.
\eea
For $I_4$, applying H\"older's inequality, we find
\bea
|I_4|&\leq& Ct^{-\frac{1}{10}}\int_0^\infty v^{\frac{1}{p}-\frac{1}{2}}\e^{-9v^5}\bigg(\int_{\sqrt{3}v}^\infty\e^{-qu^4v}\dd u\bigg)^{\frac{1}{q}}\dd v\nn\\
&\leq& Ct^{-\frac{1}{10}}\int_0^\infty v^{\frac{5}{4p}-\frac{3}{4}}\e^{-9v^5}\dd v\leq Ct^{-\frac{1}{10}}.
\eea$\hfill\Box$
\end{proof}
Recalling the transformations \eqref{3.13} and \eqref{3.22}, we have
\be
\begin{aligned}
M(x,t;k)&=M^{(2)}(z)M^{RHP}(z)[\mathcal{R}^{(1)}(z)]^{-1}.
\end{aligned}
\ee
Using the reconstruction formula \eqref{1.21} and \eqref{3.18}-\eqref{3.19}, \eqref{3.30}-\eqref{3.32}, we immediately find the asymptotics of the solution $u(x,t)$ in Region III when $x>0$:
\be\label{3.36}
u(x,t)=\bigg(\frac{8}{5t}\bigg)^{\frac{1}{5}}u_p\bigg(\frac{x}{(20 t)^{\frac{1}{5}}}\bigg)+O(t^{-\frac{3}{10}}).
\ee
For $x<0$, the two stationary points become
\be
\pm k_0=\pm\ii\sqrt[4]{\frac{-x}{80t}}.
\ee
We again perform the scaling $k\rightarrow(20t)^{-\frac{1}{5}}z$ and the contour deformation as shown in Fig. \ref{fig7}.
\begin{figure}[htbp]
  \centering
  \includegraphics[width=3in]{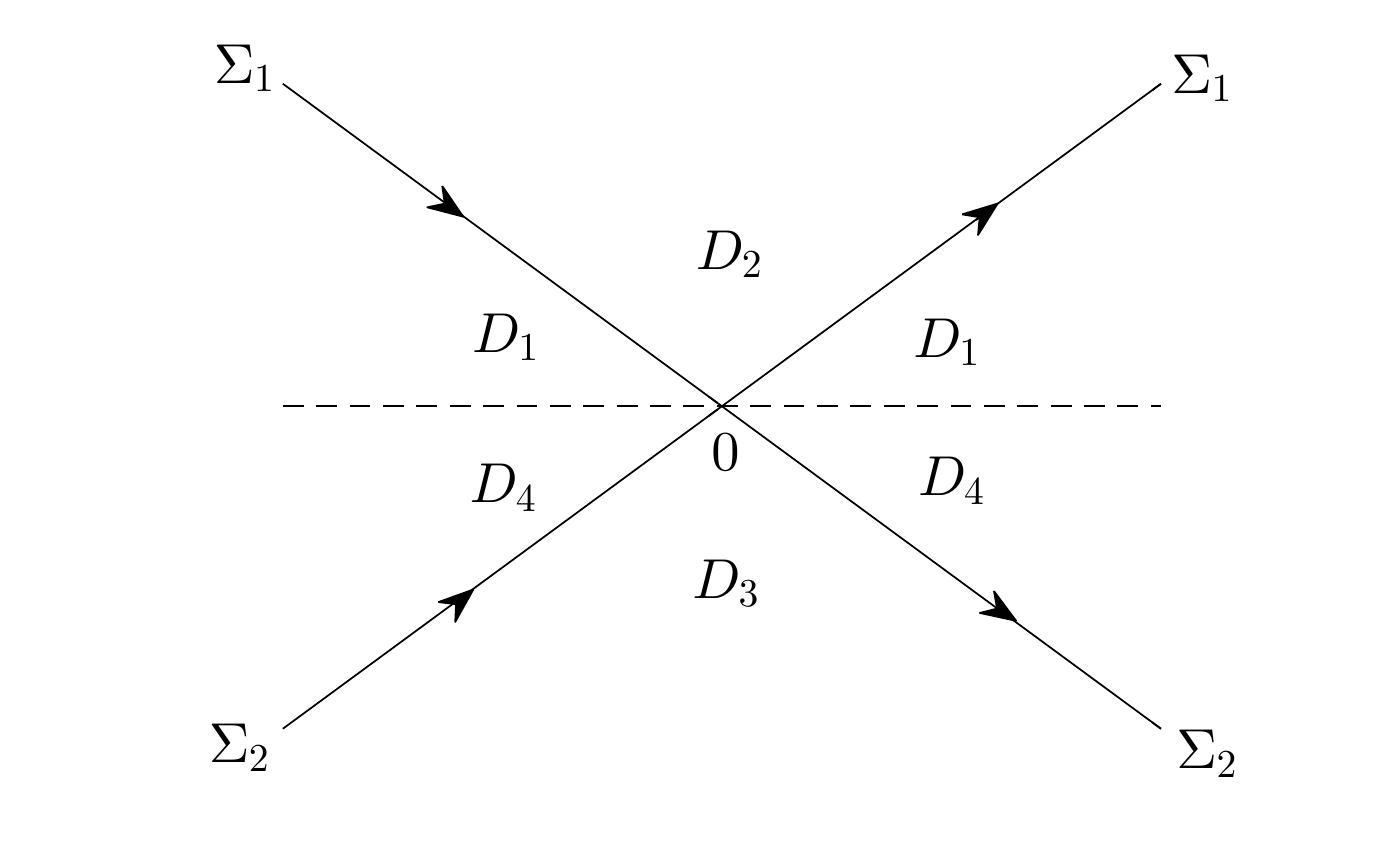}
  \caption{The contour $\Sigma$ and the open sets $\{D_j\}_1^4$ in the complex $z$-plane.}\label{fig7}
\end{figure}
Now the $\bar{\partial}$ extension in $\bar{D}_1$ turns into
\be
R_1(z)=\left\{
\begin{aligned}
&-r((20t)^{-\frac{1}{5}}z),\quad z\in(-\infty,\infty),\\
&-r(0),\qquad\qquad~~ z\in\Sigma_1,
\end{aligned}
\right.
\ee
and the interpolation is given by
\be
R_1(z)=-r(0)+\bigg(-r((20t)^{-\frac{1}{5}}z)+r(0)\bigg)\cos(3\phi).
\ee
Then we have
\bea
|\bar{\partial}R_1(z)|&=&\bigg|-\frac{1}{2}(20t)^{-\frac{1}{5}}
r'((20t)^{-\frac{1}{5}}\text{Re}z)\cos(3\phi)+
\frac{3\ii}{2}\e^{\ii\phi}\frac{r((20t)^{-\frac{1}{5}}\text{Re}z)-r(0)}
{|z|}\sin(3\phi)\bigg|\nn\\
&\leq&c_1t^{-\frac{1}{5}}|r'((20t)^{-\frac{1}{5}}\text{Re}z)|
+c_2t^{-\frac{1}{10}}|z|^{-\frac{1}{2}}.
\eea
And if we write $s=u+\ii v$ in $D_1^1=D_1\cap\{\text{Re}z>0\}$, then $u\geq\sqrt{3}v\geq0$. Note that $y<0$, and then
\be
\text{Re}\bigg(2\ii\bigg(\frac{4}{5}s^5-ys\bigg)\bigg)
=\frac{8}{5}v(10u^2v^2-5u^4-v^4)+2yv\leq-2u^4v.
\ee
Thus we can repeat the analysis as the case above for $x>0$ and obtain the same long-time asymptotics as \eqref{3.36} for the solution $u(x,t)$.

\section{Asymptotics in Regions II, IV and V}\label{sec4}
\setcounter{equation}{0}
\setcounter{lemma}{0}
\setcounter{theorem}{0}
\setcounter{proposition}{0}
We first derive the asymptotics for the fifth-order modified KdV equation \eqref{5mKdV} in region IV, then V and finally II.
\subsection{Region IV: $k_0\leq M,\tau\geq\tilde{M}$}
Again, we perform the scaling transformation
\be\label{4.1}
k\rightarrow(20t)^{-\frac{1}{5}}z.
\ee
Letting $s=u+\ii v$, we find $\text{Re}\bigg(2\ii\bigg(\frac{4}{5}s^5-ys\bigg)\bigg)
=\frac{8}{5}v(10u^2v^2-5u^4-v^4)+2yv.$
As $y=-4(20\tau)^{\frac{4}{5}}<0$, we can select $\rho>0$ which is sufficiently small and independent of $y$
and construct a new contour $\gamma=\gamma_1\cup\gamma_2$ given in Fig. \ref{fig8}, such that in sector $E_1^1$,
\bea
\text{Re}\bigg(2\ii\bigg(\frac{4}{5}s^5-ys\bigg)\bigg)\leq
\frac{8}{5}vu^2(10\rho^2-5u^2)-8(20\tau)^{\frac{4}{5}}v\leq-u^2v<0.
\eea
\begin{figure}[htbp]
  \centering
  \includegraphics[width=3.3in]{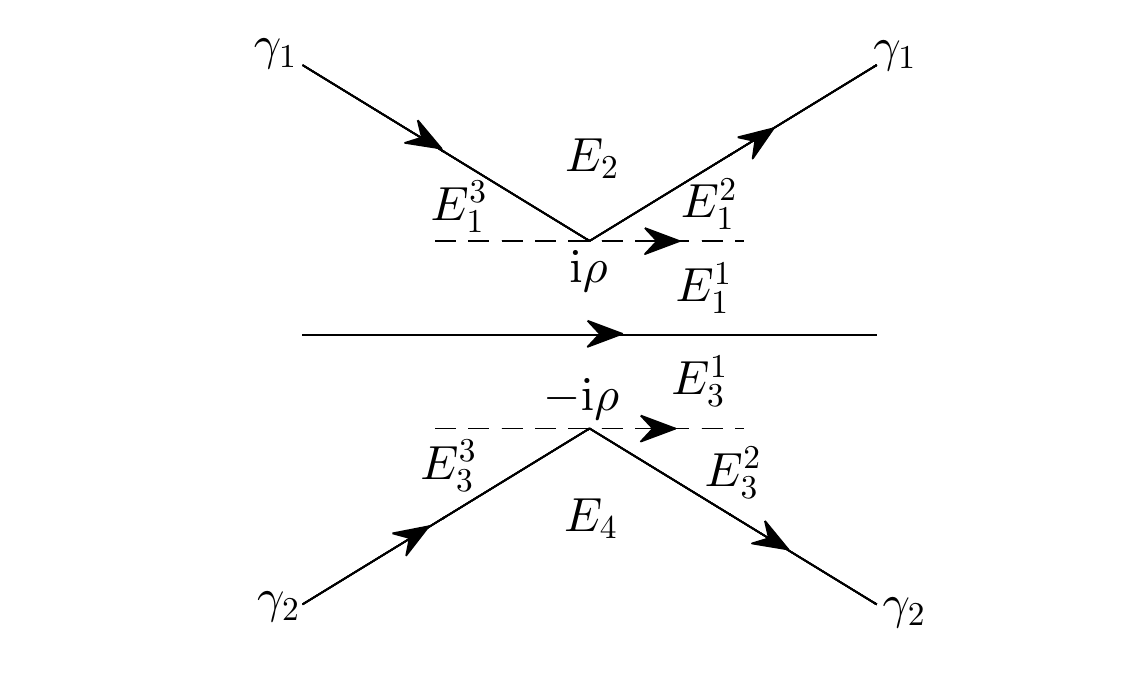}
  \caption{The contour $\gamma=\gamma_1\cup\gamma_2$ and the open sets $\{E_j\}_1^4$ in the complex $z$-plane.}\label{fig8}
\end{figure}
We define the extension in $\bar{E}_1^1$ by $R_1^1(z)=-r((20t)^{-\frac{1}{5}}\text{Re}z)$, in $\bar{E}_3^1$ by $R_3^1(z)=-\overline{r((20t)^{-\frac{1}{5}}\text{Re}z)}$. Then, we set
\be\label{4.3}
M^{(1)}(y,t;z)=M(x,t;k)\mathcal{R}^{(1)}(z),
\ee
where
\be
\mathcal{R}^{(1)}(z)=\left\{
\begin{aligned}
&\begin{pmatrix}
1 &~ 0\\
R_1^1(z)\e^{2\ii(\frac{4}{5}z^5-yz)} &~ 1
\end{pmatrix},~\quad z\in E_1^1,\\
&\begin{pmatrix}
1 &~ R_3^1(z)\e^{-2\ii(\frac{4}{5}z^5-yz)}\\
0 &~ 1
\end{pmatrix},\quad z\in E_3^1\\
&I,\qquad\qquad\qquad\qquad\qquad\quad~ \text{elsewhere}.
\end{aligned}
\right.
\ee
Thus, $M^{(1)}(y,t;z)$ satisfies the following $\bar{\partial}$-Riemann--Hilbert problem.\\
\textbf{$\bar{\partial}$-Riemann--Hilbert problem 4.1.} Find a function $M^{(1)}(y,t;z)$ with the following properties:

1. $M^{(1)}(y,t;z)$ is continuous with sectionally continuous first partial derivatives in $\bfC\setminus\{z|\text{Im}z=\pm\rho\}$.

2. Across $\text{Im}z=\pm\rho$, the boundary values satisfy the jump relation
\be
M^{(1)}_+(y,t;z)=M^{(1)}_-(y,t;z)J^{(1)}(y,t;z),
\ee
where
\bea\label{4.6}
J^{(1)}(y,t;z)&=\left\{
\begin{aligned}
&\begin{pmatrix}
1 &~ 0\\
r((20t)^{-\frac{1}{5}}\text{Re}z)\e^{2\ii(\frac{4}{5}z^5-yz)} &~ 1
\end{pmatrix},\quad\quad~~ z\in\{z|\text{Im}z=\rho\},\\
&\begin{pmatrix}
1 &~ -\overline{r((20t)^{-\frac{1}{5}}\text{Re}z)}\e^{-2\ii(\frac{4}{5}z^5-yz)}\\
0 &~ 1
\end{pmatrix},\quad z\in\{z|\text{Im}z=-\rho\}.
\end{aligned}
\right.
\eea

3. For $z\in\bfC\setminus\{z|\text{Im}z=\pm\rho\}$, we have
\be\label{4.7}
\begin{aligned}
\bar{\partial}M^{(1)}(z)&=M^{(1)}(z)\bar{\partial}\mathcal{R}^{(1)}(z),\\
\bar{\partial}\mathcal{R}^{(1)}(z)&=\left\{
\begin{aligned}
&\begin{pmatrix}
1 &~ 0\\
\bar{\partial}R_1^1(z)\e^{2\ii(\frac{4}{5}z^5-yz)} &~ 1
\end{pmatrix},\quad z\in E^1_1,\\
&\begin{pmatrix}
1 &~ \bar{\partial}R_3^1(z)\e^{-2\ii(\frac{4}{5}z^5-yz)}\\
0 &~ 1
\end{pmatrix},~~ z\in E^1_3,\\
&0,\qquad\qquad\qquad\qquad\qquad\quad~ \text{elsewhere}.
\end{aligned}
\right.
\end{aligned}
\ee

4. $M^{(1)}(y,t;z)$ enjoys asymptotics: $M^{(1)}(y,t;z)=I+O\big(\frac{1}{z}\big),~z\rightarrow\infty.$\\
Let $M^{(2)}(y,t;z)$ be the solution of the RH problem resulting from setting $\bar{\partial}\mathcal{R}^{(1)}\equiv0$ in $\bar{\partial}$-RH problem 4.1.
We now define the ratio
\be\label{4.8}
M^{(3)}(z)=M^{(1)}(z)[M^{(2)}(z)]^{-1},
\ee
then it is easy to verify that $M^{(3)}(z)$ satisfies the following pure $\bar{\partial}$ problem.\\
\textbf{$\bar{\partial}$ problem 4.2.} Find a function $M^{(3)}(z)$ satisfying the following properties:

1. $M^{(3)}(z)$ is continuous with sectionally continuous first partial derivatives in $\bfC\setminus\{z|\text{Im}z=\pm\rho\}$.

2. For $z\in\bfC\setminus\{z|\text{Im}z=\pm\rho\}$, we have
\bea
\begin{aligned}
\bar{\partial}M^{(3)}(z)&=M^{(3)}(z)W^{(3)}(z),\\
W^{(3)}(z)&=M^{(2)}(z)\bar{\partial}\mathcal{R}^{(1)}(z)[M^{(2)}(z)]^{-1},
\end{aligned}
\eea
where $\bar{\partial}\mathcal{R}^{(1)}(z)$ is defined by \eqref{4.7}.

3. $M^{(3)}(z)$ admits asymptotics: $M^{(3)}(z)=I+O\big(\frac{1}{z}\big),~z\rightarrow\infty.$\\
Consider the integral
\bea
\iint\limits_{E_1^1}\frac{|W^{(3)}(s)|}{|s-z|}\dd A(s)
&\leq&\iint\limits_{E_1^1}
\frac{|\bar{\partial}R_1(s)|e^{-u^2v}}{|s-z|}\dd A(s)\nn\\
&\leq&\int_0^\rho\int_{-\infty}^\infty\frac{1}{|s-z|}
\bigg|t^{-\frac{1}{5}}r'((20t)^{-\frac{1}{5}}\text{Re}s)\bigg|e^{-u^2v}\dd u\dd v\nn\\
&\leq&Ct^{-\frac{1}{10}}\int_0^\rho\frac{1}{|v-\beta|^{\frac{1}{2}}}\dd v\leq Ct^{-\frac{1}{10}}.
\eea
On the other hand, we have
\bea
\iint\limits_{E_1^1}|W^{(3)}(s)|\dd A(s)
&\leq&\iint\limits_{E_1^1}|\bar{\partial}R_1(s)|e^{-u^2v}\dd A(s)\nn\\
&\leq&\int_0^\rho\int_{-\infty}^\infty
\bigg|t^{-\frac{1}{5}}r'((20t)^{-\frac{1}{5}}\text{Re}s)\bigg|e^{-u^2v}\dd u\dd v\nn\\
&\leq&Ct^{-\frac{1}{10}}\int_0^\rho\frac{1}{\sqrt[4]{v}}\dd v\leq Ct^{-\frac{1}{10}}.
\eea
Thus, $M^{(3)}(z)$ exists and if we expand $M^{(3)}(z)$ as
\be\label{4.12}
M^{(3)}(z)=I+\frac{M_1^{(3)}}{z}+O\bigg(\frac{1}{z^2}\bigg),\quad z\rightarrow\infty,
\ee
we have
\be\label{4.13}
|M_1^{(3)}|\leq Ct^{-\frac{1}{10}}.
\ee

Next, we analyze $M^{(2)}(y,t;z)$ which satisfies the conditions 1, 2, 4 of $\bar{\partial}$-RH problem 4.1. We aim to deform the contour $\{z|\text{Im}z=\pm\rho\}$ to $\gamma=\gamma_1\cup\gamma_2$. For this purpose, we introduce a new unknown $M^{(4)}$ obtained from $M^{(2)}$ as
\be\label{4.14}
M^{(4)}(z)=M^{(2)}(z)\mathcal{R}^{(2)}(z).
\ee
We choose $\mathcal{R}^{(2)}(z)$ to remove the jump on the contour $\{z|\text{Im}z=\pm\rho\}$. More precisely, we define functions $R_1^2$ and $R_3^2$ satisfying
\bea
R_1^2(z)&=&\left\{
\begin{aligned}
&-r((20t)^{-\frac{1}{5}}z),\quad \text{Im}z=\rho,\\
&-r(0),\qquad\qquad z\in\gamma_1,
\end{aligned}
\right.\label{4.15}\\
R_3^2(z)&=&\left\{
\begin{aligned}
&-\overline{r((20t)^{-\frac{1}{5}}z)},\quad \text{Im}z=-\rho,\\
&-\overline{r(0)},\qquad\qquad z\in\gamma_2,
\end{aligned}
\right.\label{4.16}
\eea
and then we can select $\mathcal{R}^{(2)}(z)$ as follows
\bea
\mathcal{R}^{(2)}(z)=\left\{
\begin{aligned}
&\begin{pmatrix}
1 &~ 0\\
R_1^2(z)\e^{2\ii(\frac{4}{5}z^5-yz)} &~ 1
\end{pmatrix},~\quad z\in E_1^2\cup E_1^3,\\
&\begin{pmatrix}
1 &~ R_3^2(z)\e^{-2\ii(\frac{4}{5}z^5-yz)}\\
0 &~ 1
\end{pmatrix},\quad z\in E_3^2\cup E_3^3,\\
&I,\qquad\qquad\qquad\qquad\qquad\quad~ \text{elsewhere}.
\end{aligned}
\right.
\eea
\begin{lemma}
There exists functions $R_1^2(z)$ in $E_1^2\cup E_1^3$ and $R_3^2$ in $E_3^2\cup E_3^3$ satisfying the boundary value conditions \eqref{4.15}-\eqref{4.16}, such that
\be\label{4.18}
|\bar{\partial}R_j^2|\leq c_1t^{-\frac{1}{5}}|r'((20t)^{-\frac{1}{5}}\text{Re}z)|
+c_2t^{-\frac{1}{10}}|z|^{-\frac{1}{2}},\quad j=1,3.
\ee
\end{lemma}
Thus, $M^{(4)}(z)$ satisfies the following $\bar{\partial}$-Riemann--Hilbert problem.\\
\textbf{$\bar{\partial}$-Riemann--Hilbert problem 4.3.} Find a function $M^{(4)}(z)$ with the following properties:

1. $M^{(4)}(z)$ is continuous with sectionally continuous first partial derivatives in $\bfC\setminus\gamma$.

2. Across $\gamma$, the boundary values $M^{(4)}_\pm(z)$ satisfy the jump relation
\be
M^{(4)}_+(z)=M^{(4)}_-(z)J^{(4)}(y,t;z),
\ee
where
\bea\label{4.20}
J^{(4)}(y,t;z)&=\left\{
\begin{aligned}
&\begin{pmatrix}
1 &~ 0\\
r(0)\e^{2\ii(\frac{4}{5}z^5-yz)} &~ 1
\end{pmatrix},\quad\quad~~ z\in\gamma_1,\\
&\begin{pmatrix}
1 &~ -\overline{r(0)}\e^{-2\ii(\frac{4}{5}z^5-yz)}\\
0 &~ 1
\end{pmatrix},\quad z\in\gamma_2.
\end{aligned}
\right.
\eea

3. For $z\in\bfC\setminus\gamma$, we have
\be\label{4.21}
\begin{aligned}
\bar{\partial}M^{(4)}(z)&=M^{(4)}(z)\bar{\partial}\mathcal{R}^{(2)}(z),\\
\bar{\partial}\mathcal{R}^{(2)}(z)&=\left\{
\begin{aligned}
&\begin{pmatrix}
1 &~ 0\\
\bar{\partial}R_1^2(z)\e^{2\ii(\frac{4}{5}z^5-yz)} &~ 1
\end{pmatrix},\quad z\in E^2_1\cup E_1^3,\\
&\begin{pmatrix}
1 &~ \bar{\partial}R_3^2(z)\e^{-2\ii(\frac{4}{5}z^5-yz)}\\
0 &~ 1
\end{pmatrix},~~ z\in E^2_3\cup E_3^3,\\
&0,\qquad\qquad\qquad\qquad\qquad\quad~ \text{elsewhere}.
\end{aligned}
\right.
\end{aligned}
\ee

4. $M^{(4)}(z)=I+O\big(\frac{1}{z}\big),~z\rightarrow\infty.$\\
Letting $\bar{\partial}\mathcal{R}^{(2)}=0$ in $\bar{\partial}$-RH problem 4.3, it easy to see that the remaining pure RH problem is equivalent to the fourth order Painlev\'e II RH problem A.1 by setting $s=r(0)$ up to a trivial contour deformation. Thus, the ratio
\be\label{4.22}
M^{(5)}(z)=M^{(4)}(z)[M^P(z)]^{-1}
\ee
is a continuously differentiable function satisfying the following $\bar{\partial}$ problem.\\
\textbf{$\bar{\partial}$ problem 4.4.} Find a function $M^{(5)}(z)$ with the following properties:

1. $M^{(5)}(z)$ is continuous with sectionally continuous first partial derivatives in $\bfC\setminus\gamma$.

2. For $z\in\bfC\setminus\gamma$, we have
\bea
\begin{aligned}
\bar{\partial}M^{(5)}(z)&=M^{(5)}(z)W^{(5)}(z),\\
W^{(5)}(z)&=M^{P}(z)\bar{\partial}\mathcal{R}^{(2)}(z)[M^{P}(z)]^{-1},
\end{aligned}
\eea
where $\bar{\partial}\mathcal{R}^{(2)}(z)$ is defined by \eqref{4.21}.

3. $M^{(5)}(z)$ admits asymptotics: $M^{(5)}(z)=I+O\big(\frac{1}{z}\big),~z\rightarrow\infty.$\\
In order to show the existence of $M^{(5)}(z)$, as the discussion in Section 2, we need to check the boundedness of the operator $K_W$ defined by
\be\label{4.24}
(K_Wf)(z)=-\frac{1}{\pi}\iint\limits_{\bfC}\frac{f(s)W^{(5)}(s)}{s-z}\dd A(s).
\ee
Indeed, we have the following proposition.
\begin{proposition}
For large time, the integral operator $K_W$ given by \eqref{4.24} obeys the estimate
\be
\|K_W\|_{L^\infty\rightarrow L^\infty}\leq  Ct^{-\frac{1}{10}}\e^{-8(20\tau)^{\frac{4}{5}}\rho}.
\ee
\end{proposition}
\begin{proof}
In $E_1^2$, we write $s=u+\ii(v+\rho)$, where $u\geq\sqrt{3}v\geq0$. Then we have
\bea
\text{Re}\bigg(2\ii\bigg(\frac{4}{5}s^5-ys\bigg)\bigg)
&=&\frac{8}{5}(v+\rho)\bigg(10u^2(v+\rho)^2-5u^4-(v+\rho)^4\bigg)+2y(v+\rho)\nn\\
&\leq&-u^4v-8(20\tau)^{\frac{4}{5}}\rho.
\eea
Thus, it follows from \eqref{4.18} that
\be
\|K_W\|_{L^\infty\rightarrow L^\infty}\leq C(I_1+I_2),
\ee
where
\bea
I_1&=&\int_0^\infty\int_{\sqrt{3}(v+\rho)}^\infty\frac{1}{|s-z|}
\bigg|t^{-\frac{1}{5}}r'((20t)^{-\frac{1}{5}}\text{Re}s)
\bigg|e^{-u^4v-8(20\tau)^{\frac{4}{5}}\rho}\dd u\dd v,\nn\\
I_2&=&\int_0^\infty\int_{\sqrt{3}(v+\rho)}^\infty\frac{1}{|s-z|}
\frac{1}{t^{\frac{1}{10}}|u+\ii (v+\rho)|^{\frac{1}{2}}}e^{-u^4v-8(20\tau)^{\frac{4}{5}}\rho}\dd u\dd v.\nn
\eea
Proceeding the same procedure, we can get following estimates
\be
|I_1|,|I_2|\leq Ct^{-\frac{1}{10}}\e^{-8(20\tau)^{\frac{4}{5}}\rho}.
\ee$\hfill\Box$
\end{proof}
Moreover, if we assume that
\be\label{4.29}
M^{(5)}(z)=I+\frac{M^{(5)}_1}{z}+O\bigg(\frac{1}{z^2}\bigg),~z\rightarrow\infty,
\ee
we then have
\be\label{4.30}
|M^{(5)}_1|\leq Ct^{-\frac{1}{10}}\e^{-8(20\tau)^{\frac{4}{5}}\rho}.
\ee
Combining all the transforms \eqref{4.3}, \eqref{4.8}, \eqref{4.14} and \eqref{4.22}, we find
\be
M(x,t;k)=M^{(3)}(z)M^{(5)}(z)M^{P}(z)[\mathcal{R}^{(2)}(z)]^{-1}[\mathcal{R}^{(1)}(z)]^{-1}.
\ee
We now can use \eqref{4.1}, \eqref{4.12}-\eqref{4.13}, \eqref{4.29}-\eqref{4.30} and \eqref{A.3}-\eqref{A.4} to obtain the long-time asymptotics of $u(x,t)$ in Region IV,
\be
u(x,t)=\bigg(\frac{8}{5t}\bigg)^{\frac{1}{5}}u_p\bigg(\frac{x}{(20 t)^{\frac{1}{5}}}\bigg)+O(t^{-\frac{3}{10}}\e^{-8(20\tau)^{\frac{4}{5}}\rho}).
\ee

\subsection{Region V: $k_0\geq M, x\rightarrow-\infty$}
Observe that on $\gamma_1$,
\be
\|r(0)\e^{2\ii(\frac{4}{5}z^5-yz)}\|_{L^1\cap L^2\cap L^\infty}\leq C\e^{-8(20\tau)^{\frac{4}{5}}\rho}.
\ee
However, we may choose $\rho\geq(20\tau)^{\frac{1}{5}}$. It then follows that in Region V
\be
u(x,t)=O\bigg(t^{-\frac{1}{5}}\e^{-c\tau}+t^{-\frac{3}{10}}\e^{-8(20\tau)^{\frac{4}{5}}\rho}\bigg).
\ee

\subsection{Region II: $o(t^{\frac{2}{7}})=\tau\geq\tilde{M}$}
We now scale:
\be\label{4.35}
k\rightarrow k_0z,
\ee
and then have
\be\label{4.36}
J(x,t;z)=\begin{pmatrix}
1 &~ -\overline{r(k_0z)}\e^{-32\ii\tau(z^5-5z)}\\
0 &~ 1
\end{pmatrix}\begin{pmatrix}
1 &~ 0\\
r(k_0z)\e^{32\ii\tau(z^5-5z)}  &~ 1
\end{pmatrix}.
\ee
Construct the contour $\Upsilon$ as shown in Fig. \ref{fig9}, and define the function $R_1(z)$ in $F_1$ with the boundary values
\be
R_1(z)=\left\{
\begin{aligned}
&-r(k_0z),\quad~ z\in(-\infty,-1)\cup(1,\infty), \\
&-r(k_0),\qquad z\in\Upsilon_1.
\end{aligned}
\right.
\ee
In fact, we can choose
\be
R_1(z)=-r(k_0)+\bigg(r(k_0)-r(k_0\text{Re}z)\bigg)\cos(3\phi).
\ee
A simple computation shows that
\be\label{4.39}
|\bar{\partial}R_1(z)|\leq c_1k_0r'(k_0\text{Re}z)+c_2k_0^{\frac{1}{2}}|z-1|^{-\frac{1}{2}}.
\ee
However, if we write $z=u+1+\ii v$ in $F_1\cap\{\text{Re}z>1\}$, then one can get
\bea\label{4.40}
\text{Re}\bigg(32\ii\tau(z^5-5z)\bigg)&=&32\tau v(10u^2v^2+20uv^2+10v^2-5u^4-20u^3-30u^2-20u-v^4)\nn\\
&\leq&-2\tau u^4v.
\eea
On the other hand, on $\Upsilon_1$, for $\tau\geq\tilde{M}$, we have
\be
\|r(k_0)\e^{32\ii\tau(z^5-5z)}-r(0)\e^{32\ii\tau(z^5-5z)}\|_{L^p}\leq ck_0,\quad p\in[1,\infty].
\ee
\begin{figure}[htbp]
  \centering
  \includegraphics[width=3in]{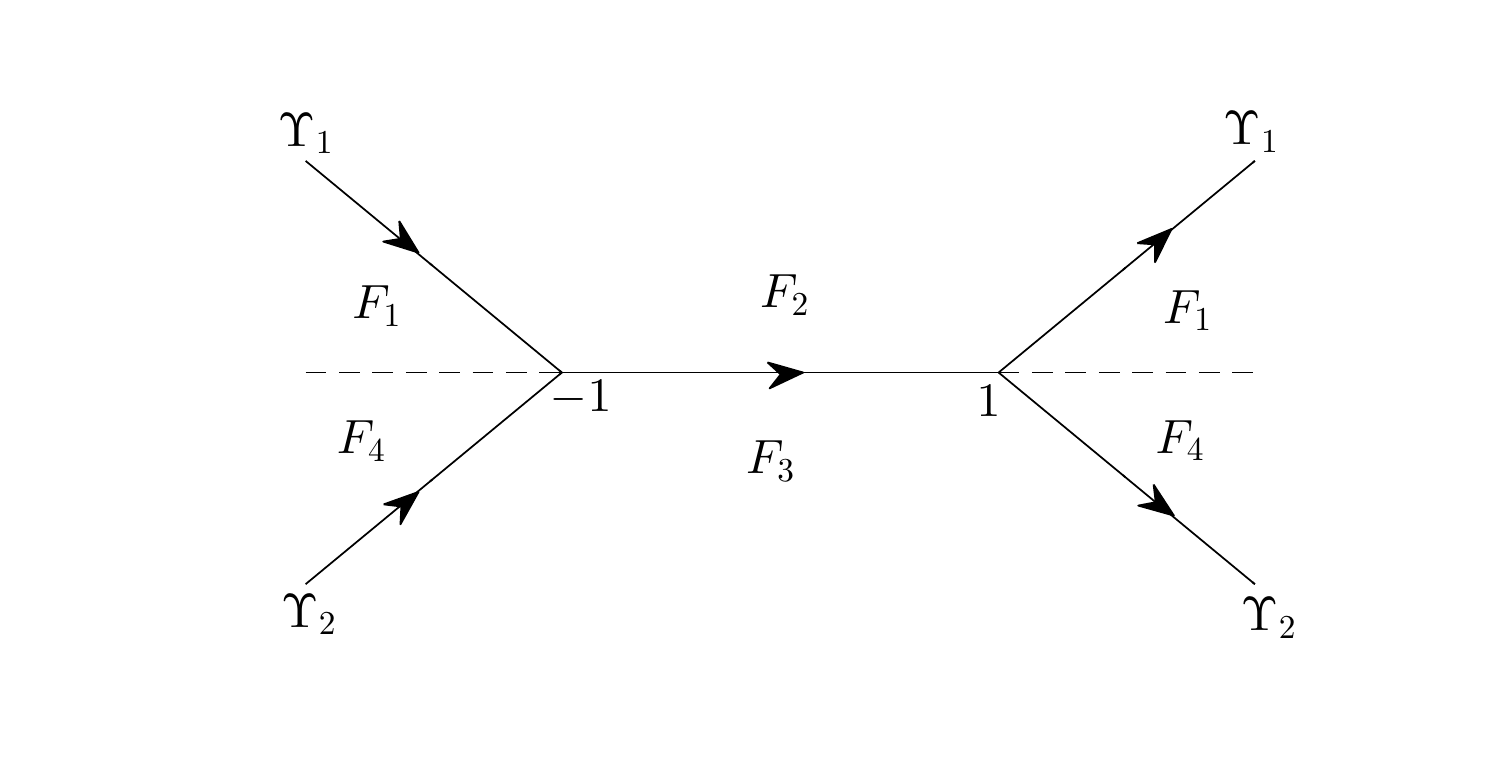}
  \caption{The contour $\Upsilon$ and the open sets $\{F_j\}_1^4$ in the complex $z$-plane.}\label{fig9}
\end{figure}
Following the notation and analysis of Section 3, we then only need to give the following estimates.
\begin{proposition}
The integral operator $K_W$ given by \eqref{2.73} obeys the estimate
\be
\|K_W\|_{L^\infty\rightarrow L^\infty}\leq  Ck_0^{\frac{1}{2}}\tau^{-\frac{1}{10}}.
\ee
\end{proposition}
\begin{proof}
In $F_1^1:=F_1\cap\{\text{Re}z>1\}$, it follows from \eqref{4.39}-\eqref{4.40} that
\be
\|K_W\|_{L^\infty\rightarrow L^\infty}\leq C(I_1+I_2),
\ee
where
\bea
I_1&=&\int_0^\infty\int_{\sqrt{3}v}^\infty\frac{1}{|s-z|}
\bigg|k_0r'(k_0\text{Re}s)\bigg|e^{-2\tau u^4v}\dd u\dd v,\nn\\
I_2&=&\int_0^\infty\int_{\sqrt{3}v}^\infty\frac{1}{|s-z|}
\frac{k_0^{\frac{1}{2}}}{|u+\ii v|^{\frac{1}{2}}}e^{-2\tau u^4v}\dd u\dd v.\nn
\eea
For $I_1$, we can get following estimates
\bea
|I_1|\leq Ck_0^{\frac{1}{2}}\int_0^\infty\frac{\e^{-18\tau v^5}}{|v-\beta|^{\frac{1}{2}}}\dd v\leq Ck_0^{\frac{1}{2}}\tau^{-\frac{1}{10}}\int_0^\infty
\frac{\e^{-18(w+\tau^{\frac{1}{5}}\beta)^5}}{|w|^{\frac{1}{2}}}\dd w\leq C
k_0^{\frac{1}{2}}\tau^{-\frac{1}{10}}.
\eea
For $I_2$, we have
\bea
|I_2|\leq k_0^{\frac{1}{2}}\int_0^\infty\e^{-18\tau v^5}v^{\frac{1}{p}-\frac{1}{2}}|v-\beta|^{\frac{1}{q}-1}\dd v.
\eea
Observe that $\e^{-m}\leq m^{-\frac{1}{10}}$, thus, one can get
\bea
\int_0^\beta\e^{-18\tau v^5}v^{\frac{1}{p}-\frac{1}{2}}|v-\beta|^{\frac{1}{q}-1}\dd v&=&\int_0^1\beta^{\frac{1}{2}}\e^{-18\tau \beta^5w^5}w^{\frac{1}{p}-\frac{1}{2}}(1-w)^{\frac{1}{q}-1}\dd w\nn\\
&\leq&\tau^{-\frac{1}{10}}\int_0^1w^{\frac{1}{p}-1}(1-w)^{\frac{1}{q}-1}\dd w\leq\tau^{-\frac{1}{10}}.\nn
\eea
Finally, we have
\bea
\int_\beta^\infty\e^{-18\tau v^5}v^{\frac{1}{p}-\frac{1}{2}}(v-\beta)^{\frac{1}{q}-1}\dd v\leq\int_0^\infty\e^{-18\tau w^5}w^{-\frac{1}{2}}\dd w\leq\tau^{-\frac{1}{10}}.\nn
\eea$\hfill\Box$
\end{proof}
Next, we estimate the integral defined in \eqref{3.31}.
\begin{proposition}
For all large $t>0$, we find
\be\label{4.46}
|M_1^{(2)}|\leq Ck_0^{\frac{1}{2}}\tau^{-\frac{3}{10}}.
\ee
\end{proposition}
\begin{proof}
We estimate the integral \eqref{3.31} as follows:
\bea
|M_1^{(2)}|&\leq& C\iint\limits_{F_1^1}|\bar{\partial}R_1(s)|e^{-2\tau u^4v}\dd A(s)\leq C\bigg(\int_0^\infty\int_{\sqrt{3}v}^\infty\bigg|k_0r'(k_0
\text{Re}s)\bigg|e^{-2\tau u^4v}\dd u\dd v\nn\\
&&+\int_0^\infty\int_{\sqrt{3}v}^\infty\frac{k_0^{\frac{1}{2}}}{|u+\ii v|^{\frac{1}{2}}}e^{-2\tau u^4v}\dd u\dd v\bigg)\leq C(I_3+I_4).\nn
\eea
We bound $I_3$ as follows:
\bea
|I_3|&\leq& Ck_0^{\frac{1}{2}}\int_0^\infty\e^{-9\tau v^5}\bigg(\int_{\sqrt{3}v}^\infty\e^{-2\tau u^4v}\dd u\bigg)^{\frac{1}{2}}\dd v\nn\\
&\leq& Ck_0^{\frac{1}{2}}\sqrt{\Gamma(1/4)}\int_0^\infty\frac{\e^{-9\tau v^5}}{\sqrt[8]{2\tau v}}\dd v\leq Ck_0^{\frac{1}{2}}\tau^{-\frac{3}{10}}.
\eea
For $I_4$, applying H\"older's inequality, we find
\bea
|I_4|&\leq& Ck_0^{\frac{1}{2}}\int_0^\infty v^{\frac{1}{p}-\frac{1}{2}}\e^{-9\tau v^5}\bigg(\int_{\sqrt{3}v}^\infty\e^{-q\tau u^4v}\dd u\bigg)^{\frac{1}{q}}\dd v\nn\\
&\leq& Ck_0^{\frac{1}{2}}\tau^{-\frac{1}{4q}}\int_0^\infty v^{\frac{5}{4p}-\frac{3}{4}}\e^{-9\tau v^5}\dd v\leq Ck_0^{\frac{1}{2}}\tau^{-\frac{3}{10}}.
\eea$\hfill\Box$
\end{proof}
Using these estimates and the scaling transform \eqref{4.35}, we immediately find the asymptotics of the solution $u(x,t)$ in Region II:
\bea
u(x,t)&=&\bigg(\frac{8}{5t}\bigg)^{\frac{1}{5}}u_p\bigg(\frac{x}{(20 t)^{\frac{1}{5}}}\bigg)+O(k_0^{\frac{3}{2}}\tau^{-\frac{3}{10}}+k_0^2)\nn\\
&=&\bigg(\frac{8}{5t}\bigg)^{\frac{1}{5}}u_p\bigg(\frac{x}{(20 t)^{\frac{1}{5}}}\bigg)+O\bigg(t^{-\frac{3}{10}}
+\bigg(\frac{\tau}{t}\bigg)^{\frac{2}{5}}\bigg).
\eea
\begin{remark}\label{rem4.2}
We show how to match the asymptotic formulas of solution $u(x,t)$ in the overlaps of Regions I and II. It is easy to see that the final model RH problem on $\Upsilon$ corresponding to a special case of in Region I in which one replace
\be
t\rightarrow\tau,\quad k_0\rightarrow1,\quad r(z)\rightarrow r(0).\nn
\ee
Then, we find as $\tau\rightarrow\infty$,
\bea
\bigg(\frac{8}{5t}\bigg)^{\frac{1}{5}}u_p\bigg(\frac{x}{(20 t)^{\frac{1}{5}}}\bigg)&=&k_0\frac{\sqrt{\nu(r(0))}}{2\sqrt{10\tau}}
\cos\bigg(128\tau+\nu(r(0))\ln(2560 \tau)\nn\\
&&-\frac{3\pi}{4}-\arg r(0)+\arg\Gamma(\ii\nu(r(0)))\bigg)
+k_0O\bigg(\tau^{-\frac{3}{4}}\bigg),
\eea
and thus in Region II
\bea
u(x,t))&=&\frac{\sqrt{\nu(r(0))}}{2k_0\sqrt{10tk_0}}
\cos\bigg(128\tau+\nu(r(0))\ln(2560 \tau)\nn\\
&&-\frac{3\pi}{4}-\arg r(0)+\arg\Gamma(\ii\nu(r(0)))\bigg)
+k_0O\bigg(\tau^{-\frac{3}{4}}\bigg)+O\bigg(\bigg(\frac{\tau}{t}\bigg)^{\frac{2}{5}}\bigg).
\eea
On the other hand, in Region I, it follows from the similar analysis in \cite{DZ1993} that
\bea
u(x,t)&=&\frac{\sqrt{\nu(r(0))}}{2k_0\sqrt{10k_0t}}\cos\bigg(128\tau+\nu(r(0))\ln(2560 \tau)-\frac{3\pi}{4}-\arg r(0)\nn\\
&&+\arg\Gamma(\ii\nu(r(0)))\bigg)+(tk_0^3)^{-\frac{1}{2}}O\bigg(\tau^{-1}
(tk_0^3)^{\frac{1}{2}}+(k_0^3t)^{-\frac{1}{4}}+k_0\ln\tau\bigg).
\eea
However, for $\tau=o(t^{\frac{2}{7}})$, we find
\be
k_0\tau^{-\frac{3}{4}}=(tk_0^3)^{-\frac{1}{2}}\tau^{-\frac{1}{4}},\quad k_0\ln\tau=o\bigg(\frac{\ln t}{t^{\frac{1}{7}}}\bigg),\quad \bigg(\frac{\tau}{t}\bigg)^{\frac{2}{5}}=(tk_0^3)^{-\frac{1}{2}}
\bigg(\frac{\tau^{\frac{7}{2}}}{t}\bigg)^{\frac{1}{5}}.
\ee
Therefore, we have
\be
|u_{I}(x,t)-u_{II}(x,t)|=(tk_0^3)^{-\frac{1}{2}}O\bigg(\tau^{-\frac{1}{4}}+
\bigg(\frac{\tau^{\frac{7}{2}}}{t}\bigg)^{\frac{1}{5}}+\frac{\ln t}{t^{\frac{1}{7}}}+\tau^{-1}
(tk_0^3)^{\frac{1}{2}}+(k_0^3t)^{-\frac{1}{4}}\bigg)
\ee
which is of order $o(\sup|u_{as}(x,t)|)$ as $\tau\rightarrow\infty$ with $\tau=o(t^{\frac{2}{7}})$.
\end{remark}

\section{Asymptotic behavior in low regularity spaces}
\setcounter{equation}{0}
\setcounter{lemma}{0}
\setcounter{theorem}{0}
\setcounter{proposition}{0}
In this section, we first obtain the global well-posedness for the initial value problem of Equation \eqref{5mKdV}, and then extend the long-time asymptotic behavior of the solution to the low regularity spaces $H^{s,1}$, $19/22<s\leq1$.

$X^{s,b}$ method is extremely useful for studying the Cauchy problem of rough initial data. The motivation behind $X^{s,b}$ method is based on the dispersion relation. We consider the corresponding linear equation of fifth-order KdV equation:
$$\partial_tu +\partial_x^5 u=0.$$
Taking Fourier transform with respect to both space and time variable, we can get
$$(\tau-\xi^5)\widehat{u}(\xi,\tau)=0.$$

It is easy to see that $\widehat{u}(\xi,\tau)$ is supported on the surface $\{(\tau, \xi): \tau=\xi^5\}$. $\tau=\xi^5$ and $|\tau-\xi^5|$ are called the dispersion relation and dispersion modulation, respectively. Let $s,b\in \mathbb{R}$, the Bourgain space $X^{s,b}_{\tau=\xi^5}(\mathbb{R}\times \mathbb{R})$, or simply denoted as $X^{s,b}$, is defined to be the closure of the Schwartz functions $\mathscr{S}(\mathbb{R}\times \mathbb{R})$ under the norm
$$\|u\|_{X^{s,b}_{\tau=\xi^5}}=\big\|\langle\xi\rangle^s\langle\tau-\xi^5\rangle^b\widehat{u}(\xi,\tau)\big\|_{L^2_{\tau,\xi}}.$$
We will also need the truncated version of the above norm:
\begin{align}\label{restricted}
\|u\|_{X_\delta^{s,b}}=\inf_{\tilde{u}=u\ {\rm on}\ t\in[0,\delta]} \|\tilde{u}\|_{X^{s,b}}.
\end{align}

\subsection{Estimates in Bourgain spaces}

Let $\psi:\mathbb{R}\to[0,1]$ denote an even smooth function supported in $[-2,2]$ and equal to $1$ in $[-1,1]$. Denote $\psi_\mu(\cdot)=\psi(\cdot/\mu)$ for any $\mu>0$. $X^{s,b}$ method reduces to the following crucial estimates.
\begin{lemma}\label{estimates} For $s\in\mathbb{R}$ and $1/2<b<b'<1/2+\epsilon$, then
\begin{itemize}
  \item [(1)] (Embedding) For any $u\in X^{s,b}$, we have $\|u\|_{L^\infty(\mathbb{R}; H^s(\mathbb{R}))}\lesssim \|u\|_{X^{s,b}}$;
  \item [(2)] (Homogeneous estimate) $\big\|\psi(t)\e^{-t\partial_x^5}{u_0}\big\|_{X^{s,b}}\lesssim \|u_0\|_{H^s}$;
  \item [(3)] (Inhomogeneous estimate) $\big\|\psi(t)\int^t_0 \e^{-(t-t')\partial_x^5} {F(u)}(t')\, \dd t'  \big\|_{X^{s,b}}\lesssim \|F(u)\|_{X^{s,b-1}}$;
  \item [(4)] (Property) Let $0<\mu<1$, then $\|\psi_\mu(t)F\|_{X^{s,b-1}}\lesssim \mu^{b'-b}\|F\|_{X^{s,b'-1}}$;
  \item [(5)] (Nonlinear estimates) Let $s\geq 3/4$, we have
 $$
  \|u_1u_2\partial_x^3u_3\|_{X^{s,b'-1}}+\|\partial_xu_1\partial_xu_2\partial_xu_3\|_{X^{s,b'-1}}+\|u_1\partial_xu_2\partial_x^2u_3\|_{X^{s,b'-1}}\lesssim\prod_{j=1}^{3}\|u_j\|_{X^{s,b}};
 $$
  $$\|\partial_x(u_1u_2u_3u_4u_5)\|_{X^{s,b'-1}}\lesssim \prod_{j=1}^{5}\|u_j\|_{X^{s,b}}.$$
\end{itemize}
\end{lemma}
\begin{proof}
The proofs of (1-4) can be found in many literatures, we refer to \cite{KPV96,SK08,WaHaHuGu11}. The nonlinear estimate (5) needs condition $s\geq 3/4$, and its proof was followed by \cite{SK08}. $\hfill\Box$
\end{proof}

We recall some important estimates for the fifth-order KdV equations. These estimates are due to semi-group estimates and the Extension Lemma of \cite{WaHaHuGu11}, their proofs can be obtained by the methods in \cite{CuiTao,GuoPengWang,KPV91a,KPV91b}.
\begin{lemma} Let $u\in\mathscr{S}(\mathbb{R}\times\mathbb{R})$, then we have
\begin{itemize}
  \item[(1)](Strichartz estimates) Let $2\leq q,r\leq\infty$ and $2/q+1/r=1/2$, then
  \begin{align}\label{Str}
  \|D^{3/q}_xu\|_{L^q_tL^r_x}\lesssim \|u\|_{X^{0,\frac{1}{2}+}}.
  \end{align}
  \item[(2)](Local smoothing effect estimates)
  \begin{align}\label{Loc}
  \|D^2_xu\|_{L^\infty_xL^2_t}\lesssim \|u\|_{X^{0,\frac{1}{2}+}}.
  \end{align}
  \item[(3)](Maximal function estimates)
  \begin{align}\label{Max}
  \|u\|_{L^4_xL^\infty_t}\lesssim \|u\|_{X^{\frac{1}{4},\frac{1}{2}+}}.
  \end{align}
\end{itemize}
\end{lemma}
\begin{lemma}(Bilinear estimates) Suppose $u,v\in X^{0,\frac{1}{2}+}$ be supported on spatial frequencies on $|\xi|\sim N_1,N_2$, respectively. Let $N_1\gg N_2$, then we have
\begin{align}\label{bilinear}
\|uv\|_{L^2_{x,t}}\lesssim N_1^{-2}\|u\|_{X^{0,\frac{1}{2}+}}\|v\|_{X^{0,\frac{1}{2}+}}.
\end{align}
\end{lemma}
\subsection{Local well-posedness in $H^{s},s\geq 3/4$}
In this subsection, we use the $X^{s,b}$ method and the contraction principle to obtain the local existence of the solution in Sobolev space $H^s, s\geq3/4$. This result was  mainly stated by S. Kwon \cite{SK08}.
For the sake of completeness, we give the outline of its proof.
\begin{theorem}\label{LWP} Let $s\geq 3/4$ and $u_0\in H^s(\mathbb{R})$. Then there exists a time $T=T(\|u_0\|_{H^s(\mathbb{R})})>0$ such that \eqref{5mKdV} has a unique solution in $C([0,T];H^s(\mathbb{R}))$. Moreover, the solution map from data to the solutions is real-analytic.
\end{theorem}
\begin{proof} We construct the mapping:
\begin{align}
\mathfrak{M}: u(x,t)\longrightarrow \psi(t)\e^{-t\partial_x^5}{u_0} -\psi(t)\int^t_0 \e^{-(t-t')\partial_x^5}\psi_{T}(t'){F(u)}(t')\, \dd t',
\end{align}
and show that it is a contraction if $T<1$ is sufficiently small. Assume $u_0\in H^s(s\geq 3/4)$, we define the metric space:
\begin{align*}
\mathfrak{D}=\big\{u: \|u\|_{X^{s,b}}\leq 2C\|u_0\|_{H^s}\big\};\quad\quad
d(u,v)=\|u-v\|_{X^{s,b}}.
\end{align*}
Therefore, we have from the estimates in Lemma \ref{estimates} that
\begin{align}
\|\mathfrak{M}u\|_{X^{s,b}}&\leq C \|u_0\|_{H^s}+ C T^{b'-b}\|u\|^3_{X^{s,b}}+CT^{b'-b}\|u\|^5_{X^{s,b}}\nonumber\\
&\leq  C\|u_0\|_{H^s} + CT^{b'-b}\big((2C)^3\|u_0\|^2_{H^s} +(2C)^5 \|u_0\|^4_{H^s} \big)\|u_0\|_{H^s}.
\end{align}
Thus we choose $T$ small enough such that
$$T^{b'-b}\big((2C)^3\|u_0\|^2_{H^s} +(2C)^5 \|u_0\|^4_{H^s} \big)<1/2,$$
then we know that $\mathfrak{M}u\in \mathfrak{D}$. Similarly, assume $(u_1,u_2)\in \mathfrak{D}$, we have
$$\|\mathfrak{M}u_1-\mathfrak{M}u_2\|_{X^{s,b}}\leq 1/2 \|u_1-u_2\|_{X^{s,b}}.$$
This implies the local well-posedness for large data. By similar argument, we can get the solution map is even real-analytic. $\hfill\Box$ \end{proof}

\subsection{Global well-posedness in $H^s$, $19/22<s\leq 1$}

As mentioned in the introduction, the global well-posedness of \eqref{5mKdV} in $H^1(\mathbb{R})$ can be immediately obtained by energy conservation. One can conjecture that \eqref{5mKdV} is in fact globally well-posed in time from all initial data contained in the local theory. The biggest obstacle in getting global solutions in $H^s$ with $0<s<1$ is the lack of any conservation law. Therefore, in this subsection we introduce an ``almost conservation law" by utilizing $I$-method, then we can get the global well-posedness.

First of all, we introduce the definition of $I$-operator. Given $s<1$, assume $\varphi(\xi): \mathbb{R}\rightarrow \mathbb{R}$ is a smooth, even, real-valued function, which is equal to 1 in $[-1,1]$, and defined by $|\xi|^{s-1}$ in the set $\{\xi:|\xi|\geq 2\}$. For a parameter $N\gg1$, define
\begin{align}\label{symbol}
m_N(\xi):=\varphi(\xi/N)=
\begin{cases}
\ 1\quad\quad\quad &|\xi|\leq N;\\
\ \big(\frac{N}{|\xi|}\big)^{1-s} \quad\quad\quad &|\xi|\geq 2N.
\end{cases}
\end{align}
 and the $I$-operator as follows:
\begin{align}\label{operator}
I_N u=\mathscr{F}_\xi^{-1}m_N(\xi)\mathscr{F}u
\end{align}
For convenient, we will omit the parameter $N$, writing $m(\xi)$ and $Iu$ in \eqref{symbol} and \eqref{operator}.

Recall the energy $E(u)$ in \eqref{conserv}, the quantity $E(Iu)(t)$ can be compared with $\|u(\cdot,t)\|_{H^s}$, indeed, we have
\begin{gather}
E(Iu)(t)\lesssim N^{2(1-s)}\|u(\cdot,t)\|^2_{\dot{H}^{s}}+\|u(\cdot,t)\|^4_{L^4};\label{EIu}\\
\|u(\cdot,t)\|^2_{H^{s}}\lesssim E(Iu)(t)+\|u_0\|^2_{L^2}.\label{EIu2}
\end{gather}
Moreover, we will show that $E(Iu)(t)$ is an ``almost conservation law" in the $H^s$ level, and its increment can be described as the following proposition, whose proof shall be given later.
\begin{proposition}\label{conservation} Let $19/22<s<1$, $N\gg 1$. Given initial data $u_0\in C_0^\infty(\mathbb{R})$ with $E(Iu_0)\leq 1$, then there exists a $\delta=\delta(\|u_0\|_{L^2})>0$ so that the solution
\begin{align}
u(x,t)\in C([0,\delta], H^s(\mathbb{R}))
\end{align}
of \eqref{5mKdV} satisfies
\begin{align}
E(Iu)(t)=E(Iu)(0)+ O(N^{-1/2})\ \ \ \  {\text{for\ all}\ t\in[0,\delta]}.
\end{align}
\end{proposition}
{\bf Proof of Theorem \ref{the1.2}.} We only need to consider $19/22<s<1$. We may assume $u_0\in C_0^\infty(\mathbb{R})$, and then show that the solution in the $H^s$ norm grows at most polynomially, that is to say there exist $C_1,\ C_2$ and $M$, which depend only on $\|u_0\|_{H^s}$ such that
\begin{align}\label{grow}
\|u(\cdot,t)\|_{H^s}\leq C_1 t^M +C_2.
\end{align}
Then the global result is immediately obtained by \eqref{grow}, the local well-posedness and a standard density argument. From \eqref{EIu2}, we only need to show that
\begin{align}\label{grow1}
E(Iu)(t)\lesssim (1+t)^M.
\end{align}

Recall the scaling symmetry, let $0<\lambda<1$, then \eqref{EIu} and the Sobolev embedding inequality imply that
\begin{align}\label{EIu3}
E(Iu_{0,\lambda})\lesssim N^{2(1-s)}\lambda^{2s+1}\|u_{0}\|^2_{\dot{H}^{s}}+\lambda^3\|u_{0}\|^4_{L^4}\lesssim N^{2(1-s)}\lambda^{2s+1}\big(1+\|u_0\|_{H^{s}}\big)^4.
\end{align}
Assuming $N\gg 1$ is given, we choose $\lambda=\lambda(N,\|u_0\|_{H^{s}})$
\begin{align}\label{lambda}
\lambda=N^{-\frac{2(1-s)}{2s+1}}(2C_0)^{-\frac{1}{2s+1}}\big(1+\|u_0\|_{H^{s}}\big)^{-\frac{4}{2s+1}}
\end{align}
such that $E(Iu_{0,\lambda})\leq 1/2$, where $C_0$ is the implicit constant in \eqref{EIu3}. We apply Proposition \ref{conservation} at least $C N^{\frac{1}{2}}$ times to get
\begin{align}
E(Iu_\lambda)(CN^{\frac{1}{2}}\delta)\sim 1.
\end{align}

For any $T_0\gg 1$, we choose $N\gg 1$ so that
\begin{align}\label{T0}
T_0\sim \lambda^5 N^{\frac{1}{2}}\delta\sim N^{\frac{22s-19}{2(2s+1)}},
\end{align}
where we used \eqref{lambda}. In order to keep the exponent of $N$ is positive, we need the condition $s>19/22$. From the scaling symmetry, we have that
\begin{align}\label{Escaling}
E(Iu)(t)\leq \lambda^{-3} E(Iu_\lambda)(\lambda^{-5}t).
\end{align}
Therefore, from \eqref{lambda},\eqref{T0} and \eqref{Escaling}, for any $T_0\gg1$, we have
\begin{align}
E(Iu)(T_0)\leq C T_0^{\frac{12(1-s)}{22s-19}},
\end{align}
where $N$ is determined by \eqref{T0} and $C=C(\|u_0\|_{H^{s}},\delta)$. We now get the conclusion \eqref{grow1}, and complete the proof. $\hfill\Box$

To prove Proposition \ref{conservation}, we need the following variant local well-posedness result:
\begin{proposition}\label{variant}
Let $3/4\leq s<1$. Assume $u_0$ satisfies $E(Iu_0)\leq 1$, then there exists a constant $\delta=\delta(\|u_0\|_{L^2})$ and a unique solution $u$ to \eqref{5mKdV} such that
\begin{align}
\|Iu\|_{X_\delta^{1,\frac{1}{2}+}}\lesssim 1,
\end{align}
where the implicit constant is independent of $N$.
\end{proposition}

\begin{proof} We have the following estimates, which are analogous to Lemma \ref{estimates} (2-3),
\begin{gather}
 \big\|\e^{-t\partial_x^5}{u_0}\big\|_{X_\delta^{1,\frac{1}{2}+}}\lesssim \|u_0\|_{H^1};\label{s=1}\\
 \bigg\|\int^t_0 \e^{-(t-t')\partial_x^5} {F(u)}(t')\, \dd t'  \bigg\|_{X_\delta^{1,\frac{1}{2}+}}\lesssim \|F(u)\|_{X_\delta^{1,-\frac{1}{2}+}}.\label{s=1b}
\end{gather}
$Iu$ satisfies the integral equation on $t\in [0, \delta]$:
\begin{align*}
Iu(x,t)=\e^{-t\partial_x^5}I{u_0} -\int^t_0 \e^{-(t-t')\partial_x^5}\psi_\delta(t')I{F(u)}(t')\, \dd t'.
\end{align*}
By the definition of the restricted norm \eqref{restricted}, we can choose $\tilde{u}\in X^{1,\frac{1}{2}+}$, $\tilde{u}|_{[0,\delta]}=u$ such that
\begin{align}\label{Iuu}
\|I\tilde{u}\|_{X^{1,\frac{1}{2}+}}\sim\|Iu\|_{X_\delta^{1,\frac{1}{2}+}}.
\end{align}
Duhamel's principle,\eqref{s=1},\eqref{s=1b} and  (4) in Lemma \ref{estimates} give us
\begin{align}\label{Iu}
\|Iu\|_{X_\delta^{1,\frac{1}{2}+}}&\lesssim \|Iu_0\|_{H^1}+\big\|\psi_\delta(t)I{F(\tilde{u})}(t)\big\|_{X^{1,-\frac{1}{2}+}}\nonumber\\
&\lesssim \|Iu_0\|_{H^1}+\delta^\varepsilon\big\|I{F(\tilde{u})}(t)\big\|_{X^{1,-\frac{1}{2}++}}.
\end{align}
We divide each $\tilde{u}$ into a part supported on frequencies on $|\xi|\lesssim N$ and a part supported on frequencies $|\xi|\gg N$, then from nonlinear estimates in Lemma \ref{estimates} we immediately obtain that
\begin{align}
\big\|I{F(\tilde{u})}(t)\big\|_{X^{1,-\frac{1}{2}++}}\lesssim \|I\tilde{u}\|^3_{X^{1,\frac{1}{2}+}}+\|I\tilde{u}\|^5_{X^{1,\frac{1}{2}+}}.
\end{align}
Thus \eqref{Iuu} and \eqref{Iu} yield that
\begin{align}\label{continuous}
\|Iu\|_{X_\delta^{1,\frac{1}{2}+}}\lesssim \|Iu_0\|_{H^1}+\delta^\varepsilon\Big(\|Iu\|^3_{X_\delta^{1,\frac{1}{2}+}}+\|Iu\|^5_{X_\delta^{1,\frac{1}{2}+}}\Big).
\end{align}
It is easy to see that
\begin{align}\label{continuous2}
\|Iu_0\|_{H^1}\lesssim \big(E(Iu_0)\big)^{1/2}+\|u_0\|_{L^2}\leq 1+ \|u_0\|_{L^2}.
\end{align}
Since $Q(\delta):=\|Iu\|_{X_\delta^{1,\frac{1}{2}+}}$ is continuous in the variable $\delta$, a continuous argument implies the conclusion from \eqref{continuous} and \eqref{continuous2}. $\hfill\Box$\end{proof}

In the end, we turn to prove Proposition \ref{conservation}.

{\bf Proof of Proposition \ref{conservation}.}
By simple calculation, we can change all nonlinear terms in \eqref{5mKdV} into the divergence form, that is to say \eqref{5mKdV} becomes the following form:
\begin{align}\label{5mKdV1}
\partial_tu +\partial_x^5 u+6\partial_x(u^5)-5\partial_x\big(u\partial^2_x(u^2)\big)=0.
\end{align}
The energy in \eqref{conserv} is known to be conserved by differentiating in time, using the equation \eqref{5mKdV1} and integrating by parts, we have that
\begin{align}
\frac{\dd}{\dd t}E(u)&=\int_\mathbb{R} (-2\partial^2_x u+4u^3)\cdot\partial_tu\  \dd x\nonumber\\
&=\int_\mathbb{R} (-2\partial^2_x u+4u^3)\Big(-\partial_x^5 u-6\partial_x(u^5)+5\partial_x\big(u\partial^2_x(u^2)\big)\Big)\  \dd x\nonumber\\
&=2\int_\mathbb{R}\partial^2_x u\partial_x^5 u\ \dd x-24\int_\mathbb{R}u^3\partial_x(u^5)\ \dd x\nonumber\\
&\quad-2\int_\mathbb{R}2u^3\partial_x^5 u+5\partial^2_x u\partial_x\big(u\partial^2_x(u^2)\big)\ \dd x+4\int_\mathbb{R}3\partial^2_x u\partial_x(u^5)+5u^3\partial_x\big(u\partial^2_x(u^2)\big)\ \dd x\nonumber\\
&:=T_1+T_2+T_3+T_4=0,
\end{align}
In fact, we can get $T_1=T_2=T_3=T_4=0$ only by integrating by parts, which will be used in the following discussion. We apply $I$ to \eqref{5mKdV1}, then
\begin{align}\label{EqI}
\partial_tIu=-\partial_x^5 Iu-6\partial_xI(u^5)+5\partial_xI\big(u\partial^2_x(u^2)\big).
\end{align}
Following the same strategy, we estimate the growth of $E(Iu)(t)$. By the equation \eqref{EqI}, we know that
\begin{align}\label{EIu4}
\frac{\dd}{\dd t}E(Iu)(t)&=\int_\mathbb{R} (-2\partial^2_x Iu+4(Iu)^3)\cdot\partial_tIu\  \dd x\nonumber\\
&=\int_\mathbb{R} (-2\partial^2_x Iu+4(Iu)^3)\Big(-\partial_x^5 Iu-6\partial_xI(u^5)+5\partial_xI\big(u\partial^2_x(u^2)\big)\Big)\  \dd x\nonumber\\
&=2\int_\mathbb{R}\partial^2_x Iu\cdot\partial_x^5 Iu\ dx+72\int_\mathbb{R}\partial_xIu\cdot(Iu)^2\cdot I(u^5)\ \dd x\nonumber\\
&\quad+4\int_\mathbb{R}3\partial^2_x Iu\cdot\partial_xI(u^5)+5(Iu)^3\cdot\partial_xI\big(u\partial^2_x(u^2)\big)\ \dd x\nonumber\\
&\quad-2\int_\mathbb{R}2(Iu)^3\cdot\partial_x^5 Iu+5\partial^2_x Iu\cdot\partial_xI\big(u\partial^2_x(u^2)\big)\ \dd x\nonumber\\
&:=T_{I1}+T_{I2}+T_{I3}+T_{I4}.
\end{align}
It is easy to see that
$$T_{I1}=-2\int_\mathbb{R}\partial^3_x Iu\cdot\partial_x^4 Iu\ dx=-\int_\mathbb{R}\partial_x\Big((\partial_x^3 Iu)^2\Big)\ dx=0.$$
We recall the $k$-Parseval formula:
\begin{align}
\int_\mathbb{R} f_1(x)f_2(x)\cdots f_k(x)\ \dd x=\int_{\xi_1+\xi_2+\cdots+\xi_k=0}\widehat{f_1}(\xi_1)\widehat{f_2}(\xi_2)\cdots\widehat{f_k}(\xi_k),
\end{align}
For simplicity, denote $\Gamma_k=\{(\xi_1,\xi_2,\cdots,\xi_k)\in\mathbb{R}^k, \xi_1+\xi_2+\cdots+\xi_k=0\}$.
Therefore, $k$-Parseval formula combined with integrating by parts, yields that
\begin{align}
T_{I2}=&72i\int_{\Gamma_8}\xi_1\cdot\frac{m(\xi_4+\xi_5+\xi_6+\xi_7+\xi_8)}{m(\xi_4)m(\xi_5)m(\xi_6)
m(\xi_7)m(\xi_8)}\widehat{Iu}(\xi_1)\widehat{Iu}(\xi_2)\cdots\widehat{Iu}(\xi_8);\label{I2}\\
T_{I3}=&12i\int_{\Gamma_6}\bigg(\xi_1^3\frac{m(\xi_2+\xi_3+\cdots+\xi_6)}{m(\xi_2)m(\xi_3)\cdots m(\xi_6)}+5\xi_1(\xi_2+\xi_3)^2\cdot\frac{m(\xi_2+\xi_3+\xi_4)}{m(\xi_2)m(\xi_3)m(\xi_4)}\bigg)\nonumber\\
&\quad\quad\quad\quad\quad\quad\quad\quad\quad\quad\quad\quad\quad\quad\quad\quad\quad\quad\widehat{Iu}(\xi_1)
\widehat{Iu}(\xi_2)\cdots\widehat{Iu}(\xi_6);\label{I3}\\
T_{I4}=&-2i\int_{\Gamma_4}\bigg(2\xi_1^5-5\xi^3_1(\xi_2+\xi_3)^2\cdot\frac{m(\xi_2+\xi_3+\xi_4)}{m(\xi_2)
m(\xi_3)m(\xi_4)}\bigg)\widehat{Iu}(\xi_1)\widehat{Iu}(\xi_2)\cdots\widehat{Iu}(\xi_4).\label{I4}
\end{align}
Integrating in time for \eqref{EIu4}, it suffices to control
\begin{align}\label{term3}
E(Iu(\delta))-E(Iu(0))=\int_0^\delta T_{I2}+T_{I3}+T_{I4}.
\end{align}
The term $T_{I2}$ is the easiest one to estimate, since there are fewer derivatives and more $u$'s. On the contrary, the term $T_{I4}$ is the worst one. We will consider every term in the following three propositions. Have these propositions in hand, we can complete the proof by utilizing Proposition \ref{variant}. $\hfill\Box$

 It remains to control every term in \eqref{term3}. Before that, we give some observations. The same argument as $T_2=T_3=T_4=0$, using $k$-Parseval formula and integrating by parts, show that
\begin{gather}
\int_{\Gamma_8}\xi_1\widehat{Iu}(\xi_1)\widehat{Iu}(\xi_2)\cdots\widehat{Iu}(\xi_8)=0;\label{obs2}\\
\int_{\Gamma_6}\big(\xi_1^3+5\xi_1(\xi_2+\xi_3)^2\big)\widehat{Iu}(\xi_1)\widehat{Iu}(\xi_2)\cdots\widehat{Iu}(\xi_6)=0;\label{obs3}\\
\int_{\Gamma_4}\big(2\xi_1^5-5\xi^3_1(\xi_2+\xi_3)^2\big)\widehat{Iu}(\xi_1)\widehat{Iu}(\xi_2)\cdots\widehat{Iu}(\xi_4)=0.\label{obs4}
\end{gather}

\begin{proposition} Let $3/4\leq s<1$ and $T_{I2}$ be as in \eqref{I2}, then
\begin{align}
\bigg|\int_0^\delta T_{I2}\bigg|\lesssim N^{-5}\|Iu\|^8_{X_\delta^{1,\frac{1}{2}+}}.
\end{align}
\end{proposition}
\begin{proof} We divide $Iu$ into a sum of dyadic parts $P_{k}Iu$, whose frequency is supported on $\{\xi: |\xi|\sim 2^k\}$, $k=0,1,2,\cdots$. It suffices to show that
\begin{align}\label{N8}
\sum_{N_1,\cdots,N_8}\bigg|\int_0^\delta \int_{\Gamma_8}\xi_1\cdot\frac{m(\xi_4+\cdots+\xi_8)}{m(\xi_4)\cdots m(\xi_8)}\widehat{u_1}(\xi_1)\cdots\widehat{u_8}(\xi_8)\bigg|\lesssim N^{-5}\prod_{j=1}^8\|u_j\|_{X_\delta^{1,\frac{1}{2}+}},
\end{align}
for any function $u_j$, $j=1,2,\cdots,8$ with the frequency supported on $|\xi_j|\sim 2^{k_j}\equiv N_j$, $k_j\in \{0,1,\cdots\}$. We may assume all $\widehat{u_j}$'s are non-negative. In the following discussion, we will pull the symbol out of the integral, then reverse the $k$-Parseval formula and use H\"older's inequality to estimate the remaining integrals.

Note that the derivative is located in $u_1$, and by the symmetry of the above multiplier in $\xi_4,\xi_5,\cdots,\xi_8$, we may assume that
\begin{align}
N_1\geq N_2\geq N_3,\quad\quad\quad N_4\geq N_5\geq \cdots \geq N_8.
\end{align}
Let $N_{max}$ and $N_{sec}$ denote the maximum and the second maximum of the numbers $N_1,N_2,\cdots,N_8$. From the constraint condition $\Gamma_8=\big\{\sum_{j=1}^8\xi_j=0\big\}$ in the above integration, we know that $$N_{max}\sim N_{sec}.$$
Since $m(\xi)=1$ for $\xi\leq N$, \eqref{I2} and \eqref{obs2} shows that $T_{I2}$ vanishes when $|\xi_4|\leq N/5$. Thus, we may assume that
$$N_4 \gtrsim N.$$
According to $N_{max}$ and $N_{sec}$, we split the proof into three cases: $\{N_{max},N_{sec}\}=\{N_1,N_4\}$, $\{N_{max},N_{sec}\}=\{N_4,N_5\}$, $\{N_{max},N_{sec}\}=\{N_1,N_2\}$. For the sake of brevity, we only consider the first case since the other two cases can be obtained by the same techniques.

If $\{N_{max},N_{sec}\}=\{N_1,N_4\}$, we have here $N_1\sim N_4\gtrsim N$. Using local smoothing effect estimates \eqref{Loc} to $u_1$, $u_4$, maximal function estimates \eqref{Max} to $u_5$-$u_8$, it follows that
\begin{align*}
&{\rm LHS\  of}\  \eqref{N8}\nonumber\\
\lesssim&\sum_{N_1\sim N_4\gtrsim N \atop N1,\cdots,N_8} \frac{N_1N_4^{1-s}}{N^{1-s}m(N_5)\cdots m(N_8)} \|u_1\|_{L^\infty_xL^2_t}\|u_4\|_{L^\infty_xL^2_t}\prod^3_{i=2}\|u_i\|_{L^\infty_xL^\infty_t}\prod^8_{j=5}\|u_j\|_{L^4_xL^\infty_t}\nonumber\\
\lesssim&\sum_{N_1\sim N_4\gtrsim N \atop N1,\cdots,N_8} \frac{N_1N_4^{1-s}}{N^{1-s}m(N_5)\cdots m(N_8)} \|u_1\|_{X^{-2,\frac{1}{2}+}_\delta}\|u_4\|_{X^{-2,\frac{1}{2}+}_\delta}\prod^3_{i=2}\|u_i\|_{X^{\frac{1}{2}+,\frac{1}{2}+}_\delta}\prod^8_{j=5}\|u_j\|_{X^{\frac{1}{4},\frac{1}{2}+}_\delta}\nonumber\\
\lesssim&N^{s-1}\sum_{N_1\sim N_4\gtrsim N \atop N1,\cdots,N_8}N_1^{-2}N_4^{-2-s}(N_2N_3)^{-\frac{1}{2}+}(N_5N_6N_7N_8)^{-\frac{1}{2}}\prod^8_{j=1}\|u_j\|_{X^{1,\frac{1}{2}+}_\delta}
\lesssim N^{-5}\prod_{j=1}^8\|u_j\|_{X_\delta^{1,\frac{1}{2}+}},
\end{align*}
where we used Sobolev embedding and Strichart estimate to get $\|u_i\|_{L^\infty_tL^\infty_x}\lesssim \|\langle\nabla\rangle^{\frac{1}{2}+}u_i\|_{L^\infty_tL^2_x}\lesssim\|u_i\|_{X^{\frac{1}{2}+,\frac{1}{2}+}_\delta}$. Since $3/4\leq s<1$, we know that $m(N_j)N_j^{1/4}\geq 1$ for $j=5,\cdots,8$. $\hfill\Box$\end{proof}

\begin{proposition} Let $3/4\leq s<1$ and $T_{I3}$ be as in \eqref{I3}, then
\begin{align}
\bigg|\int_0^\delta T_{I3}\bigg|\lesssim N^{-3}\|Iu\|^6_{X_\delta^{1,\frac{1}{2}+}}.
\end{align}
\end{proposition}
\begin{proof} At first, for getting some vanishing properties, we insert the zero term \eqref{obs3} into $T_{I3}$. Then as we discussed in the above proposition, it suffices to control
\begin{align*}
I_1:=\sum_{N_1,\cdots,N_6}\bigg|\int_0^\delta \int_{\Gamma_6}\xi_1^3\Big(\frac{m(\xi_2+\cdots+\xi_6)}{m(\xi_2)\cdots m(\xi_6)}-1\Big)\widehat{u_1}(\xi_1)\cdots\widehat{u_6}(\xi_6)\bigg|,
\end{align*}
and
\begin{align*}
I_2:=\sum_{N_1,\cdots,N_6}\bigg|\int_0^\delta \int_{\Gamma_6}5\xi_1(\xi_2+\xi_3)^2\cdot\Big(\frac{m(\xi_2+\xi_3+\xi_4)}{m(\xi_2)m(\xi_3)m(\xi_4)}-1\Big)\widehat{u_1}(\xi_1)\cdots\widehat{u_6}(\xi_6)\bigg|,
\end{align*}
for any function $u_j$, $j=1,2,\cdots,6$ with the frequency supported on $|\xi_j|\sim N_j$.

For term $I_1$, by the symmetry of the multiplier in $\xi_2,\xi_3,\cdots,\xi_6$, we can assume that
\begin{align}
 N_2\geq N_3\geq \cdots \geq N_6,\quad\quad\quad N_2\gtrsim N,
\end{align}
where we use that if $N_2\leq N/10$,  $I_1$ vanishes. Because of $N_{max}\sim N_{sec}$, we divide the proof into two cases: $N_1\sim N_2\gtrsim N$ and $N_2\sim N_3\gtrsim N$, $N_3\geq N_1$.

If $N_1\sim N_2\gtrsim N$, we take local smoothing effect estimates \eqref{Loc} to $u_1$, $u_2$, maximal function estimates \eqref{Max} to $u_3$-$u_6$, and then get that
\begin{align}\label{I_1}
I_1\lesssim&\sum_{N_1\sim N_2\gtrsim N \atop N1,\cdots,N_6} \frac{N_1^3 m(N_1)}{m(N_2)\cdots m(N_6)} \|u_1\|_{L^\infty_xL^2_t}\|u_2\|_{L^\infty_xL^2_t}\prod^6_{j=3}\|u_j\|_{L^4_xL^\infty_t}\nonumber\\
\lesssim&\sum_{N_1\sim N_2\gtrsim N \atop N1,\cdots,N_6} \frac{N_1^3}{m(N_3)\cdots m(N_6)} \|u_1\|_{X^{-2,\frac{1}{2}+}_\delta}\|u_2\|_{X^{-2,\frac{1}{2}+}_\delta}\prod^6_{j=3}\|u_j\|_{X^{\frac{1}{4},\frac{1}{2}+}_\delta}\nonumber\\
\lesssim&\sum_{N_1\sim N_2\gtrsim N \atop N1,\cdots,N_6}N_2^{-3}(N_3N_4N_5N_6)^{-\frac{1}{2}}\prod^6_{j=1}\|u_j\|_{X^{1,\frac{1}{2}+}_\delta}
\lesssim N^{-3}\prod_{j=1}^6\|u_j\|_{X_\delta^{1,\frac{1}{2}+}}.
\end{align}

If $N_2\sim N_3\gtrsim N$, $N_3\geq N_1$, we can get similarly that
\begin{align}\label{I_12}
I_1
\lesssim&\sum_{N_2\sim N_3\gtrsim N \atop N1,\cdots,N_6} \frac{N_1^3 m(N_1)}{m(N_2)\cdots m(N_6)} \|u_2\|_{L^\infty_xL^2_t}\|u_3\|_{L^\infty_xL^2_t}\|u_1\|_{L^4_xL^\infty_t}\prod^6_{j=4}\|u_j\|_{L^4_xL^\infty_t}\nonumber\\
\lesssim&N^{1-s}\sum_{N_2\sim N_3\gtrsim N\atop N1,\cdots,N_6} \frac{N_1^{s+2}}{m(N_2)\cdots m(N_6)} \|u_1\|_{X^{\frac{1}{4},\frac{1}{2}+}_\delta}\|u_2\|_{X^{-2,\frac{1}{2}+}_\delta}\|u_3\|_{X^{-2,\frac{1}{2}+}_\delta}\prod^6_{j=4}\|u_j\|_{X^{\frac{1}{4},\frac{1}{2}+}_\delta}\nonumber\\
\lesssim&N^{1-s}\sum_{N_2\sim N_3\gtrsim N\atop N1,\cdots,N_6}N_1^{s+\frac{5}{4}}N_2^{-\frac{11}{4}}N_3^{-\frac{11}{4}}(N_4N_5N_6)^{-\frac{1}{2}}\prod^6_{j=1}\|u_j\|_{X^{1,\frac{1}{2}+}_\delta}\nonumber\\
\lesssim& N^{-3}\prod_{j=1}^6\|u_j\|_{X_\delta^{1,\frac{1}{2}+}}.
\end{align}

For term $I_2$, the same argument as before, we may assume that
\begin{align}
 N_1\geq N_5\geq N_6,\quad\quad  N_2\geq N_3\geq N_4,\quad\quad N_2\gtrsim N.
\end{align}
If $N_1\sim N_2\gtrsim N$, the symbol is controlled by $N_1^3 m(N_1)/m(N_2)m(N_3)m(N_4)$, which is smaller than $N_1^3 m(N_1)/m(N_2)\cdots m(N_6)$, therefore, we can get the same bound as \eqref{I_1}. If $N_2\sim N_3\gtrsim N$, using the same techniques as \eqref{I_12} yields the estimates. If $N_1\sim N_5\gtrsim N$, this case is better, thus we can obtain the conclusion as before. $\hfill\Box$\end{proof}

\begin{proposition} Let $3/4\leq s<1$ and $T_{I4}$ be as in \eqref{I4}, then
\begin{align}
\bigg|\int_0^\delta T_{I4}\bigg|\lesssim N^{-\frac{1}{2}}\|Iu\|^4_{X_\delta^{1,\frac{1}{2}+}}.
\end{align}
\end{proposition}
\begin{proof} As the above proposition, inserting the zero term \eqref{obs4} into $T_{I4}$, it suffices to control
\begin{align}\label{I4}
\sum_{N_1,\cdots,N_4}\bigg|\int_0^\delta \int_{\Gamma_4}\xi^3_1(\xi_2+\xi_3)^2\cdot\Big(\frac{m(\xi_2+\xi_3+\xi_4)}{m(\xi_2)m(\xi_3)m(\xi_4)}-1\Big)\widehat{u_1}(\xi_1)\cdots\widehat{u_4}(\xi_4)\bigg|,
\end{align}
for any function $u_j$, $j=1,2,3,4$ with the frequency supported on $|\xi_j|\sim N_j$. A glance at the symbol, we may assume that
\begin{align}
 N_2\geq N_3\geq N_4,\quad\quad\quad N_2\gtrsim N,
\end{align}
since that if $N_2\leq N/8$, the symbol vanishes. The other orders of $N_2$, $N_3$ and $N_4$ are similar and much easier. Then we control the above summation by considering the following three cases:
\begin{itemize}
  \item [1.] the non-resonant case, where $N_2\gg N_3\geq N_4$;
  \item [2.] the semi-resonant case, where $N_2\sim N_3\gg N_4$;
  \item [3.] the resonant case, where $N_2\sim N_3\sim N_4$.
\end{itemize}
1. In the non-resonant case, from the constraint condition $\Gamma_4=\big\{\sum_{j=1}^4\xi_j=0\big\}$, it follows immediately that $N_1\sim N_2\gg N_3\geq N_4$. Utilizing bilinear estimates \eqref{bilinear} can efficiently control \eqref{I4} by
\begin{align}
&\sum_{N_1\sim N_2\gtrsim N\atop N_1\sim N_2\gg N_3\geq N_4} \frac{N_1^3N_2^2 m(N_1)}{m(N_2)m(N_3)m(N_4)} \|u_1u_3\|_{L^2_{x,t}}\|u_2u_4\|_{L^2_{x,t}}\nonumber\\
\lesssim& \sum_{N_1\sim N_2\gtrsim N\atop N_1\sim N_2\gg N_3\geq N_4} \frac{N_1}{m(N_3)m(N_4)} \prod^4_{j=1}\|u_j\|_{X_\delta^{0,\frac{1}{2}+}}\nonumber\\
\lesssim& \sum_{N_1\sim N_2\gtrsim N\atop N_1\sim N_2\gg N_3\geq N_4} N_2^{-1}(N_3N_4)^{-\frac{3}{4}} \prod^4_{j=1}\|u_j\|_{X_\delta^{1,\frac{1}{2}+}}\lesssim N^{-1}\prod^4_{j=1}\|u_j\|_{X_\delta^{1,\frac{1}{2}+}}.
\end{align}
2. In the semi-resonant case, it is easy to see that $N_2\sim N_3\gg N_4\sim N_1$. Bilinear estimates \eqref{bilinear} imply that \eqref{I4} is controlled by
\begin{align}
&\sum_{N_2\sim N_3\gtrsim N\atop N_2\sim N_3\gg N_4\sim N_1} \frac{N_1^3N_2^2 m(N_1)}{m(N_2)m(N_3)m(N_4)} \|u_1u_3\|_{L^2_{x,t}}\|u_2u_4\|_{L^2_{x,t}}\nonumber\\
\lesssim& \sum_{N_2\sim N_3\gtrsim N\atop N_2\sim N_3\gg N_4\sim N_1} \frac{N_1^3N_3^{-2}}{m(N_2)m(N_3)} \prod^4_{j=1}\|u_j\|_{X_\delta^{0,\frac{1}{2}+}}\nonumber\\
\lesssim& \sum_{N_2\sim N_3\gtrsim N\atop N_2\sim N_3\gg N_4\sim N_1} N_1N_2^{-\frac{3}{4}}N_3^{-\frac{11}{4}} \prod^4_{j=1}\|u_j\|_{X_\delta^{1,\frac{1}{2}+}}\lesssim N^{-\frac{5}{2}}\prod^4_{j=1}\|u_j\|_{X_\delta^{1,\frac{1}{2}+}}.
\end{align}
3. In the resonant case, we have $N_1\sim N_2\sim N_3\sim N_4$. Strichartz estimates \eqref{Str} yield that
\begin{align}
\eqref{I4}\lesssim
&\ \delta^{\frac{1}{2}}\sum_{N_1\sim N_2\sim N_3\sim N_4\gtrsim N} \frac{N_1^3N_2^2 m(N_1)}{m(N_2)m(N_3)m(N_4)} \prod^4_{j=1}\|u_j\|_{L^8_tL^4_x}\nonumber\\
\lesssim&\ \delta^{\frac{1}{2}} \sum_{N_1\sim N_2\sim N_3\sim N_4\gtrsim N} \frac{N_1^3N_2^2}{m(N_3)m(N_4)} \prod^4_{j=1}\|u_j\|_{X_\delta^{-\frac{3}{8},\frac{1}{2}+}}\nonumber\\
\lesssim&\ \delta^{\frac{1}{2}} N^{2s-2}\sum_{N_1\gtrsim N} N_1^{\frac{3}{2}-2s}\prod^4_{j=1}\|u_j\|_{X_\delta^{1,\frac{1}{2}+}}\lesssim \delta^{\frac{1}{2}} N^{-\frac{1}{2}}\prod^4_{j=1}\|u_j\|_{X_\delta^{1,\frac{1}{2}+}}.
\end{align}
We now complete the proof. $\hfill\Box$
\end{proof}

\subsection{Global approximation in low regularity spaces}
In this subsection, we shall extend the long-time asymptotic behavior to the rough data.  At the beginning, we show that the Beals-Coifman solution is equal to the strong solution given by \eqref{I5mKdV}.
\begin{lemma}\label{LEqual}
Let $u_{0}\in H^{4,1}(\mathbb{R})$, the Beals-Coifman solution and the strong solution are the same
(up to a measure zero set)
\begin{align*}
u =\frac{1}{\pi}\left(\int\mu\big(w^{+}+w^{-}\big)\right)_{12}
 =\e^{-t\partial_x^5}{u_0} -\int^t_0 \e^{-(t-t')\partial_x^5}{F(u)}(t')\, \dd t'
\end{align*}
in $[-T,T]$ where $T$ is given as in Theorem \ref{LWP}.
\end{lemma}
\begin{proof}  For $u_{0}\in H^{4,1}(\mathbb{R})$, there exists a Cauchy sequence $u_{0,k}\in \mathscr{S}(\mathbb{R})$, such that
\begin{align*}
\lim_{k\rightarrow\infty}\|u_{0,k}-u_0\|_{H^{4,1}}=0, \quad  \text{and} \quad \|u_{0,k}\|_{H^4}\leq \|u_{0,k}\|_{H^{4,1}}\leq C.
\end{align*}
On the one hand, from Theorem \ref{LWP}, there exist a strong solution $u_k$ with initial data $u_{0,k}$ such that for $b>1/2$,
\begin{align*}
\|u_k-u_l\|_{X^{4,b}}\lesssim \|u_{0,k}-u_{0,l}\|_{H^4}\rightarrow 0 \quad \text{as} \quad k,l\rightarrow \infty,
\end{align*}
Therefore, there exists $u_\infty$ such that
\begin{align}\label{BS1}
\sup_{t\in [-T,T]}\|u_k-u_\infty\|_{L^\infty}\lesssim \|u_k-u_\infty\|_{X^{4,b}}\rightarrow 0, \quad \text{as} \quad k\rightarrow \infty,
\end{align}
On the other hand, from the inverse scattering transform, we have the Beals-Coifman solutions $\tilde{u}_k$ with initial data $u_{0,k}$. By utilizing the bijectivity of the transformation and bi-Lipschitz continuity, we know reflection coefficients satisfy
\begin{align*}
\|r_k-r_l\|_{H^{1,4}}\lesssim \|u_{0,k}-u_{0,l}\|_{H^{4,1}}\rightarrow 0 \quad as \quad k,l\rightarrow \infty.
\end{align*}
By the resolvent estimates,
\begin{align*}
\|\tilde{u}_k-\tilde{u}_l\|_{L^\infty}\lesssim \|r_k-r_l\|_{H^1}\rightarrow 0 \quad as \quad k,l\rightarrow \infty.
\end{align*}
Therefore, there exists $\tilde{u}_\infty$ such that
\begin{align}\label{BS2}
\tilde{u}_\infty=\lim_{k\rightarrow\infty}\tilde{u}_k.
\end{align}
Since $u_{0,k}\in \mathscr{S}(\mathbb{R})$, $u_k$ and $\tilde{u}_k$ are also Schwartz functions, then $u_k=\tilde{u}_k$.
Therefore, we can get the conclusion from \eqref{BS1} and \eqref{BS2}. $\hfill\Box$\end{proof}

In the end, we divide the proof of Theorem \ref{the1.3} into the following two theorems. We first consider the long-time behavior in the Sobolev space $H^{1,1}(\mathbb{R})$. It is well known that $u(\cdot,t)$ is uniformly bounded in $H^1(\mathbb{R})$. Indeed, from the mass conservation and energy conservation, we know that
\begin{align*}
\|u\|_{H^1(\mathbb{R})}=\|u\|_{L^2}+\|u\|_{\dot{H}^1}\leq \|u_0\|_{L^2}+\|\partial_x u_0\|_{L^2}+\|u_0\|^2_{L^4}.
\end{align*}
From Gagliardo-Nirenberg's inequality and Young's inequality,
\begin{align*}
\|u_0\|^2_{L^4}\lesssim \|\partial_x u_0\|^{1/2}_{L^2}\|u_0\|^{3/2}_{L^2}\lesssim \frac{1}{\epsilon}\|\partial_x u_0\|_{L^2}+\epsilon \|u_0\|^3_{L^2}.
\end{align*}
We may choose $\epsilon$ satisfying $\epsilon \|u_0\|^2_{L^2}\leq 1$, and get that
\begin{align}\label{H1}
\|u(\cdot,t)\|_{H^1}\lesssim \|u_0\|_{H^1}.
\end{align}

\begin{theorem}
Let  $u_{0}\in H^{1,1}(\mathbb{R})$, the strong solution given by the integral form \eqref{I5mKdV} has the same asymptotic behavior as in Theorem \ref{the1.1}.
\end{theorem}
\begin{proof}  For $u_{0}\in H^{1,1}(\mathbb{R})$, there exists a sequence $u_{0,k}\in H^{4,1}(\mathbb{R})$, such that
\begin{align*}
\lim_{k\rightarrow\infty}\|u_{0,k}-u_0\|_{H^{1,1}}=0, \quad  and \quad \|u_{0,k}\|_{H^1}\leq \|u_{0,k}\|_{H^{1,1}}\leq C.
\end{align*}
Following the same strategy as in the proof of Lemma \ref{LEqual}, we know that there exist a pointwise limit $u_\infty$ of the strong solutions $u_k$,  and a pointwise limit $\tilde{u}_\infty$ of Beals-Coifman solutions $\tilde{u}_k$, with initial data $u_{0,k}$. Moreover, one has $u_\infty=\tilde{u}_\infty$ pointwise for $t\in [-T,T]$. Since $u(\cdot,t)$ is uniformly bounded in $H^1(\mathbb{R})$ as \eqref{H1}, one can repeat the above argument many times to extend the time interval to whole line $\mathbb{R}$. Because $\tilde{u}_\infty$ has asymptotic behavior in Theorem \ref{the1.1}, $u_\infty$ also does. $\hfill\Box$\end{proof}
\begin{theorem}
Let  $u_{0}\in H^{s,1}(\mathbb{R})$, $19/22<s<1$, the strong solution given by the integral form \eqref{I5mKdV} has the same asymptotic behavior as in Theorem \ref{the1.1}.
\end{theorem}
\begin{proof}  For $u_{0}\in H^{s,1}(\mathbb{R})$, $19/22<s<1$, there exists a sequence $u_{0,k}\in H^{4,1}(\mathbb{R})$, such that
\begin{align*}
\lim_{k\rightarrow\infty}\|u_{0,k}-u_0\|_{H^{s,1}}=0, \quad  and \quad \|u_{0,k}\|_{H^s}\leq \|u_{0,k}\|_{H^{s,1}}\leq C.
\end{align*}
Denote $u_\infty$ and $\tilde{u}_\infty$ the pointwise limit of the strong solutions $u_k$ and Beals-Coifman solutions $\tilde{u}_k$ with initial data $u_{0,k}$, respectively. As Lemma \ref{LEqual}, we know that $u_\infty=\tilde{u}_\infty$ pointwise for $t\in [-T,T]$.

From the global well-posedness theory Theorem \ref{the1.2}, $u_\infty$ exists in $H^s$ globally. One can also construct $\tilde{u}_\infty$ globally. Suppose that $u_\infty=\tilde{u}_\infty$ does not hold for all $t$. We may assume
\begin{align}\label{contradiction}
T_*=\inf\big\{t\geq 0; u_\infty(t)\neq\tilde{u}_\infty(t)\big\}.
\end{align}
From \eqref{grow}, we know that
\begin{align}
\|u(\cdot,T_*)\|_{H^s}\leq C_1 T_*^M +C_2,
\end{align}
then as in the proof of Theorem \ref{LWP}, there exists $t_0>0$ such that
\begin{align*}
\sup_{t\in [0,T_*+t_0]}\|u_k-u_\infty\|_{L^\infty}\lesssim \|u_k-u_\infty\|_{X^{s,b}}\rightarrow 0, \quad as \quad k\rightarrow \infty,
\end{align*}
Combining
\begin{align*}
\sup_{t\in [0,T_*+t_0]}\|\tilde{u}_k-\tilde{u}_\infty\|_{L^\infty}\rightarrow 0, \quad as \quad k\rightarrow \infty,\quad  and \quad u_k=\tilde{u}_k,
\end{align*}
we can conclude that $u_\infty(t)=\tilde{u}_\infty(t)$ pointwise for $t\in [0, T_*+t_0]$,  this is a contradiction with \eqref{contradiction}. Hence we can obtain that $u_\infty(t)=\tilde{u}_\infty(t)$ pointwise for all $t\geq 0$. Since $\tilde{u}_\infty$ has asymptotic behavior in Theorem \ref{the1.1}, $u_\infty$ also does. $\hfill\Box$
\end{proof}

{\bf Acknowledgments.}

N. Liu is supported by the China Postdoctoral Science Foundation under Grant Nos. 2019TQ0041, 2019M660553; M. Chen is supported in part by the China Postdoctoral Science Foundation under Grant No. 2019M650019.

\appendix
\section{Fourth order Painlev\'e II RH problem}
\setcounter{equation}{0}
\setcounter{lemma}{0}
\setcounter{theorem}{0}
\setcounter{proposition}{0}
\begin{figure}[htbp]
\centering
\includegraphics[width=3in]{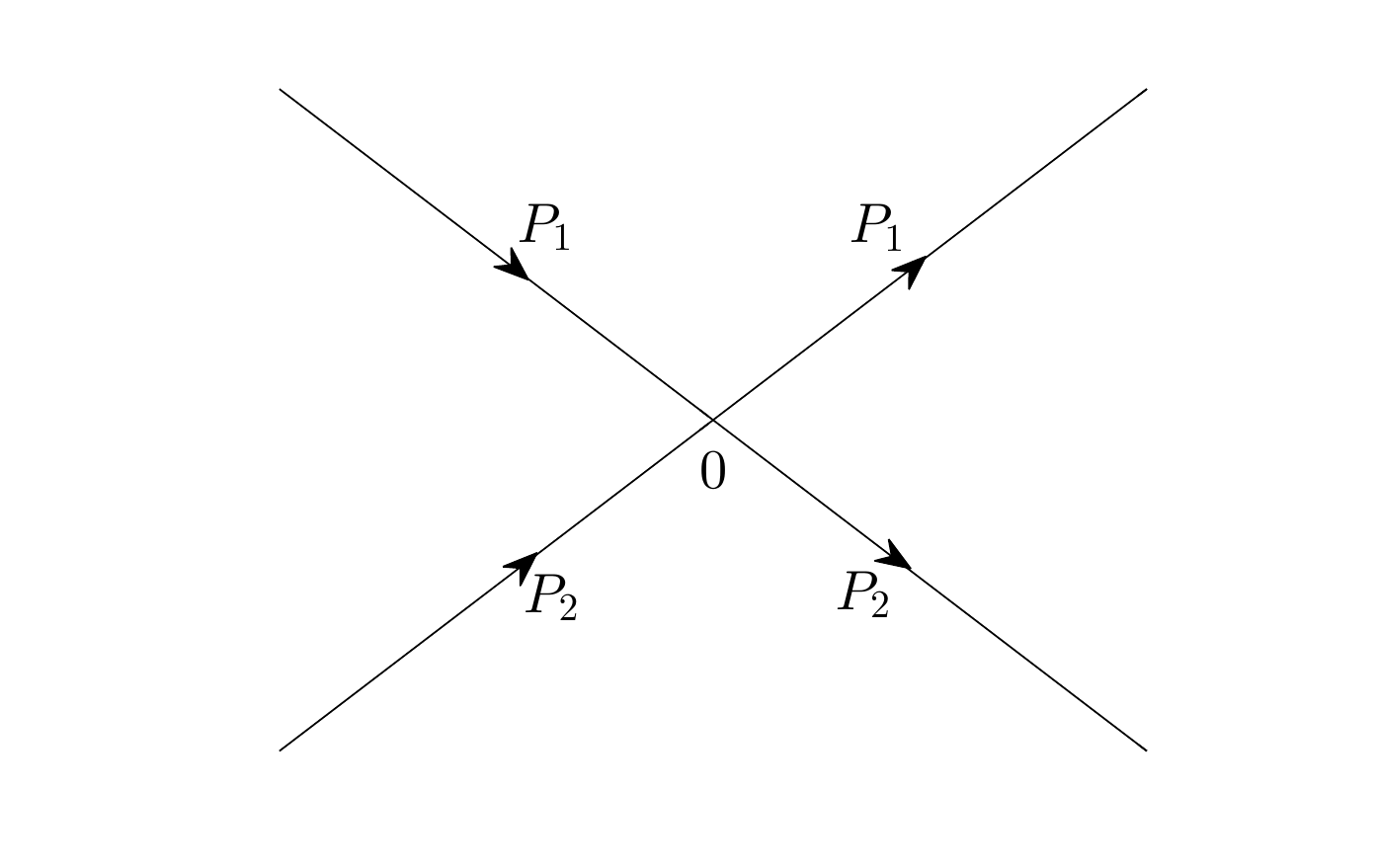}
\caption{The oriented contour $P$.}\label{fig10}
\end{figure}
Let $P$ denote the contour $P=P_1\cup P_2$ oriented as in Fig. \ref{fig10}, where
\berr
P_1=\{l\e^{\frac{\pi\ii}{6}}|l\geq0\}\cup\{l\e^{\frac{5\pi\ii}{6}}|l\geq0\},\quad
P_2=\{l\e^{-\frac{\pi\ii}{6}}|l\geq0\}\cup\{l\e^{-\frac{5\pi\ii}{6}}|l\geq0\}.
\eerr
\textbf{Fourth order Painlev\'e II RH problem A.1.}
Let $s\in\bfR$ be a real number. Find an analytic function $M^P(y;z)$ in $\bfC\setminus P$ parametrized by $y\in\bfR,s\in\bfR$ such that:

1. For $z\in P$, the continuous boundary values $M^P_\pm(y;z)$ satisfy
\be\label{A.1}
M^P_+(y;z)=M^P_-(y;z)v^P(s,y;z),~~z\in P,
\ee
where the jump matrix $v^P(s,y;z)$ is defined by
\be\label{A.2}
v^P(s,y;z)=\left\{
\begin{aligned}
&\begin{pmatrix}
1 &~ 0\\
s\e^{2\ii(\frac{4}{5}z^5-yz)} &~ 1
\end{pmatrix},\qquad~ z\in P_1,\\
&\begin{pmatrix}
1 &~ -s\e^{-2\ii(\frac{4}{5}z^5-yz)}\\
0 &~ 1
\end{pmatrix},\quad z\in P_2.
\end{aligned}
\right.
\ee

2. $M^P(y;z)=I+O(\frac{1}{z})$, as~$z\rightarrow\infty$.\\
\begin{lemma} The RH problem A.1 has a unique solution $M^P(y;z)$ for each $y\in\bfR$. Moreover, there exists smooth functions $\{M_j^P(y)\}^4$ of $y\in\bfR$ with decay as $y\rightarrow-\infty$ such that
\be\label{A.3}
M^P(y;z)=I+\sum_{j=1}^4\frac{M_j^P(y)}{z^j}+O(z^{-5}), \quad z\rightarrow\infty,
\ee
uniformly for $y$ in compact subsets of $\bfR$ and for $\arg z\in[0,2\pi]$. The leading coefficient $M_1^Y$ is given by
\be\label{A.4}
M_1^P(y)=\ii\begin{pmatrix}
-2\int^y_{-\infty}u_p^2(y')\dd y' ~& u_p(y)\\
-u_p(y) ~& 2\int^y_{-\infty}u_p^2(y')\dd y'
\end{pmatrix},
\ee
where the real-valued function $u_p(y)$ satisfies the following fourth order Painlev\'e II equation (see \cite{KS})
\be\label{A.5}
u_p^{''''}(y)-40u_p^2(y)u_p^{''}(y)-40u_p(y)u_p'^2(y)+96u_p^5(y)-4yu_p(y)=0.
\ee
\end{lemma}

\section{ Model RH problem in Region III}
\begin{figure}[htbp]
\centering
\includegraphics[width=3in]{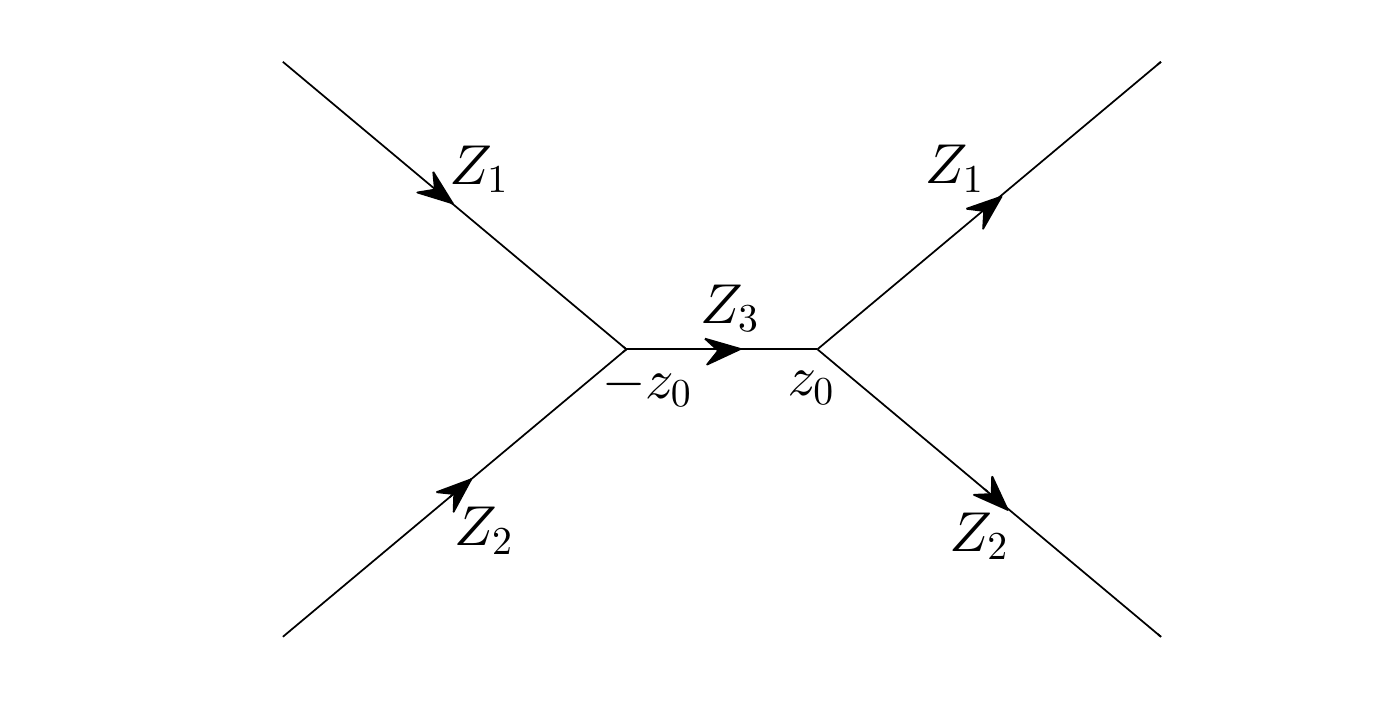}
\caption{The oriented contour $Z$.}\label{fig11}
\end{figure}
Given a number $z_0\geq0$, let $Z$ denote the contour $Z=Z_1\cup Z_1\cup Z_3$, where the line segments
\be\label{B.1}
\begin{aligned}
Z_1&=\{z_0+l\e^{\frac{\pi\ii}{6}}|l\geq0\}\cup\{-z_0+l\e^{\frac{5\pi\ii}{6}}|l\geq0\},\\
Z_2&=\{z_0+l\e^{-\frac{\pi\ii}{6}}|l\geq0\}\cup\{-z_0+l\e^{-\frac{5\pi\ii}{6}}|l\geq0\},\\
Z_3&=\{l|-z_0\leq l\leq z_0\}
\end{aligned}
\ee
are oriented as in Fig. \ref{fig11}. \\
\textbf{Model RH problem B.1.} Find an analytic function $M^Z(y;z,z_0)$ in $\bfC\setminus Z$ parametrized by $y>0,s\in\bfR,z_0\geq0$ such that:

1. The continuous boundary values $M^Z_\pm(y;z,z_0)$ satisfy the jump condition
\be\label{B.2}
M^Z_+(y;z,z_0)=M^Z_-(y;z,z_0)v^Z(s,y;z,z_0),~~~z\in Z,
\ee
where the jump matrix $v^Z(s,y,z;z_0)$ is defined by
\be\label{B.3}
v^Z(s,y;z,z_0)=\left\{
\begin{aligned}
&\begin{pmatrix}
1 &~ 0\\
s\e^{2\ii(\frac{4}{5}z^5-yz)} &~ 1
\end{pmatrix},\qquad\qquad\qquad\qquad~~~~~~~ z\in Z_1,\\
&\begin{pmatrix}
1 &~ -s\e^{-2\ii(\frac{4}{5}z^5-yz)}\\
0 &~ 1
\end{pmatrix},~~\quad\qquad\qquad\qquad~~~ z\in Z_2,\\
&\begin{pmatrix}
1 &~ -s\e^{-2\ii(\frac{4}{5}z^5-yz)}\\
0 &~ 1
\end{pmatrix}\begin{pmatrix}
1 &~ 0\\
s\e^{2\ii(\frac{4}{5}z^5-yz)} &~ 1
\end{pmatrix},~~ z\in Z_3.
\end{aligned}
\right.
\ee

2. As $z\rightarrow\infty$, $M^Z(y;z,z_0)=I+O(\frac{1}{z})$.
\begin{lemma} Define the parameter subset
\be\label{B.4}
\Bbb P=\{(y,t,z_0)\in\bfR^3|0<y<C_1,t\geq1,\sqrt[4]{y}/\sqrt{2}\leq z_0\leq C_2\},
\ee
where $C_1,C_2>0$ are constants. Then for $(y,t,z_0)\in\Bbb P$, the RH problem B.1 has a unique solution $M^Z(s,y;z,z_0)$ which satisfies
\be\label{B.5}
M^Z(y;z,z_0)=I+\frac{\ii}{z}\begin{pmatrix}
-2\int^y_{\infty}u_p^2(y')\dd y' ~& -u_p(y)\\
u_p(y) ~& 2\int^y_{\infty}u_p^2(y')\dd y'
\end{pmatrix}+O\bigg(\frac{1}{z^2}\bigg),\quad z\rightarrow\infty,
\ee
where $u_p(y)$ denotes the solution of the fourth order Painlev\'e II equation \eqref{A.5} and $M^Z(y,z,z_0)$ is uniformly bounded for $z\in\bfC\setminus Z$. Furthermore, $M^Z$ obeys the symmetries
\be\label{B.6}
M^Z(y;-z,z_0)=\overline{M^Z(y;\bar{z},z_0)}=\sigma_1M^Z(y;z,z_0)\sigma_1.
\ee
\end{lemma}

\medskip
\small{

}
\end{document}